\documentclass[11pt]{preprint}
\usepackage[full]{textcomp}
\usepackage[osf]{newtxtext} 
\usepackage[cal=boondoxo]{mathalfa}
\usepackage{colortbl}

\usepackage{comment}

\usepackage{amssymb}
\usepackage{lmodern}
\usepackage{mathtools}

\usepackage{hyperref}
\usepackage{breakurl}
\usepackage{aligned-overset}
\usepackage{mhenvs}
\usepackage{mhequ} 
\newcommand{\be}{\begin{equation*}}
\newcommand{\ee}{\end{equation*}}
\usepackage{mhsymb}
\usepackage{booktabs}
\usepackage{tikz}
\usepackage{tcolorbox}
\usepackage{mathrsfs}
\usepackage[utf8]{inputenc}
\usepackage{longtable}
\usepackage{wrapfig}
\usepackage{subcaption}
\usepackage{mathrsfs}
\usepackage{epsfig}
\usepackage{microtype}
\usepackage{comment}
\usepackage{wasysym}
\usepackage{centernot}
\usepackage{enumitem}
\usepackage{bm}
\usepackage{stackrel}
\usepackage{graphicx}

\makeatletter
\newcommand{\globalcolor}[1]{%
  \color{#1}\global\let\default@color\current@color
}
\makeatother

\usetikzlibrary{calc}
\usetikzlibrary{decorations}
\usetikzlibrary{positioning}
\usetikzlibrary{shapes}
\usetikzlibrary{shapes.misc}

\tikzset{cross/.style={cross out, draw=black, fill=none, minimum size=2*(#1-\pgflinewidth), inner sep=0pt, outer sep=0pt}, cross/.default={2pt}}

\definecolor{blush}{rgb}{0.87, 0.36, 0.51}
	\definecolor{brightcerulean}{rgb}{0.11, 0.67, 0.84}
	\definecolor{greenryb}{rgb}{0.4, 0.69, 0.2}

\newif\ifdark
\darkfalse

\ifdark
\definecolor{darkred}{rgb}{0.9,0.2,0.2}
\definecolor{darkblue}{rgb}{0.7,0.3,1}
\definecolor{darkgreen}{rgb}{0.1,0.9,0.1}
\definecolor{franck}{rgb}{0,0.8,1}
\definecolor{pagebackground}{rgb}{.15,.21,.18}
\definecolor{pageforeground}{rgb}{.84,.84,.85}
\pagecolor{pagebackground}
\AtBeginDocument{\globalcolor{pageforeground}}
\definecolor{symbols}{rgb}{0,0.7,1}
\colorlet{connection}{red!80!black}
\colorlet{boxcolor}{blue!50}

\else

\definecolor{darkred}{rgb}{0.7,0.1,0.1}
\definecolor{darkblue}{rgb}{0.4,0.1,0.8}
\definecolor{darkgreen}{rgb}{0.1,0.7,0.1}
\definecolor{franck}{rgb}{0,0,1}
\definecolor{pagebackground}{rgb}{1,1,1}
\definecolor{pageforeground}{rgb}{0,0,0}
\colorlet{symbols}{blue!90!black}
\colorlet{connection}{red!30!black}
\colorlet{boxcolor}{blue!50!black}

\fi

\def\slash{\leavevmode\unskip\kern0.18em/\penalty\exhyphenpenalty\kern0.18em}
\def\dash{\leavevmode\unskip\kern0.18em--\penalty\exhyphenpenalty\kern0.18em}

\DeclareMathAlphabet{\mathbbm}{U}{bbm}{m}{n}

\DeclareFontFamily{U}{BOONDOX-calo}{\skewchar\font=45 }
\DeclareFontShape{U}{BOONDOX-calo}{m}{n}{
  <-> s*[1.05] BOONDOX-r-calo}{}
\DeclareFontShape{U}{BOONDOX-calo}{b}{n}{
  <-> s*[1.05] BOONDOX-b-calo}{}
\DeclareMathAlphabet{\mcb}{U}{BOONDOX-calo}{m}{n}
\SetMathAlphabet{\mcb}{bold}{U}{BOONDOX-calo}{b}{n}
\setlist{noitemsep,topsep=4pt,leftmargin=1.5em}

\DeclareMathAlphabet{\mathbbm}{U}{bbm}{m}{n}

\DeclareMathAlphabet{\mcb}{U}{BOONDOX-calo}{m}{n}
\SetMathAlphabet{\mcb}{bold}{U}{BOONDOX-calo}{b}{n}
\DeclareFontFamily{U}{mathx}{\hyphenchar\font45}
\DeclareFontShape{U}{mathx}{m}{n}{
      <5> <6> <7> <8> <9> <10>
      <10.95> <12> <14.4> <17.28> <20.74> <24.88>
      mathx10
      }{}
\DeclareSymbolFont{mathx}{U}{mathx}{m}{n}
\DeclareMathSymbol{\bigtimes}{1}{mathx}{"91}

\providecommand{\figures}{false}
{ \ifthenelse{\equal{\figures}{false}} {#1}{\[ {\rm Figure \ missing !} \]} }{}
\def\id{\mathrm{id}}

\def\graft{\curvearrowright}
\def\ggraft{{\color{gray}\curvearrowright}}
\def\rgraft{{\color{red}\curvearrowright}}

\usepackage{mathtools}

\def\mcT{\mathcal{T}}

\tikzstyle{tinydots}=[dash pattern=on \pgflinewidth off \pgflinewidth]
\tikzstyle{superdense}=[dash pattern=on 4pt off 1pt]

\newcommand{\mcC}{\mathcal{C}}

\newcommand{\mcI}{\mathcal{I}}

\newcommand{\mcK}{\mathcal{K}}
\newcommand{\mcD}{\mathcal{D}}

\newcommand{\mcP}{\mathcal{P}}

\newcommand{\mcQ}{\mathcal{Q}}

\newcommand{\beq}{\begin{equation}}
\newcommand{\eeq}{\end{equation}}

\usepackage{empheq}

\newcommand{\dmu}{\uparrow_\mu}
\newcommand{\downg}{\mathcolor{gray}{\downarrow}}

\newcommand{\T}{\mathbf{T}}

\newcommand{\mfe}{\mathfrak{e}}
\newcommand{\mfv}{\mathfrak{v}}

\newcommand{\mfg}{\mathfrak{g}}
\newcommand{\mfs}{\mathfrak{s}}

\def\miD{\mathit{D}}

\def\Labe{\mathfrak{e}}

\def\Labn{\mathfrak{n}}

\newcommand{\eval}{\Pi_{\eps,\mu}}
\newcommand{\heval}{\Pi_{\eps,\mu}^{R\times}}
\newcommand{\teval}{\Tilde{\Pi}_{\eps,\mu}^R}
\newcommand{\loceval}{\hat{\Pi}_{\eps,\mu}^R}
\newcommand{\locevalx}{\hat{\Pi}_{\eps,\mu,x}^R}
\newcommand{\hloceval}{\hat{\Pi}_{\eps,\mu}^{R\times}}
\newcommand{\hlocevalx}{\hat{\Pi}_{\eps,\mu,x}^{R\times}}
\newcommand{\D}{\partial}
\newcommand{\pem}{\phi_{\eps,\mu}}
\newcommand{\Fem}{F_{\eps,\mu}}
\newcommand{\Rem}{R_{\eps,\mu}}
\newcommand{\Mloc}{M_{\text{\tiny{loc}}}}

\newcommand\Item[1][]{%
  \ifx\relax#1\relax  \item \else \item[#1] \fi
  \abovedisplayskip=0pt\abovedisplayshortskip=0pt~\vspace*{-\baselineskip}}
  \newcommand{\TUpsilon}{\Tilde{\Upsilon}}

\def\${|\!|\!|}

\def\scal#1{{\langle#1\rangle}}

\def\Xig{{\color{gray}\Xi}}
\newcommand{\CIg}{\mathcolor{gray}{\mcI}}

\newcounter{theorem}

\newtheorem{ex}[theorem]{Example}

\newenvironment{DIFnomarkup}{}{} % see man latexdiff

\theorembodyfont{\rmfamily}

\newfont{\indic}{bbmss12}

\def\Nabla_#1{\nabla_{\!#1}}

\makeatletter
\pgfdeclareshape{crosscircle}
{
  \inheritsavedanchors[from=circle] % this is nearly a circle
  \inheritanchorborder[from=circle]
  \inheritanchor[from=circle]{north}
  \inheritanchor[from=circle]{north west}
  \inheritanchor[from=circle]{north east}
  \inheritanchor[from=circle]{center}
  \inheritanchor[from=circle]{west}
  \inheritanchor[from=circle]{east}
  \inheritanchor[from=circle]{mid}
  \inheritanchor[from=circle]{mid west}
  \inheritanchor[from=circle]{mid east}
  \inheritanchor[from=circle]{base}
  \inheritanchor[from=circle]{base west}
  \inheritanchor[from=circle]{base east}
  \inheritanchor[from=circle]{south}
  \inheritanchor[from=circle]{south west}
  \inheritanchor[from=circle]{south east}
  \inheritbackgroundpath[from=circle]
  \foregroundpath{
    \centerpoint%
    \pgf@xc=\pgf@x%
    \pgf@yc=\pgf@y%
    \pgfutil@tempdima=\radius%
    \pgfmathsetlength{\pgf@xb}{\pgfkeysvalueof{/pgf/outer xsep}}%  
    \pgfmathsetlength{\pgf@yb}{\pgfkeysvalueof{/pgf/outer ysep}}%  
    \ifdim\pgf@xb<\pgf@yb%
      \advance\pgfutil@tempdima by-\pgf@yb%
    \else%
      \advance\pgfutil@tempdima by-\pgf@xb%
    \fi%
    \pgfpathmoveto{\pgfpointadd{\pgfqpoint{\pgf@xc}{\pgf@yc}}{\pgfqpoint{-0.707107\pgfutil@tempdima}{-0.707107\pgfutil@tempdima}}}
    \pgfpathlineto{\pgfpointadd{\pgfqpoint{\pgf@xc}{\pgf@yc}}{\pgfqpoint{0.707107\pgfutil@tempdima}{0.707107\pgfutil@tempdima}}}
    \pgfpathmoveto{\pgfpointadd{\pgfqpoint{\pgf@xc}{\pgf@yc}}{\pgfqpoint{-0.707107\pgfutil@tempdima}{0.707107\pgfutil@tempdima}}}
    \pgfpathlineto{\pgfpointadd{\pgfqpoint{\pgf@xc}{\pgf@yc}}{\pgfqpoint{0.707107\pgfutil@tempdima}{-0.707107\pgfutil@tempdima}}}
  }
}
\makeatother

\def\symbol#1{\textcolor{symbols}{#1}}

\def\decorate#1#2{
        \ifnum#2>0
    		\foreach \count in {1,...,#2}{
	       	let
				\p1 = (sourcenode.center),
                \p2 = (sourcenode.east),
				\n1 = {\x2-\x1},
				\n2 = {1mm},
				\n3 = {(1.3+0.6*(\count-1))*\n1},
				\n4 = {0.7*\n1}
			in 
        		node[rectangle,fill=symbols,rotate=30,inner sep=0pt,minimum width=0.2*\n2,minimum height=\n2] at ($(sourcenode.center) + (\n3,\n4)$) {}
				}
		\fi
        \ifnum#1>0
    		\foreach \count in {1,...,#1}{
	       	let
				\p1 = (sourcenode.center),
                \p2 = (sourcenode.east),
				\n1 = {\x2-\x1},
				\n2 = {1mm},
				\n3 = {(1.3+0.6*(\count-1))*\n1},
				\n4 = {0.7*\n1}
			in 
        		node[rectangle,fill=symbols,rotate=-30,inner sep=0pt,minimum width=0.2*\n2,minimum height=\n2] at ($(sourcenode.center) + (-\n3,\n4)$) {}
				}
		\fi
}

\tikzset{
    dectriangle/.style 2 args={
        triangle,
        alias=sourcenode,
        append after command={\decorate{#1}{#2}}
    },
    dectriangle/.default={0}{0},
}

\tikzset{
	cross/.style={path picture={ 
  		\draw[symbols]
			(path picture bounding box.south east) -- (path picture bounding box.north west) (path picture bounding box.south west) -- (path picture bounding box.north east);
		}},
root/.style={circle,fill=green!50!black,inner sep=0pt, minimum size=1.2mm},
        dot/.style={circle,fill=pageforeground,inner sep=0pt, minimum size=1mm},
        dotred/.style={circle,fill=pageforeground!50!pagebackground,inner sep=0pt, minimum size=2mm},
        var/.style={circle,fill=pageforeground!10!pagebackground,draw=pageforeground,inner sep=0pt, minimum size=3mm},
        kernel/.style={semithick,shorten >=2pt,shorten <=2pt},
        kernels/.style={snake=zigzag,shorten >=2pt,shorten <=2pt,segment amplitude=1pt,segment length=4pt,line before snake=2pt,line after snake=5pt,},
        rho/.style={densely dashed,semithick,shorten >=2pt,shorten <=2pt},
           testfcn/.style={dotted,semithick,shorten >=2pt,shorten <=2pt},
        renorm/.style={shape=circle,fill=pagebackground,inner sep=1pt},
        labl/.style={shape=rectangle,fill=pagebackground,inner sep=1pt},
        xic/.style={very thin,circle,draw=symbols,fill=symbols,inner sep=0pt,minimum size=1.2mm},
        g/.style={very thin,rectangle,draw=symbols,fill=symbols!10!pagebackground,inner sep=0pt,minimum width=2.5mm,minimum height=1.2mm},
        xi/.style={very thin,circle,draw=symbols,fill=symbols!10!pagebackground,inner sep=0pt,minimum size=1.2mm},
	xies/.style={very thin,rectangle,fill=green!50!black!25,draw=symbols,inner sep=0pt,minimum size=1.1mm},
	xiesf/.style={very thin,rectangle,fill=green!50!black,draw=symbols,inner sep=0pt,minimum size=1.1mm},
        xix/.style={very thin,crosscircle,fill=symbols!10!pagebackground,draw=symbols,inner sep=0pt,minimum size=1.2mm},
        X/.style={very thin,cross,rectangle,fill=pagebackground,draw=symbols,inner sep=0pt,minimum size=1.2mm},
	xib/.style={thin,circle,fill=symbols!10!pagebackground,draw=symbols,inner sep=0pt,minimum size=1.6mm},
	xie/.style={thin,circle,fill=green!50!black,draw=symbols,inner sep=0pt,minimum size=1.6mm},
	xid/.style={thin,circle,fill=symbols,draw=symbols,inner sep=0pt,minimum size=1.6mm},
	xibx/.style={thin,crosscircle,fill=symbols!10!pagebackground,draw=symbols,inner sep=0pt,minimum size=1.6mm},
	kernels2/.style={very thick,draw=connection,segment length=12pt},
	keps/.style={thin,draw=symbols,->},
	kepspr/.style={thick,draw=connection,->},
	krho/.style={thin,draw=symbols,superdense,->},
	krhopr/.style={thick,draw=connection,superdense},
	triangle/.style = { regular polygon, regular polygon sides=3},
	not/.style={thin,circle,draw=connection,fill=connection,inner sep=0pt,minimum size=0.5mm},
	diff/.style = {very thin,draw=symbols,triangle,fill=red!50!black,inner sep=0pt,minimum size=1.6mm},
	diff1/.style = {very thin,dectriangle={1}{0},fill=red!50!black,draw=symbols,inner sep=0pt,minimum size=1.6mm},
	diff2/.style = {very thin,dectriangle={1}{1},fill=red!50!black,draw=symbols,inner sep=0pt,minimum size=1.6mm},
		diffmini/.style = {very thin,rectangle,fill=black,draw=black,inner sep=0pt,minimum size=0.75mm},
	 kernelsmod/.style={very thick,draw=connection,segment length=12pt},
	 rec/.style = {very thin,rectangle,fill=black,draw=black,inner sep=0pt,minimum size=2mm},
	cerc/.style={very thin,circle,draw=black,fill=symbols,inner sep=0pt,minimum size=2mm},
	stars/.style={very thin,star,star points=6,star point ratio=0.5, draw=black,fill=red,inner sep=0pt,minimum size=0.7mm},
	>=stealth,
        }
        \tikzset{
root/.style={circle,fill=black!50,inner sep=0pt, minimum size=3mm},
        circ/.style={circle,fill=white,draw=black,very thin,inner sep=.5pt, minimum size=1.2mm},
        round1/.style={fill=white,outer sep = 0,inner sep=2pt,rounded corners=1mm,draw,text=black,thin,minimum size=1.2mm},
          circ1/.style={circle,fill=red!10,draw=red,very thin,inner sep=.5pt, minimum size=1.2mm},
        rect/.style={fill=white,outer sep = 0,inner sep=2pt,rectangle,draw,text=black,thin,minimum size=1.2mm},
        rect1/.style={fill=white,outer sep = 0,inner sep=2pt,rectangle,draw,text=black,thin,minimum size=1.2mm},
        round2/.style={fill=red!10,outer sep = 0,inner sep=2pt,rounded corners=1mm,draw,text=black,thin,minimum size=1.2mm},
       round3/.style={fill=blue!10,outer sep = 0,inner sep=2pt,rounded corners=1mm,draw,text=black,thin,minimum size=1.2mm}, 
        rect2/.style={fill=black!10,outer sep = 0,inner sep=2pt,rectangle,draw,text=black,thin,minimum size=1.2mm},
        dot/.style={circle,fill=black,inner sep=0pt, minimum size=1.2mm},
        dotred/.style={circle,fill=black!50,inner sep=0pt, minimum size=2mm},
        var/.style={circle,fill=black!10,draw=black,inner sep=0pt, minimum size=3mm},
        kernel/.style={semithick,shorten >=2pt,shorten <=2pt},
         diag/.style={thin,shorten >=4pt,shorten <=4pt},
        kernel1/.style={thick},
        kernels/.style={snake=zigzag,shorten >=2pt,shorten <=2pt,segment amplitude=1pt,segment length=4pt,line before snake=2pt,line after snake=5pt,},
		kernels1/.style={snake=zigzag,segment amplitude=0.5pt,segment length=2pt},
		rho1/.style={densely dotted,semithick},
        rho/.style={densely dashed,semithick,shorten >=2pt,shorten <=2pt},
           testfcn/.style={dotted,semithick,shorten >=2pt,shorten <=2pt},
           visible/.style={draw, circle, fill, inner sep=0.25ex},
        renorm/.style={shape=circle,fill=white,inner sep=1pt},
        labl/.style={shape=rectangle,fill=white,inner sep=1pt},
        xic/.style={very thin,circle,fill=symbols,draw=black,inner sep=0pt,minimum size=1.2mm},
        xi/.style={very thin,circle,fill=blue!10,draw=black,inner sep=0pt,minimum size=1.2mm},
	xib/.style={very thin,circle,fill=blue!10,draw=black,inner sep=0pt,minimum size=1.6mm},
	xie/.style={very thin,circle,fill=green!50!black,draw=black,inner sep=0pt,minimum size=1mm},
	xid/.style={very thin,circle,fill=symbols,draw=black,inner sep=0pt,minimum size=1.6mm},
	edgetype/.style={very thin,circle,draw=black,inner sep=0pt,minimum size=5mm},
	nodetype/.style={very thick,circle,draw=black,inner sep=0pt,minimum size=5mm},
	kernels2/.style={very thick,draw=connection,segment length=12pt},
clean/.style={thin,circle,fill=black,inner sep=0pt,minimum size=1mm},	not/.style={thin,circle,fill=symbols,draw=connection,fill=connection,inner sep=0pt,minimum size=0.8mm},
	>=stealth,
        }

\makeatletter
\def\DeclareSymbol#1#2#3{%
	\expandafter\gdef\csname MH@symb@#1\endcsname{\tikzsetnextfilename{symbol#1}%
	\tikz[baseline=#2,scale=0.15,draw=symbols,line join=round]{#3}}%
	\expandafter\gdef\csname MH@symb@#1s\endcsname{\scalebox{0.75}{\tikzsetnextfilename{symbol#1}%
	\tikz[baseline=#2,scale=0.15,draw=symbols,line join=round]{#3}}}%
	\expandafter\gdef\csname MH@symb@#1ss\endcsname{\scalebox{0.65}{\tikzsetnextfilename{symbol#1}%
	\tikz[baseline=#2,scale=0.15,draw=symbols,line join=round]{#3}}}%
	}
\def\<#1>{\ifthenelse{\boolean{mmode}}{\mathchoice{\csname MH@symb@#1\endcsname}{\csname MH@symb@#1\endcsname}{\csname MH@symb@#1s\endcsname}{\csname MH@symb@#1ss\endcsname}}{\csname MH@symb@#1\endcsname}}
\makeatother

\DeclareSymbol{Xi22}{0.5}{\draw (0,0) node[xi] {} -- (-1,1) node[not] {} -- (0,2) node[xi] {};} % 1 not used in the text
\DeclareSymbol{Xi2}{-2}{\draw (-1,-0.25) node[xi] {} -- (0,1) node[xi] {};} % 2
\DeclareSymbol{Xi2b}{-2}{\draw (-1,-0.25) node[xic] {} -- (0,1) node[xic] {};} % 2
\DeclareSymbol{Xi2g}{-2}{\draw (-1,-0.25) node[xies] {} -- (0,1) node[xi] {};} % 2
\DeclareSymbol{Xi2g2}{-2}{\draw (-1,-0.25) node[xi] {} -- (0,1) node[xies] {};} % 2
\DeclareSymbol{cXi2}{-2}{\draw (0,-0.25) node[xi] {} -- (-1,1) node[xic] {};}%3
\DeclareSymbol{Xi3}{0}{\draw (0,0) node[xi] {} -- (-1,1) node[xi] {} -- (0,2) node[xi] {};}%4
\DeclareSymbol{XiIIXi}{0}{\draw (0,0) node[xi] {} -- (-1,1); \draw[kernels2] (-1,1) node[not] {} -- (0,2) node[xi] {};}%4

\DeclareSymbol{Xi4}{2}{\draw (0,0) node[xi] {} -- (-1,1) node[xi] {} -- (0,2) node[xi] {} -- (-1,3) node[xi] {};}%5
\DeclareSymbol{Xi4_1}{2}{\draw (0,0) node[xic] {} -- (-1,1) node[xic] {} -- (0,2) node[xi] {} -- (-1,3) node[xi] {};}%6
\DeclareSymbol{Xi4_2}{2}{\draw (0,0) node[xic] {} -- (-1,1) node[xi] {} -- (0,2) node[xi] {} -- (-1,3) node[xic] {};}%7
\DeclareSymbol{Xi2X}{-2}{\draw (0,-0.25) node[xi] {} -- (-1,1) node[xix] {};}%8
\DeclareSymbol{XXi2}{-2}{\draw (0,-0.25) node[xix] {} -- (-1,1) node[xi] {};}%9
\DeclareSymbol{IIXi}{0}{\draw (0,-0.25) node[not] {} -- (-1,1) node[xi] {} -- (0,2) node[xi] {};}%10
\DeclareSymbol{IXi^2}{-1}{\draw (-1,1) node[xi] {} -- (0,0) node[not] {} -- (1,1) node[xi] {};}%11
\DeclareSymbol{IIXi^2}{-4}{\draw (0,-1.5) node[not] {} -- (0,0);
\draw[kernels2] (-1,1) node[xi] {} -- (0,0) node[not] {} -- (1,1) node[xi] {};}%12
\DeclareSymbol{XiX}{-2.8}{\node[xibx] {};}%13
\DeclareSymbol{tauX}{-2.8}{ \node[X] {};}%14
\DeclareSymbol{Xi}{-2.8}{\node[xib] {};}%15

\DeclareSymbol{IXiX}{-1}{\draw (0,-0.25) node[not] {} -- (-1,1) node[xix] {};}%18
\DeclareSymbol{IXi3}{2}{\draw (0,-0.25) node[not] {} -- (-1,1) node[xi] {} -- (0,2) node[xi] {} -- (-1,3) node[xi] {};}%20
\DeclareSymbol{IXi}{-2}{\draw (0,-0.25) node[not] {} -- (-1,1) node[xi] {};}%21
\DeclareSymbol{XiI}{-2}{\draw (0,-0.25) node[xi] {} -- (-1,1) node[not] {};}%22

\DeclareSymbol{Xi4b}{0}{\draw(0,1.5) node[xi] {} -- (0,0); \draw (-1,1) node[xi] {} -- (0,0) node[xi] {} -- (1,1) node[xi] {};}%23
\DeclareSymbol{Xi4b'}{0}{\draw(0,1.5) node[xi] {} -- (0,-0.2); \draw (-1,1) node[xi] {} -- (0,-0.2) node[not] {} -- (1,1) node[xi] {};}%24 not used
\DeclareSymbol{Xi4c}{0}{\draw (0,1) -- (0.8,2.2) node[xi] {};\draw (0,-0.25) node[xi] {} -- (0,1) node[xi] {} -- (-0.8,2.2) node[xi] {};}%25
\DeclareSymbol{Xi4d}{-4.5}{\draw (0,-1.5) node[not] {} -- (0,0); \draw (-1,1) node[xi] {} -- (0,0) node[xi] {} -- (1,1) node[xi] {};}%26 not used
\DeclareSymbol{Xi4e}{0}{\draw (0,2) node[xi] {} -- (-1,1) node[xi] {} -- (0,0) node[xi] {} -- (1,1) node[xi] {};}%27
\DeclareSymbol{Xi4e'}{0}{\draw (0,2) node[xi] {} -- (-1,1) node[xi] {} -- (0,-0.2) node[not] {} -- (1,1) node[xi] {};}%28 not used

\DeclareSymbol{Xitwo}%29
{0}{\draw[kernels2] (0,0) node[not] {} -- (-1,1) node[not] {}
-- (-2,2) node[not]{} -- (-3,3) node[xi]  {};
\draw[kernels2] (0,0) -- (1,1) node[xi] {};
\draw[kernels2] (-1,1) -- (0,2) node[xi] {};
\draw[kernels2] (-2,2) -- (-1,3) node[xi] {};}

\DeclareSymbol{IXitwo}%30
{0}{\draw (-.7,1.2) node[xi] {} -- (0,-0.2) -- (.7,1.2) node[xi] {};}
\DeclareSymbol{I1Xitwo}%30
{0}{\draw[kernels2] (0,0) node[not] {} -- (-1,1) node[xi] {};
\draw[kernels2] (0,0) -- (1,1) node[xi] {};}

\DeclareSymbol{I1Xitwobis}%30
{0}{\draw[kernels2] (0,0) node[not] {} -- (-1,1) node[xies] {};
\draw[kernels2] (0,0) -- (1,1) node[xies] {};}

\DeclareSymbol{I1Xitwog}%30
{0}{\draw[kernels2] (0,0) node[not] {} -- (-1,1) node[xies] {};
\draw[kernels2] (0,0) -- (1,1) node[xi] {};}

\DeclareSymbol{cI1Xitwo}%31
{0}{\draw[kernels2] (0,0) node[not] {} -- (-1,1) node[xic] {};
\draw[kernels2] (0,0) -- (1,1) node[xi] {};}

\DeclareSymbol{I1IXi3}{0}{\draw (0,0) node[xi] {} -- (-1,1) ; %32
\draw[kernels2] (-1,1) node[not] {} -- (0,2) node[xi] {};
\draw[kernels2] (-1,1) node[not] {} -- (-2,2) node[xi] {};}

\DeclareSymbol{I1Xi3c}{-1}{\draw[kernels2](0,1.5) node[xi] {} -- (0,0) node[not] {}; \draw (-1,1) node[xi] {} -- (0,0) ; \draw[kernels2] (0,0) -- (1,1) node[xi] {};}%33

\DeclareSymbol{I1Xi3cbis}{-1}{\draw[kernels2](0,1.5) node[xies] {} -- (0,0) node[not] {}; \draw (-1,1) node[xies] {} -- (0,0) ; \draw[kernels2] (0,0) -- (1,1) node[xies] {};}%33

\DeclareSymbol{I1IXi3b}{0}{\draw[kernels2] (0,0) node[not] {} -- (-1,1) ; \draw[kernels2] (0,0)   -- (1,1) node[xi] {} ;
\draw (-1,1) node[xi] {} -- (0,2) node[xi] {};
}%34

\DeclareSymbol{I1IXi3c}{0}{\draw[kernels2] (0,0) node[not] {} -- (-1,1) ; \draw[kernels2] (0,0)   -- (1,1) node[xi] {} ;
\draw[kernels2] (-1,1) node[not] {} -- (0,2) node[xi] {};
\draw[kernels2] (-1,1) node[not] {} -- (-2,2) node[xi] {};}%35

\DeclareSymbol{I1IXi3cbis}{0}{\draw[kernels2] (0,0) node[not] {} -- (-1,1) ; \draw[kernels2] (0,0)   -- (1,1) node[xies] {} ;
\draw[kernels2] (-1,1) node[not] {} -- (0,2) node[xies] {};
\draw[kernels2] (-1,1) node[not] {} -- (-2,2) node[xies] {};}%35

\DeclareSymbol{I1Xi}{0}{\draw[kernels2] (0,0) node[not] {} -- (-1,1)  node[xi] {} ;}%36

\DeclareSymbol{I1Xi4a}{2}{\draw[kernels2] (0,0) node[not] {} -- (-1,1) ; \draw[kernels2] (0,0) node[not] {} -- (1,1) node[xi] {} ;%37
\draw (-1,1) node[xi] {} -- (0,2) node[xi] {} -- (-1,3) node[xi] {};}
\DeclareSymbol{cI1Xi4a}{2}{\draw[kernels2] (0,0) node[not] {} -- (-1,1) ; \draw[kernels2] (0,0) node[not] {} -- (1,1) node[xic] {} ;%39
\draw (-1,1) node[xic] {} -- (0,2) node[xi] {} -- (-1,3) node[xi] {};}
\DeclareSymbol{I1Xi4b}{2}{\draw (0,0) node[xi] {} -- (-1,1) node[xi] {} -- (0,2) ; \draw[kernels2] (0,2) node[not] {} -- (-1,3) node[xi] {};\draw[kernels2] (0,2)  -- (1,3) node[xi] {};
}%41

\DeclareSymbol{cI1Xi4b}{2}{\draw (0,0) node[xic] {} -- (-1,1) node[xic] {} -- (0,2) ; \draw[kernels2] (0,2) node[not] {} -- (-1,3) node[xi] {};\draw[kernels2] (0,2)  -- (1,3) node[xi] {};
}%42

\DeclareSymbol{I1Xi4c}{2}{\draw (0,0) node[xi] {} -- (-1,1) node[not] {}; \draw[kernels2] (-1,1) -- (0,2) ; 
\draw[kernels2] (-1,1) -- (-2,2) node[xi] {} ;
\draw (0,2) node[xi] {} -- (-1,3) node[xi] {};}%43

\DeclareSymbol{cI1Xi4c}{2}{\draw (0,0) node[xic] {} -- (-1,1) node[not] {}; \draw[kernels2] (-1,1) -- (0,2) ; 
\draw[kernels2] (-1,1) -- (-2,2) node[xic] {} ;
\draw (0,2) node[xi] {} -- (-1,3) node[xi] {};}%44

\DeclareSymbol{I1Xi4ab}{2}{\draw[kernels2] (0,0) node[not] {} -- (-1,1) ; \draw[kernels2] (0,0) node[not] {} -- (1,1) node[xi] {};\draw (-1,1) node[xi] {} -- (0,2) ; \draw[kernels2] (0,2) node[not] {} -- (-1,3) node[xi] {};\draw[kernels2] (0,2)  -- (1,3) node[xi] {}; }%45

\DeclareSymbol{cI1Xi4ab}{2}{\draw[kernels2] (0,0) node[not] {} -- (-1,1) ; \draw[kernels2] (0,0) node[not] {} -- (1,1) node[xic] {};\draw (-1,1) node[xic] {} -- (0,2) ; \draw[kernels2] (0,2) node[not] {} -- (-1,3) node[xi] {};\draw[kernels2] (0,2)  -- (1,3) node[xi] {}; }%46

\DeclareSymbol{I1Xi4bc}{2}{\draw (0,0) node[xi] {} -- (-1,1) node[not] {}; \draw[kernels2] (-1,1) -- (0,2) ; %47
\draw[kernels2] (-1,1) -- (-2,2) node[xi] {} ; \draw[kernels2] (0,2) node[not] {} -- (-1,3) node[xi] {};\draw[kernels2] (0,2)  -- (1,3) node[xi] {};
}

\DeclareSymbol{cI1Xi4bc}{2}{\draw (0,0) node[xic] {} -- (-1,1) node[not] {}; \draw[kernels2] (-1,1) -- (0,2) ; %48
\draw[kernels2] (-1,1) -- (-2,2) node[xic] {} ; \draw[kernels2] (0,2) node[not] {} -- (-1,3) node[xi] {};\draw[kernels2] (0,2)  -- (1,3) node[xi] {};
}

\DeclareSymbol{I1Xi4abcc1}{2}{\draw[kernels2] (0,0) node[not] {} -- (-1,1) node[not] {}%50
-- (-2,2) node[not]{} -- (-3,3) node[xic]  {};
\draw[kernels2] (0,0) -- (1,1) node[xic] {};
\draw[kernels2] (-1,1) -- (0,2) node[xi] {};
\draw[kernels2] (-2,2) -- (-1,3) node[xi] {};
}

\DeclareSymbol{I1Xi4abcc1b}{2}{\draw[kernels2] (0,0) node[not] {} -- (-1,1) node[not] {}%50
-- (-2,2) node[not]{} -- (-3,3) node[xi]  {};
\draw[kernels2] (0,0) -- (1,1) node[xic] {};
\draw[kernels2] (-1,1) -- (0,2) node[xic] {};
\draw[kernels2] (-2,2) -- (-1,3) node[xi] {};
}

\DeclareSymbol{I1Xi4abcc2}{2}{\draw[kernels2] (0,0) node[not] {} -- (-1,1) node[not] {}%51
-- (-2,2) node[not]{} -- (-3,3) node[xic]  {};
\draw[kernels2] (0,0) -- (1,1) node[xi] {};
\draw[kernels2] (-1,1) -- (0,2) node[xi] {};
\draw[kernels2] (-2,2) -- (-1,3) node[xic] {};
}

\DeclareSymbol{I1Xi4ac}{2}{\draw[kernels2] (0,0) node[not] {} -- (-1,1) ; \draw[kernels2] (0,0) node[not] {} -- (1,1) node[xi] {}; 
\draw[kernels2] (-1,1) node[not] {} -- (0,2) ; %52
\draw[kernels2] (-1,1) -- (-2,2) node[xi] {} ;
\draw (0,2) node[xi] {} -- (-1,3) node[xi] {};}

\DeclareSymbol{cI1Xi4ac}{2}{\draw[kernels2] (0,0) node[not] {} -- (-1,1) ; \draw[kernels2] (0,0) node[not] {} -- (1,1) node[xic] {}; 
\draw[kernels2] (-1,1) node[not] {} -- (0,2) ; %53
\draw[kernels2] (-1,1) -- (-2,2) node[xic] {} ;
\draw (0,2) node[xi] {} -- (-1,3) node[xi] {};}

\DeclareSymbol{I1Xi4acc1}{2}{\draw[kernels2] (0,0) node[not] {} -- (-1,1) ; \draw[kernels2] (0,0) node[not] {} -- (1,1) node[xic] {}; 
\draw[kernels2] (-1,1) node[not] {} -- (0,2) ; %54
\draw[kernels2] (-1,1) -- (-2,2) node[xi] {} ;
\draw (0,2) node[xic] {} -- (-1,3) node[xi] {};}

\DeclareSymbol{I1Xi4acc2}{2}{\draw[kernels2] (0,0) node[not] {} -- (-1,1) ; \draw[kernels2] (0,0) node[not] {} -- (1,1) node[xic] {}; 
\draw[kernels2] (-1,1) node[not] {} -- (0,2) ; %55
\draw[kernels2] (-1,1) -- (-2,2) node[xi] {} ;
\draw (0,2) node[xi] {} -- (-1,3) node[xic] {};}

\DeclareSymbol{2I1Xi4}{2}{\draw[kernels2] (0,0) node[not] {} -- (-1,1) node[not] {};
\draw[kernels2] (0,0) -- (1,1) node[not] {};%56
\draw[kernels2] (-1,1) -- (-1.5,2.5) node[xi] {};
\draw[kernels2] (-1,1) -- (-0.5,2.5) node[xi] {};
\draw[kernels2] (1,1) -- (0.5,2.5) node[xi] {};
\draw[kernels2] (1,1) -- (1.5,2.5) node[xi] {};
}

\DeclareSymbol{2I1Xi4dis}{2}{\draw[kernels2] (0,0) node[not] {} -- (-1,1) node[not] {};
\draw[kernels2] (0,0) -- (1,1) node[not] {};%56
\draw[kernels2] (-1,1) -- (-1.5,2.5) node[xies] {};
\draw[kernels2] (-1,1) -- (-0.5,2.5) node[xies] {};
\draw[kernels2] (1,1) -- (0.5,2.5) node[xies] {};
\draw[kernels2] (1,1) -- (1.5,2.5) node[xies] {};
}

\DeclareSymbol{2I1Xi4c1}{2}{\draw[kernels2] (0,0) node[not] {} -- (-1,1) node[not] {};%57
\draw[kernels2] (0,0) -- (1,1) node[not] {};
\draw[kernels2] (-1,1) -- (-1.5,2.5) node[xic] {};
\draw[kernels2] (-1,1) -- (-0.5,2.5) node[xi] {};
\draw[kernels2] (1,1) -- (0.5,2.5) node[xic] {};
\draw[kernels2] (1,1) -- (1.5,2.5) node[xi] {};
}

\DeclareSymbol{2I1Xi4c2}{2}{\draw[kernels2] (0,0) node[not] {} -- (-1,1) node[not] {};%58
\draw[kernels2] (0,0) -- (1,1) node[not] {};
\draw[kernels2] (-1,1) -- (-1.5,2.5) node[xic] {};
\draw[kernels2] (-1,1) -- (-0.5,2.5) node[xic] {};
\draw[kernels2] (1,1) -- (0.5,2.5) node[xi] {};
\draw[kernels2] (1,1) -- (1.5,2.5) node[xi] {};
}

\DeclareSymbol{2I1Xi4b}{2}{\draw[kernels2] (0,0) node[not] {} -- (-1,1) ;%59
\draw[kernels2] (0,0) -- (1,1);
\draw (-1,1) node[xi] {} -- (-1,2.5) node[xi] {};
\draw (1,1)  node[xi] {} -- (1,2.5) node[xi] {};
}

\DeclareSymbol{2I1Xi4bb}{2}{\draw[kernels2] (0,0) node[not] {} -- (-1,1) ;%59
\draw[kernels2] (0,0) -- (1,1);
\draw (-1,1) node[xi] {} -- (-1,2.5) node[xiesf] {};
\draw (1,1)  node[xi] {} -- (1,2.5) node[xic] {};
}

\DeclareSymbol{2I1Xi4c}{2}{\draw[kernels2] (0,0) node[not] {} -- (-1,1);%60
\draw[kernels2] (0,0) -- (1,1) node[not] {};
\draw (-1,1)  node[xi] {} -- (-1,2.5) node[xi] {};
\draw[kernels2] (1,1) -- (0.4,2.5) node[xi] {};
\draw[kernels2] (1,1) -- (1.6,2.5) node[xi] {};
}

\DeclareSymbol{2I1Xi4cc1}{2}{\draw[kernels2] (0,0) node[not] {} -- (-1,1);%61
\draw[kernels2] (0,0) -- (1,1) node[not] {};
\draw (-1,1)  node[xic] {} -- (-1,2.5) node[xi] {};
\draw[kernels2] (1,1) -- (0.4,2.5) node[xic] {};
\draw[kernels2] (1,1) -- (1.6,2.5) node[xi] {};
}

\DeclareSymbol{2I1Xi4cc2}{2}{\draw[kernels2] (0,0) node[not] {} -- (-1,1);%62
\draw[kernels2] (0,0) -- (1,1) node[not] {};
\draw (-1,1)  node[xic] {} -- (-1,2.5) node[xic] {};
\draw[kernels2] (1,1) -- (0.4,2.5) node[xi] {};
\draw[kernels2] (1,1) -- (1.6,2.5) node[xi] {};
}

\DeclareSymbol{Xi4ba}{0}{\draw(-0.5,1.5) node[xi] {} -- (0,0); \draw (-1.5,1) node[xi] {} -- (0,0) node[not] {}; \draw[kernels2] (0,0) -- (1.5,1) node[xi] {};
\draw[kernels2] (0,0) -- (0.5,1.5) node[xi] {} ;}%63

\DeclareSymbol{Xi4badis}{0}{\draw(-0.5,1.5) node[xies] {} -- (0,0); \draw (-1.5,1) node[xies] {} -- (0,0) node[not] {}; \draw[kernels2] (0,0) -- (1.5,1) node[xies] {};
\draw[kernels2] (0,0) -- (0.5,1.5) node[xies] {} ;}%63

\DeclareSymbol{Xi4ba1}{0}{\draw(-0.5,1.5) node[xi] {} -- (0,0); \draw (-1.5,1) node[xi] {} -- (0,0) node[not] {}; \draw[kernels2] (0,0) -- (1.5,1) node[xic] {};
\draw[kernels2] (0,0) -- (0.5,1.5) node[xic] {} ;}%64

\DeclareSymbol{Xi4ba1b}{0}{\draw(-0.5,1.5) node[xic] {} -- (0,0); \draw (-1.5,1) node[xic] {} -- (0,0) node[not] {}; \draw[kernels2] (0,0) -- (1.5,1) node[xi] {};
\draw[kernels2] (0,0) -- (0.5,1.5) node[xi] {} ;}%64

\DeclareSymbol{Xi4ba1bdiff}{0}{\draw(-0.5,1.5) node[xic] {} -- (0,0); \draw (-1.5,1) node[xic] {} -- (0,0) node[not] {}; \draw (0,0) -- (1.5,1) node[xi] {};
\draw (0,0) -- (0.5,1.5) node[xi] {};
\draw(0,0) node[diff] {};}%64

\DeclareSymbol{Xi4ba1bb}{0}{\draw(-0.5,1.5) node[xic] {} -- (0,0); \draw (-1.5,1) node[xiesf] {} -- (0,0) node[not] {}; \draw[kernels2] (0,0) -- (1.5,1) node[xi] {};
\draw[kernels2] (0,0) -- (0.5,1.5) node[xi] {} ;}%64

\DeclareSymbol{Xi4ba2}{0}{\draw(-0.5,1.5) node[xi] {} -- (0,0); \draw (-1.5,1) node[xic] {} -- (0,0) node[not] {}; \draw[kernels2] (0,0) -- (1.5,1) node[xi] {};
\draw[kernels2] (0,0) -- (0.5,1.5) node[xic] {} ;}%65

\DeclareSymbol{Xi4ba2b}{0}{\draw(-0.5,1.5) node[xi] {} -- (0,0); \draw (-1.5,1) node[xic] {} -- (0,0) node[not] {}; \draw[kernels2] (0,0) -- (1.5,1) node[xi] {};
\draw[kernels2] (0,0) -- (0.5,1.5) node[xiesf] {} ;}%65

\DeclareSymbol{Xi4ca}{0}{\draw (0,1) -- (-1,2.2) node[xi] {};\draw (0,-0.25) node[xi] {} -- (0,1) ; \draw[kernels2] (0,1) node[not] {} -- (1,2.2) node[xi] {};%67
\draw[kernels2] (0,1) {} -- (0,2.7) node[xi] {};
}

\DeclareSymbol{Xi4cb}{0}{\draw (-1,1) -- (-2,2) node[xi] {};\draw[kernels2] (0,0)  -- (-1,1) node[xi] {} ; \draw[kernels2] (0,0) node[not] {} -- (1,1) node[xi] {} ; 
\draw (-1,1) node[xi] {} -- (0,2) node[xi] {};}

\DeclareSymbol{Xi4cbb}{0}{\draw (-1,1) -- (-2,2) node[xiesf] {};\draw[kernels2] (0,0)  -- (-1,1) node[xi] {} ; \draw[kernels2] (0,0) node[not] {} -- (1,1) node[xi] {} ; 
\draw (-1,1) node[xi] {} -- (0,2) node[xic] {};}

\DeclareSymbol{Xi4cbc1}{0}{\draw (-1,1) -- (-2,2) node[xic] {};\draw[kernels2] (0,0)  -- (-1,1) node[xic] {} ; \draw[kernels2] (0,0) node[not] {} -- (1,1) node[xi] {} ; 
\draw (-1,1) node[xic] {} -- (0,2) node[xi] {};}

\DeclareSymbol{Xi4cbc2}{0}{\draw (-1,1) -- (-2,2) node[xi] {};\draw[kernels2] (0,0)  -- (-1,1) node[xi] {} ; \draw[kernels2] (0,0) node[not] {} -- (1,1) node[xic] {} ; 
\draw (-1,1) node[xic] {} -- (0,2) node[xi] {};}

\DeclareSymbol{Xi4cab}{0}{\draw (-1,1) -- (-2,2) node[xi] {};\draw[kernels2] (0,0)  -- (-1,1); \draw[kernels2] (0,0) node[not] {} -- (1,1) node[xi] {} ; 
\draw[kernels2] (-1,1)  {} -- (0,2) node[xi] {};
\draw[kernels2] (-1,1) node[not] {} -- (-1,2.5) node[xi] {};
}%69

\DeclareSymbol{Xi4cabdis}{0}{\draw (-1,1) -- (-2,2) node[xies] {};\draw[kernels2] (0,0)  -- (-1,1); \draw[kernels2] (0,0) node[not] {} -- (1,1) node[xies] {} ; 
\draw[kernels2] (-1,1)  {} -- (0,2) node[xies] {};
\draw[kernels2] (-1,1) node[not] {} -- (-1,2.5) node[xies] {};
}%69

\DeclareSymbol{Xi4cabc1}{0}{\draw (-1,1) -- (-2,2) node[xi] {};\draw[kernels2] (0,0)  -- (-1,1); \draw[kernels2] (0,0) node[not] {} -- (1,1) node[xic] {} ; 
\draw[kernels2] (-1,1)  {} -- (0,2) node[xic] {};
\draw[kernels2] (-1,1) node[not] {} -- (-1,2.5) node[xi] {};
}%69

\DeclareSymbol{Xi4cabc2}{0}{\draw (-1,1) -- (-2,2) node[xic] {};\draw[kernels2] (0,0)  -- (-1,1); \draw[kernels2] (0,0) node[not] {} -- (1,1) node[xic] {} ; 
\draw[kernels2] (-1,1)  {} -- (0,2) node[xi] {};
\draw[kernels2] (-1,1) node[not] {} -- (-1,2.5) node[xi] {};
}%69

\DeclareSymbol{Xi4ea}{1.5}{\draw (-1,2.5) node[xi] {} -- (-1,1) node[xi] {} -- (0,0); %71
 \draw[kernels2] (0,0)  -- (1,1) node[xi] {};
\draw[kernels2] (0,0) node[not] {} -- (0,1.5) node[xi] {}; }

\DeclareSymbol{Xi4eac1}{1.5}{\draw (-1,2.5) node[xic] {} -- (-1,1) node[xi] {} -- (0,0); %72
 \draw[kernels2] (0,0)  -- (1,1) node[xic] {};
\draw[kernels2] (0,0) node[not] {} -- (0,1.5) node[xi] {}; }

\DeclareSymbol{Xi4eac1b}{1.5}{\draw (-1,2.5) node[xic] {} -- (-1,1) node[xi] {} -- (0,0); %72
 \draw[kernels2] (0,0)  -- (1,1) node[xiesf] {};
\draw[kernels2] (0,0) node[not] {} -- (0,1.5) node[xi] {}; }

\DeclareSymbol{Xi4eac2}{1.5}{\draw (-1,2.5) node[xic] {} -- (-1,1) node[xic] {} -- (0,0); %73
 \draw[kernels2] (0,0)  -- (1,1) node[xi] {};
\draw[kernels2] (0,0) node[not] {} -- (0,1.5) node[xi] {}; }

\DeclareSymbol{Xi4eact1}{1.5}{\draw (-1,2.5) node[xic] {} -- (-1,1) node[xi] {} -- (0,0); %74
 \draw (0,0)  -- (1,1) node[xic] {};
\draw[rho] (0,0) node[not] {} -- (0,1.5) node[xi] {}; }

\DeclareSymbol{Xi4eact2}{1.5}{\draw[rho] (-1,2.5) node[xic] {} -- (-1,1) node[xi] {} -- (0,0); %75
 \draw (0,0)  -- (1,1) node[xic] {};
\draw (0,0) node[not] {} -- (0,1.5) node[xi] {}; }

\DeclareSymbol{Xi4eabis}{1.5}{\draw (-1,2.5) node[xi] {} -- (-1,1) ; \draw[kernels2] (-1,1) node[xi] {} -- (0,0); 
 \draw (0,0)  -- (1,1) node[xi] {};%76
\draw[kernels2] (0,0) node[not] {} -- (0,1.5) node[xi] {}; }

\DeclareSymbol{Xi4eabisc1}{1.5}{\draw (-1,2.5) node[xic] {} -- (-1,1) ; \draw[kernels2] (-1,1) node[xi] {} -- (0,0); 
 \draw (0,0)  -- (1,1) node[xi] {};%77
\draw[kernels2] (0,0) node[not] {} -- (0,1.5) node[xic] {}; }

\DeclareSymbol{Xi4eabisc1b}{1.5}{\draw (-1,2.5) node[xic] {} -- (-1,1) ; \draw[kernels2] (-1,1) node[xi] {} -- (0,0); 
 \draw (0,0)  -- (1,1) node[xi] {};%77
\draw[kernels2] (0,0) node[not] {} -- (0,1.5) node[xiesf] {}; }

\DeclareSymbol{Xi4eabisc1bis}{1.5}{\draw (-1,2.5) node[xi] {} -- (-1,1) ; \draw[kernels2] (-1,1) node[xi] {} -- (0,0); 
 \draw (0,0)  -- (1,1) node[xi] {};%78
\draw[kernels2] (0,0) node[not] {} -- (0,1.5) node[xi] {};
\draw (-2,1) node[] {\tiny{$i$}};
\draw (-2,2.5) node[] {\tiny{$\ell$}};
\draw (2,1) node[] {\tiny{$k$}};
\draw (0,2.5) node[] {\tiny{$j$}};
 }

\DeclareSymbol{Xi4eabisc1tris}{1.5}{\draw (-1,2.5) node[xi] {} -- (-1,1) ; \draw[kernels2] (-1,1) node[xi] {} -- (0,0); 
 \draw (0,0)  -- (1,1) node[xi] {};%78
\draw[kernels2] (0,0) node[not] {} -- (0,1.5) node[xi] {};
\draw (-2,1) node[] {\tiny{i}};
\draw (-2,2.5) node[] {\tiny{j}};
\draw (2,1) node[] {\tiny{j}};
\draw (0,2.5) node[] {\tiny{i}};
 }

\DeclareSymbol{Xi4eabisc1quater}{1.5}{\draw (-1,2.5) node[xic] {} -- (-1,1) ; \draw[kernels2] (-1,1) node[xi] {} -- (0,0); 
 \draw (0,0)  -- (1,1) node[xic] {};%78
\draw[kernels2] (0,0) node[not] {} -- (0,1.5) node[xi] {};
 }

\DeclareSymbol{Xi4eabisc2}{1.5}{\draw (-1,2.5) node[xic] {} -- (-1,1) ; \draw[kernels2] (-1,1) node[xi] {} -- (0,0); %79
 \draw (0,0)  -- (1,1) node[xic] {};
\draw[kernels2] (0,0) node[not] {} -- (0,1.5) node[xi] {}; }

\DeclareSymbol{Xi4eabisc2l}{1.5}{\draw (-1,2.5) node[xiesf] {} -- (-1,1) ; \draw[kernels2] (-1,1) node[xi] {} -- (0,0); %79
 \draw (0,0)  -- (1,1) node[xic] {};
\draw[kernels2] (0,0) node[not] {} -- (0,1.5) node[xi] {}; }

\DeclareSymbol{Xi4eabisc2r}{1.5}{\draw (-1,2.5) node[xic] {} -- (-1,1) ; \draw[kernels2] (-1,1) node[xi] {} -- (0,0); %79
 \draw (0,0)  -- (1,1) node[xiesf] {};
\draw[kernels2] (0,0) node[not] {} -- (0,1.5) node[xi] {}; }

\DeclareSymbol{Xi4eabisc3}{1.5}{\draw (-1,2.5) node[xic] {} -- (-1,1) ; \draw[kernels2] (-1,1) node[xic] {} -- (0,0); %80
 \draw (0,0)  -- (1,1) node[xi] {};
\draw[kernels2] (0,0) node[not] {} -- (0,1.5) node[xi] {}; }

\DeclareSymbol{Xi4eb}{0}{%81
\draw[kernels2] (0,2) node[xi] {} -- (-1,1) ; \draw[kernels2] (-2,2)  node[xi] {} -- (-1,1) ; \draw (-1,1)  node[not] {} -- (0,0); 
 \draw (0,0) node[xi] {}  -- (1,1) node[xi] {};
}

\DeclareSymbol{Xi4eab}{1.5}{\draw[kernels2] (-1,2.5) node[xi] {} -- (-1,1) ; \draw[kernels2] (-2,2)  node[xi] {} -- (-1,1) ; \draw (-1,1)  node[not] {} -- (0,0); %82
 \draw[kernels2] (0,0)  -- (1,1) node[xi] {};
\draw[kernels2] (0,0) node[not] {} -- (0,1.5) node[xi] {}; 
}

\DeclareSymbol{Xi4eabdis}{1.5}{\draw[kernels2] (-1,2.5) node[xies] {} -- (-1,1) ; \draw[kernels2] (-2,2)  node[xies] {} -- (-1,1) ; \draw (-1,1)  node[not] {} -- (0,0); %82
 \draw[kernels2] (0,0)  -- (1,1) node[xies] {};
\draw[kernels2] (0,0) node[not] {} -- (0,1.5) node[xies] {}; 
}

\DeclareSymbol{Xi4eabc1}{1.5}{\draw[kernels2] (-1,2.5) node[xic] {} -- (-1,1) ; \draw[kernels2] (-2,2)  node[xi] {} -- (-1,1) ; \draw (-1,1)  node[not] {} -- (0,0); %83
 \draw[kernels2] (0,0)  -- (1,1) node[xic] {};
\draw[kernels2] (0,0) node[not] {} -- (0,1.5) node[xi] {}; 
}

\DeclareSymbol{Xi4eabc2}{1.5}{\draw[kernels2] (-1,2.5) node[xi] {} -- (-1,1) ; \draw[kernels2] (-2,2)  node[xi] {} -- (-1,1) ; \draw (-1,1)  node[not] {} -- (0,0); %84
 \draw[kernels2] (0,0)  -- (1,1) node[xic] {};
\draw[kernels2] (0,0) node[not] {} -- (0,1.5) node[xic] {}; 
}

\DeclareSymbol{Xi4eabbis}{1.5}{\draw[kernels2] (-1,2.5) node[xi] {} -- (-1,1) ; \draw[kernels2] (-2,2)  node[xi] {} -- (-1,1) ; \draw[kernels2] (-1,1)  node[not] {} -- (0,0); %85
 \draw (0,0)  -- (1,1) node[xi] {};
\draw[kernels2] (0,0) node[not] {} -- (0,1.5) node[xi] {}; 
}

\DeclareSymbol{Xi4eabbisc1}{1.5}{\draw[kernels2] (-1,2.5) node[xic] {} -- (-1,1) ; \draw[kernels2] (-2,2)  node[xi] {} -- (-1,1) ; \draw[kernels2] (-1,1)  node[not] {} -- (0,0); %86
 \draw (0,0)  -- (1,1) node[xic] {};
\draw[kernels2] (0,0) node[not] {} -- (0,1.5) node[xi] {}; 
}

\DeclareSymbol{Xi4eabbisc1perm}{1.5}{\draw[kernels2] (-1,2.5) node[xi] {} -- (-1,1) ; \draw[kernels2] (-2,2)  node[xic] {} -- (-1,1) ; \draw[kernels2] (-1,1)  node[not] {} -- (0,0); %86
 \draw (0,0)  -- (1,1) node[xic] {};
\draw[kernels2] (0,0) node[not] {} -- (0,1.5) node[xi] {}; 
}

\DeclareSymbol{Xi4eabbisc2}{1.5}{\draw[kernels2] (-1,2.5) node[xi] {} -- (-1,1) ; \draw[kernels2] (-2,2)  node[xi] {} -- (-1,1) ; \draw[kernels2] (-1,1)  node[not] {} -- (0,0); %87
 \draw (0,0)  -- (1,1) node[xic] {};
\draw[kernels2] (0,0) node[not] {} -- (0,1.5) node[xic] {}; 
}

\DeclareSymbol{Xi2cbis}{0}{\draw[kernels2] (0,1) -- (0.8,2.2) node[xi] {};\draw[kernels2] (0,-0.25) node[not] {} -- (0,1); \draw[kernels2] (0,1) node[not] {} -- (-0.8,2.2) node[xi] {};}%88

\DeclareSymbol{Xi2cbis1}{0}{\draw (0,1) -- (-0.8,2.2) node[xi] {};\draw[kernels2] (0,-0.25) node[not] {} -- (0,1) node[xi] {}; }%89

\DeclareSymbol{Xi2Xbis}{-2}{\draw[kernels2] (0,-0.25)  -- (-1,1) ; \draw (-1,1) node[xix] {};%91
\draw[kernels2] (0,-0.25) node[not] {} -- (1,1) node[xi] {};}

\DeclareSymbol{XXi2bis}{-2}{\draw[kernels2] (0,-0.25) -- (-1,1) node[xi] {};%92
\draw[kernels2] (0,-0.25) node[X] {} -- (1,1) node[xi] {};}

\DeclareSymbol{I1XiIXi}{0}{\draw[kernels2] (0,-0.25) -- (1,1) node[xi] {};%93
\draw (0,-0.25) node[not] {} -- (-1,1) node[xi] {};}

\DeclareSymbol{I1XiIXib}{0}{\draw  (0,-0.25) node[xi] {} -- (0,1) node[not] {};
\draw[kernels2] (0,1) -- (0,2.25) ; \draw (0,2.25) node[xi]{}; }%94

\DeclareSymbol{I1XiIXic}{0}{%95
\draw[kernels2] (0,0) -- (1,1) node[xi] {} ; 
\draw[kernels2] (0,0) node[not] {}  -- (-1,1) node[not] {} -- (0,2) node[xi] {};
}

\DeclareSymbol{thin}{1.4}{\draw[pagebackground] (-0.3,0) -- (0.3,0); \draw  (0,0) -- (0,2);}
\DeclareSymbol{thin2}{1.4}{\draw[pagebackground] (-0.3,0) -- (0.3,0); \draw[tinydots]  (0,0) -- (0,2);}

\DeclareSymbol{thick}{1.4}{\draw[pagebackground] (-0.3,0) -- (0.3,0); \draw[kernels2]  (0,0) -- (0,2);}

\DeclareSymbol{thick2}{1.4}{\draw[pagebackground] (-0.3,0) -- (0.3,0); \draw[kernels2,tinydots]  (0,0) -- (0,2);}

\DeclareSymbol{Xi4ind}{2}{\draw (0,0) node[xi,label={[label distance=-0.2em]right: \scriptsize  $ i $}]  { } -- (-1,1) node[xi,label={[label distance=-0.2em]left: \scriptsize  $ j $}] {} -- (0,2) node[xi,label={[label distance=-0.2em]right: \scriptsize  $ k $}] {} -- (-1,3) node[xi,label={[label distance=-0.2em]left: \scriptsize  $ \ell $}] {};}%115

\DeclareSymbol{Xi4c1}{2}{\draw (0,0) node[xic] {} -- (-1,1) node[xi] {} -- (0,2) node[xic] {} -- (-1,3) node[xi] {};} %119
\DeclareSymbol{IXi2ex}{0}{\draw (0,-0.25) node[xie] {} -- (-1,1) node[xi] {} ; \draw (0,-0.25)-- (1,1) node[xi] {};}
\DeclareSymbol{IXi2ex1}{0}{\draw (0,-0.25) node[xie] {} -- (-1,1) node[xi] {} -- (0,2) node[xi] {};}

\DeclareSymbol{Xi4b1}{0}{\draw(0,1.5) node[xic] {} -- (0,0); \draw (-1,1) node[xic] {} -- (0,0) node[xi] {} -- (1,1) node[xi] {};}

\DeclareSymbol{Xi4ec1}{0}{\draw (0,2) node[xi] {} -- (-1,1) node[xic] {} -- (0,0) node[xic] {} -- (1,1) node[xi] {};}
\DeclareSymbol{Xi4ec2}{0}{\draw (0,2) node[xic] {} -- (-1,1) node[xi] {} -- (0,0) node[xic] {} -- (1,1) node[xi] {};}
\DeclareSymbol{Xi4ec3}{0}{\draw (0,2) node[xic] {} -- (-1,1) node[xic] {} -- (0,0) node[xi] {} -- (1,1) node[xi] {};}

\DeclareSymbol{I1Xi4ac1}{2}{\draw[kernels2] (0,0) node[not] {} -- (-1,1) ; \draw[kernels2] (0,0) node[not] {} -- (1,1) node[xic] {} ;
\draw (-1,1) node[xi] {} -- (0,2) node[xic] {} -- (-1,3) node[xi] {};}%161

\DeclareSymbol{I1Xi4ac2}{2}{\draw[kernels2] (0,0) node[not] {} -- (-1,1) ; \draw[kernels2] (0,0) node[not] {} -- (1,1) node[xic] {} ;
\draw (-1,1) node[xi] {} -- (0,2) node[xi] {} -- (-1,3) node[xic] {};}

\DeclareSymbol{I1Xi4bp}{2}{\draw (0,0) node[not] {} -- (-1,1) node[xi] {} -- (0,2) ; \draw[kernels2] (0,2) node[not] {} -- (-1,3) node[xi] {};\draw[kernels2] (0,2)  -- (1,3) node[xi] {};
}

\DeclareSymbol{I1Xi4bc1}{2}{\draw (0,0) node[xic] {} -- (-1,1) node[xi] {} -- (0,2) ; \draw[kernels2] (0,2) node[not] {} -- (-1,3) node[xi] {};\draw[kernels2] (0,2)  -- (1,3) node[xic] {};
}%165

\DeclareSymbol{I1Xi4bc2}{2}{\draw (0,0) node[xic] {} -- (-1,1) node[xi] {} -- (0,2) ; \draw[kernels2] (0,2) node[not] {} -- (-1,3) node[xic] {};\draw[kernels2] (0,2)  -- (1,3) node[xi] {};
}

\DeclareSymbol{I1Xi4cp}{2}{\draw (0,0) node[not] {} -- (-1,1) node[not] {}; \draw[kernels2] (-1,1) -- (0,2) ; 
\draw[kernels2] (-1,1) -- (-2,2) node[xi] {} ;
\draw (0,2) node[xi] {} -- (-1,3) node[xi] {};}

\DeclareSymbol{I1Xi4cc1}{2}{\draw (0,0) node[xic] {} -- (-1,1) node[not] {}; \draw[kernels2] (-1,1) -- (0,2) ; 
\draw[kernels2] (-1,1) -- (-2,2) node[xi] {} ;
\draw (0,2) node[xic] {} -- (-1,3) node[xi] {};}%169

\DeclareSymbol{I1Xi4cc2}{2}{\draw (0,0) node[xic] {} -- (-1,1) node[not] {}; \draw[kernels2] (-1,1) -- (0,2) ; 
\draw[kernels2] (-1,1) -- (-2,2) node[xi] {} ;
\draw (0,2) node[xi] {} -- (-1,3) node[xic] {};}

\DeclareSymbol{I1Xi4abc1}{2}{\draw[kernels2] (0,0) node[not] {} -- (-1,1) ; \draw[kernels2] (0,0) node[not] {} -- (1,1) node[xic] {};\draw (-1,1) node[xi] {} -- (0,2) ; \draw[kernels2] (0,2) node[not] {} -- (-1,3) node[xic] {};\draw[kernels2] (0,2)  -- (1,3) node[xi] {}; }

\DeclareSymbol{I1Xi4abc2}{2}{\draw[kernels2] (0,0) node[not] {} -- (-1,1) ; \draw[kernels2] (0,0) node[not] {} -- (1,1) node[xic] {};\draw (-1,1) node[xi] {} -- (0,2) ; \draw[kernels2] (0,2) node[not] {} -- (-1,3) node[xi] {};\draw[kernels2] (0,2)  -- (1,3) node[xic] {}; }%173

\DeclareSymbol{R1}{0}{\draw (-1,1) node[xi] {} -- (0,0) node[not] {};%131
\draw[kernels2] (0,1.5) node[xic] {} -- (0,0) -- (1,1) node[xic] {};}
\DeclareSymbol{R2}{0}{\draw (-1,1) node[xic] {} -- (0,0) node[not] {};
\draw[kernels2] (0,1.5)  {} -- (0,0) -- (1,1)  {};
\draw (0,1.5) node[xi] {};
\draw (1,1) node[xic] {};
}%132
\DeclareSymbol{R3}{1}{\draw[kernels2] (-1,1.5)  {} -- (0,0) node[not] {} -- (1,1.5);%133
\draw (-1,1.5) node[xi] {};
\draw[kernels2] (0,3) {} -- (1,1.5) -- (2,3)  {};
\draw  (0,3) node[xic] {} ;
\draw (2,3) node[xic] {};}
\DeclareSymbol{R4}{1}{\draw[kernels2] (-1,1.5) node[xic] {} -- (0,0) node[not] {} -- (1,1.5);%134
\draw[kernels2] (0,3) {} -- (1,1.5) -- (2,3) node[xic] {};
\draw (0,3) node[xi] {};}

\DeclareSymbol{I1Xi4bcp}{2}{\draw (0,0) node[not] {} -- (-1,1) node[not] {}; \draw[kernels2] (-1,1) -- (0,2) ; %175
\draw[kernels2] (-1,1) -- (-2,2) node[xi] {} ; \draw[kernels2] (0,2) node[not] {} -- (-1,3) node[xi] {};\draw[kernels2] (0,2)  -- (1,3) node[xi] {};
}

\DeclareSymbol{I1Xi4bcc1}{2}{\draw (0,0) node[xic] {} -- (-1,1) node[not] {}; \draw[kernels2] (-1,1) -- (0,2) ; 
\draw[kernels2] (-1,1) -- (-2,2) node[xi] {} ; \draw[kernels2] (0,2) node[not] {} -- (-1,3) node[xi] {};\draw[kernels2] (0,2)  -- (1,3) node[xic] {};
}

\DeclareSymbol{I1Xi4bcc2}{2}{\draw (0,0) node[xic] {} -- (-1,1) node[not] {}; \draw[kernels2] (-1,1) -- (0,2) ; 
\draw[kernels2] (-1,1) -- (-2,2) node[xi] {} ; \draw[kernels2] (0,2) node[not] {} -- (-1,3) node[xic] {};\draw[kernels2] (0,2)  -- (1,3) node[xi] {};
} %177

\DeclareSymbol{2I1Xi4bc1}{2}{\draw[kernels2] (0,0) node[not] {} -- (-1,1) ;
\draw[kernels2] (0,0) -- (1,1);
\draw (-1,1) node[xic] {} -- (-1,2.5) node[xi] {};
\draw (1,1)  node[xic] {} -- (1,2.5) node[xi] {};
}

\DeclareSymbol{2I1Xi4bc2}{2}{\draw[kernels2] (0,0) node[not] {} -- (-1,1) ;
\draw[kernels2] (0,0) -- (1,1);
\draw (-1,1) node[xi] {} -- (-1,2.5) node[xic] {};
\draw (1,1)  node[xic] {} -- (1,2.5) node[xi] {};
}

\DeclareSymbol{diff2I1Xi4bc2}{2}{\draw (0,0) node[diff] {} -- (-1,1) ;
\draw (0,0) -- (1,1);
\draw (-1,1) node[xi] {} -- (-1,2.5) node[xic] {};
\draw (1,1)  node[xic] {} -- (1,2.5) node[xi] {};
}

\DeclareSymbol{2I1Xi4bc3}{2}{\draw[kernels2] (0,0) node[not] {} -- (-1,1) ;
\draw[kernels2] (0,0) -- (1,1);
\draw (-1,1) node[xic] {} -- (-1,2.5) node[xic] {};
\draw (1,1)  node[xi] {} -- (1,2.5) node[xi] {};
}

\DeclareSymbol{Xi41}{0}{\draw (0,1) -- (0.8,2.2) node[xic] {};\draw (0,-0.25) node[xi] {} -- (0,1) node[xi] {} -- (-0.8,2.2) node[xic] {};} 

\DeclareSymbol{Xi42}{0}{\draw (0,1) -- (0.8,2.2) node[xi] {};\draw (0,-0.25) node[xic] {} -- (0,1) node[xi] {} -- (-0.8,2.2) node[xic] {};}

\DeclareSymbol{Xi4ca1}{0}{\draw (0,1) -- (-1,2.2) node[xic] {};\draw (0,-0.25) node[xi] {} -- (0,1) ; \draw[kernels2] (0,1) node[not] {} -- (1,2.2) node[xic] {};
\draw[kernels2] (0,1) {} -- (0,2.7) node[xi] {};
}

\DeclareSymbol{Xi4ca2}{0}{\draw (0,1) -- (-1,2.2) node[xi] {};\draw (0,-0.25) node[xi] {} -- (0,1) ; \draw[kernels2] (0,1) node[not] {} -- (1,2.2) node[xic] {};
\draw[kernels2] (0,1) {} -- (0,2.7) node[xic] {};
}

\DeclareSymbol{Xi4cap}{0}{\draw (0,1) -- (-1,2.2) node[xi] {};\draw (0,-0.25) node[not] {} -- (0,1) ; \draw[kernels2] (0,1) node[not] {} -- (1,2.2) node[xi] {};
\draw[kernels2] (0,1) {} -- (0,2.7) node[xi] {};
}

\DeclareSymbol{Xi3a}{0}{
 \draw (-1,1)  node[xi] {} -- (0,0); 
 \draw (0,0) node[xi] {}  -- (1,1) node[xi] {};
}

\DeclareSymbol{Xi4ebc1}{0}{
\draw[kernels2] (0,2) node[xi] {} -- (-1,1) ; \draw[kernels2] (-2,2)  node[xic] {} -- (-1,1) ; \draw (-1,1)  node[not] {} -- (0,0); 
 \draw (0,0) node[xic] {}  -- (1,1) node[xi] {};
}

\DeclareSymbol{Xi4ebc2}{0}{
\draw[kernels2] (0,2) node[xi] {} -- (-1,1) ; \draw[kernels2] (-2,2)  node[xi] {} -- (-1,1) ; \draw (-1,1)  node[not] {} -- (0,0); 
 \draw (0,0) node[xic] {}  -- (1,1) node[xic] {};
}

\DeclareSymbol{Xi2cbispex}{0}{\draw[kernels2] (0,1) -- (0.8,2.2) node[xi] {};\draw (0,-0.25) node[xie] {} -- (0,1); \draw[kernels2] (0,1) node[not] {} -- (-0.8,2.2) node[xi] {};}

\DeclareSymbol{Xi2cbis1p}{0}{\draw (0,1) -- (-0.8,2.2) node[xi] {};\draw (0,-0.25) node[not] {} -- (0,1) node[xi] {}; }

\DeclareSymbol{Xi2Xp}{-2}{\draw (0,-0.25) node[not] {} -- (-1,1) node[xix] {};} % 229 not used

\DeclareSymbol{I1XiIXib}{0}{\draw  (0,-0.25) node[xi] {} -- (0,1) node[not] {};
\draw[kernels2] (0,1) -- (0,2.25) ; \draw (0,2.25) node[xi]{}; }

\DeclareSymbol{IXi2b}{0}{\draw  (0,-0.25) node[xi] {} -- (0,1) node[not] {};
\draw (0,1) -- (0,2.25) ; \draw (0,2.25) node[xi]{}; }

\DeclareSymbol{IXi2bex}{0}{\draw  (0,-0.25) node[xi] {} -- (0,1) node[xie] {};
\draw (0,1) -- (0,2.25) ; \draw (0,2.25) node[xi]{}; }

 \def\1{\mathbf{\symbol{1}}}

\def\one{\mathbf{1}}

\def\eps{\varepsilon}

\DeclareSymbol{diff}{0}{
\draw (0,0.5) node[diff] {};
}

\DeclareSymbol{diff1}{0}{
\draw (0,0.5) node[diff1] {};
}

\DeclareSymbol{diff2}{0}{
\draw (0,0.5) node[diff2] {};
}

\DeclareSymbol{geo}{0}{
\draw (0,0) node[diff] {};
\draw (0.3,0) node[diff] {};
}

\DeclareSymbol{generic}{0}{
\draw (0,0.6) node[xi] {};
}

\DeclareSymbol{g}{0}{
\draw (0,0.6) node[g] {};
}

\DeclareSymbol{Ito}{0}{
\draw (0,0.6) node[xies] {};
}

\DeclareSymbol{Itob}{0}{
\draw (0,0.6) node[xiesf] {};
}

\DeclareSymbol{greycirc}{0}{
\draw (0,0.3) node[xi] {};
}

\DeclareSymbol{not}{0}{
\draw (0,0.6) node[not] {};
\draw[tinydots] (0,0.6) circle (0.8);
}

\DeclareSymbol{genericb}{0}{
\draw (0,0.6) node[xic] {};
}

\DeclareSymbol{bluecirc}{0}{
\draw (0,0.3) node[xic] {};
}

\DeclareSymbol{genericxix}{0}{
\draw (0,0.6) node[xix] {};
}

\DeclareSymbol{genericX}{0}{
\draw (0,0.6) node[X] {};
}

\DeclareSymbol{diffIto}{1}{
\draw  (0,2.5) -- (0,0) ;
\draw (0,-0.1) node[diff] {};
\draw (0,2.5) node[xies] {};
}
\DeclareSymbol{Itodiff}{2}{
\draw(0,2.9) -- (0,-0.2);
\draw (0,2.9) node[diff] {};
\draw (0,-0.1) node[xies] {};
}

\DeclareSymbol{diffgeneric}{1}{
\draw  (0,2.5) -- (0,0) ;
\draw (0,-0.1) node[diff] {};
\draw (0,2.5) node[xi] {};
}

\DeclareSymbol{genericdiff}{2}{
\draw(0,2.9) -- (0,-0.2);
\draw (0,2.9) node[diff] {};
\draw (0,-0.1) node[xi] {};
}

\DeclareSymbol{diffdot}{2}{
\draw  (0,3) -- (0,-0.1) ;
\draw (0,3) node[not] {};
\draw (0,-0.1) node[diff] {};
}

\DeclareSymbol{diffdotmini}{0}{
\draw  (0,0) -- (0,1.2) ;
\draw (0,1.2) node[not] {};
\draw (0,0) node[diffmini] {};
}

\DeclareSymbol{dotdiff}{2}{
\draw[kernelsmod]  (0,3) -- (0,-0.1) ;
\draw (0,3) node[diff] {};
\draw (0,-0.1) node[not] {};
}

\DeclareSymbol{dotdiff1}{2}{
\draw[kernelsmod]  (0,3) -- (0,-0.1) ;
\draw (0,3) node[diff1] {};
\draw (0,-0.1) node[not] {};
}

\DeclareSymbol{dotdiff1mini}{0}{
\draw[kernelsmod]  (0,1.2) -- (0,0) ;
\draw (0,1.2) node[diffmini] {};
\draw (0,0) node[not] {};
}

\DeclareSymbol{dotdiff2}{2}{
\draw (0,3) -- (0,-0.1) ;
\draw (0,3) node[diff] {};
\draw (0,-0.1) node[not] {};
}

\DeclareSymbol{dotdiff2mini}{0}{
\draw (0,1.2) -- (0,0) ;
\draw (0,1.2) node[diffmini] {};
\draw (0,0) node[not] {};
}

\DeclareSymbol{dotdiffstraight}{0}{
\draw  (0,3) -- (0,-0.1) ;
\draw (0,3) node[diff] {};
\draw (0,-0.1) node[not] {};
}

\DeclareSymbol{arbre1}{0}{
\draw  (0,0) -- (1.5,1.5) ;
\draw (1.5,1.5) node[not] {};
\draw (0,0) node[not] {};
}

\DeclareSymbol{arbre2}{0}{
\draw  (0,0) -- (1.5,1.5) ;
\draw[kernelsmod] (0,0) -- (-1.5,1.5);
\draw (1.5,1.5) node[not] {};
\draw (0,0) node[not] {};
\draw (-1.5,1.5) node[xi] {};
}

\DeclareSymbol{arbre3}{0}{
\draw  (0,0) -- (1.5,1.5) ;
\draw[kernelsmod] (1.5,1.5) -- (0,3);
\draw (0,0) node[not] {};
\draw (1.5,1.5) node[not] {};
\draw (0,3) node[xi] {};
}

\DeclareSymbol{treeeval}{0}{
\draw (0,0) -- (1,1);
\draw (0,0) node[xi] {};
\draw (1.25,1.25) node[xi] {};
\draw (-0.6,0.6) node[]{\tiny{$i$}};
\draw (0.65,1.85) node[]{\tiny{$j$}};
}

\DeclareSymbol{testeval}{0}{
\draw (0,0) -- (1,1);
\draw (0,0) -- (-1,1);
\draw (0,0) node[xi] {};
\draw (1.25,1.25) node[xi] {};
\draw (-1.25,1.25) node[xi] {};
\draw (-0.6,-0.6) node[]{\tiny{$i$}};
\draw (0.65,1.85) node[]{\tiny{$j$}};
\draw (-1.95,1.85) node[]{\tiny{$k$}};
}

\DeclareSymbol{treeeval2}{0}{
\draw[kernelsmod] (-0.25,-1) -- (1,0.5) ;
\draw[kernelsmod] (1,0.5) -- (-0.25,2);
\draw (1,0.5) node[diff2] {};
\draw (-0.25,-1) node[not] {};
\draw (-0.25,2) node[xi] {};
\draw (-0.6,1.2) node[]{\tiny{1}};
}

\DeclareSymbol{arbreact}{1}{
\draw (0,0) node[not] {};
\draw[kernelsmod] (0,0) -- (1,1);
\draw[kernelsmod] (0,0) -- (-1,1);
\draw (-1,1) node[xic] {};
\draw  (0,2) -- (1,1) ;
\draw (0,2) node[xic] {};
\draw (1,1) node[xi] {};
}

\DeclareSymbol{arbreact1}{0}{
\draw (0,-1.5) -- (0,0);
\draw[kernelsmod] (0,0) -- (1,1);
\draw[kernelsmod] (0,0) -- (-1,1);
\draw  (0,2) -- (1,1) ;
\draw (0,-1.5) node[diff] {};
\draw (0,0) node[not] {};
\draw (-1,1) node[xic] {};
\draw (0,2) node[xic] {};
\draw (1,1) node[xi] {};
}

\DeclareSymbol{arbreact2}{0}{
\draw (0,-0.75) -- (-1,0.5); 
\draw (0,-0.75) -- (1,0.5);
\draw (0,1.5) -- (1,0.5);
\draw (0,1.5) node[xic] {};
\draw (1,0.5) node[xi] {};
\draw (-1,0.5) node[xic] {};
\draw (0,-0.75) node[diff] {};
}

\DeclareSymbol{arbreact3}{0}{
\draw[kernelsmod] (0,-0.75) -- (-1,0.5); 
\draw[kernelsmod] (0,-0.75) -- (1,0.5);
\draw (0,1.75) -- (1,0.5);
\draw (2,1.75) -- (1,0.5);
\draw (0,1.75) node[xic] {};
\draw (1,0.5) node[diff] {};
\draw (-1,0.5) node[xic] {};
\draw (2,1.75) node[xi] {};
\draw (0,-0.75) node[not] {};
}

\DeclareSymbol{pre_im_I1Xitwo}{0}{
\draw[kernels2] (0,-0.3) node[not] {} -- (-0.6,0.7) ;
\draw[kernels2] (0,-0.3) -- (0.6,0.7);
\draw (0,0.9) node[g] {};
}

\DeclareSymbol{pre_im_cI1Xi4ab}{2}{
\draw[kernels2] (0,-1) node[not] {} -- (-0.6,0) ;
\draw[kernels2] (0,-1) -- (0.6,0);
\draw (0,0.2) node[g] {};
\draw (0,0.6) -- (0,1.5);
\draw[kernels2] (0,1.5) node[not] {} -- (-0.6,2.5) ;
\draw[kernels2] (0,1.5) -- (0.6,2.5);
\draw (0,2.7) node[g] {};
}%46

\DeclareSymbol{pre_im_I1Xi4acc2}{0}{
\draw[kernels2] (-1,-0.5) node[not] {} -- (-1.6,0.5) ;
\draw[kernels2] (-1,-0.5) -- (-0.4,0.5);
\draw[kernels2] (-1,-0.5) -- (0.2,-1.5) node[not] {} ;
\draw (-1,1.1) -- (-1,2);
\draw[kernels2] (0.2,-1.5) -- (0.2,2);
\draw (-1,0.7) node[g] {};
\draw (-0.3,2.2) node[g] {};
}

\DeclareSymbol{pre_im_I1Xi4abcc2}{2}{
\draw[kernels2] (0,-1) node[not] {} -- (-1,0) node[not] {};
\draw[kernels2] (-1,1.2) node[not] {} -- (-1,0);
\draw[kernels2] (-1,1.2) -- (-1.5,2.5);
\draw[kernels2] (-1,1.2) -- (-0.5,2.5);
\draw[kernels2] (-1,0) -- (0.7,2.5);
\draw[kernels2] (0,-1) -- (1.5,2.5);
\draw (-1,2.7) node[g] {};
\draw (1,2.7) node[g] {};
}

\DeclareSymbol{pre_im_2I1Xi4c1}{2}{
\draw[kernels2] (0,-0.5) node[not] {} -- (-1,0.5) node[not] {};
\draw[kernels2] (0,-0.5) -- (1,0.5) node[not] {};
\draw[kernels2] (-1,0.5) node[not] {}-- (-1.7,2);
\draw[kernels2]  (-1,2) -- (1,0.5);
\draw[kernels2] (-1,0.5) -- (1,2);
\draw[kernels2] (1,0.5) -- (1.7,2);
\draw (-1.2,2.2) node[g] {};
\draw (1.2,2.2) node[g] {};
}

\DeclareSymbol{pre_im_Xi4eabisc2}{2}{
\draw[kernels2] (1.2,-0.5) node[not] {} -- (-0.7,0.8) ;
\draw[kernels2] (1.2,-0.5) -- (0.4,0.8);
\draw (0,1.4)  -- (0,2.2);
\draw (1.2,2.2) -- (1.2,-0.6);
\draw (0,1) node[g] {};
\draw (0.6,2.4) node[g] {};
}

\DeclareSymbol{pre_im_Xi4eabisc22}{2}{
\draw (1.2,-0.5) node[not] {} -- (-0.7,0.8) ;
\draw[kernels2] (1.2,-0.5) -- (0.4,0.8);
\draw (0,1.4)  -- (0,2.2);
\draw[kernels2] (1.2,2.2) -- (1.2,-0.6);
\draw (0,1) node[g] {};
\draw (0.6,2.4) node[g] {};
}

\DeclareSymbol{pre_im_Xi4eabisc222}{2}{
\draw[kernels2] (0.4,-0.5) node[not] {} -- (-0.6,1) ;
\draw[kernels2] (1.2,0) -- (0.3,1);
\draw (0,1.1)  -- (0,2.5);
\draw[kernels2] (1.2,2.5) -- (1.2,0) node[not] {} -- (0.4,-0.6);
\draw (0,1.2) node[g] {};
\draw (0.6,2.5) node[g] {};
}

\DeclareSymbol{pre_im_Xi4eabc2}{2}{
\draw (0,-0.5) node[not] {} -- (-1,0.5) node[not] {};
\draw[kernels2] (-1,0.5) -- (-1.5,2);
\draw[kernels2] (0,-0.5)  -- (0.7,2);
\draw[kernels2] (-1,0.5) -- (-0.5,2);
\draw[kernels2] (0,-0.5) -- (1.5,2);
\draw (-1,2.2) node[g] {};
\draw (1,2.2) node[g] {};
}

\DeclareSymbol{pre_im_Xi4eabbisc2}{2}{
\draw[kernels2] (0,-0.5) node[not] {} -- (-1,0.5) node[not] {};
\draw[kernels2] (-1,0.5) -- (-1.5,2);
\draw[kernels2] (0,-0.5)  -- (0.7,2);
\draw[kernels2] (-1,0.5) -- (-0.5,2);
\draw (0,-0.5) -- (1.5,2);
\draw (-1,2.2) node[g] {};
\draw (1,2.2) node[g] {};
}

\DeclareSymbol{pre_im_I1Xi4abcc1}{2}{
\draw[kernels2] (0,-1) node[not] {} -- (-1,0) node[not] {};
\draw[kernels2] (0,1.1) node[not] {} -- (-1,0);
\draw[kernels2] (-1,0) -- (-1.5,2.5);
\draw[kernels2] (0,1.1) node[not] {} -- (-0.5,2.5);
\draw[kernels2] (0,1.1) -- (0.5,2.5);
\draw[kernels2] (0,-1) -- (1.5,2.5);
\draw (-1,2.7) node[g] {};
\draw (1,2.7) node[g] {};
}

\DeclareSymbol{pre_im_Xi4eabc1}{2}{
\draw (0,-0.5) node[not] {} -- (-1,0.5) node[not] {};
\draw[kernels2] (-1,0.5) -- (-1.5,2);
\draw[kernels2] (0,-0.5)  -- (-0.5,2);
\draw[kernels2] (-1,0.5) -- (0.8,2);
\draw[kernels2] (0,-0.5) -- (1.5,2);
\draw (-1,2.2) node[g] {};
\draw (1,2.2) node[g] {};
}

\DeclareSymbol{pre_im_Xi4ba1b}{2}{
\draw[kernels2] (0,0) node[not] {}  -- (1.8,1.5);
\draw[kernels2] (0,0) -- (0.8,1.5);
\draw (0,-0.1) -- (-1.8,1.5);
\draw (0,-0.1) -- (-0.8,1.5);
\draw (-1,1.7) node[g] {};
\draw (1,1.7) node[g] {};
}

\DeclareSymbol{pre_im_Xi4ba2}{2}{
\draw (0,-0.1) node[not] {}  -- (1.8,1.5);
\draw[kernels2] (0,0) -- (0.8,1.5);
\draw (0,-0.1) -- (-1.8,1.5);
\draw[kernels2] (0,0) -- (-0.8,1.5);
\draw (-1,1.7) node[g] {};
\draw (1,1.7) node[g] {};
}

\DeclareSymbol{pre_im_Xi4cabc2}{2}{
\draw[kernels2] (0,-0.5) node[not] {} -- (-1,0.5) node[not] {};
\draw[kernels2] (-1,0.5) -- (-1.5,2);
\draw (-1,0.5)  -- (0.7,2);
\draw[kernels2] (-1,0.5) -- (-0.5,2);
\draw[kernels2] (0,-0.5) -- (1.5,2);
\draw (-1,2.2) node[g] {};
\draw (1,2.2) node[g] {};
}

\DeclareSymbol{pre_im_Xi4cabc1}{2}{
\draw[kernels2] (0,-0.5) node[not] {} -- (-1,0.5) node[not] {};
\draw (-1,0.5) -- (-1.5,2);
\draw[kernels2] (-1,0.5)  -- (0.7,2);
\draw[kernels2] (-1,0.5) -- (-0.5,2);
\draw[kernels2] (0,-0.5) -- (1.5,2);
\draw (-1,2.2) node[g] {};
\draw (1,2.2) node[g] {};
}

\DeclareSymbol{pre_im_Xi4eabbisc1}{2}{
\draw[kernels2] (0,-0.5) node[not] {} -- (-1,0.5) node[not] {};
\draw[kernels2] (-1,0.5) -- (-1.5,2);
\draw[kernels2] (0,-0.5)  -- (-0.5,2);
\draw[kernels2] (-1,0.5) -- (0.8,2);
\draw (0,-0.5) -- (1.5,2);
\draw (-1,2.2) node[g] {};
\draw (1,2.2) node[g] {};
}

\DeclareSymbol{pre_im_1}{0}{
\draw[kernels2] (0,-0.5) node[not] {} -- (-0.6,0.5) ;
\draw[kernels2] (0,-0.5) -- (0.6,0.5);
\draw (0,1.1)  -- (-0.55,2);
\draw (0,1.1)  -- (0.55,2);
\draw (0,0.7) node[g] {};
\draw (0,2.2) node[g] {};
}

\DeclareSymbol{disconnect}{0}{
\draw[kernels2] (0,-0.5) node[not] {} -- (-0.6,0.5) ;
\draw[kernels2] (0,-0.5) -- (0.6,0.5);
\draw (-0.55,1.1)  -- (-0.55,2.3);
\draw (0.55,2.3) -- (0.55,1.5) -- (1.2,1.5) -- (1.2,3.5) -- (0.55,3.5) -- (0.55,2.7);
\draw (0,0.7) node[g] {};
\draw (0,2.5) node[g] {};
}

\DeclareSymbol{pre_im_2}{2}{\draw[kernels2] (0,0) node[not] {} -- (-1,1) node[not] {};
\draw[kernels2] (0,0) -- (1,1) node[not] {};
\draw[kernels2] (-1,1) -- (-1.5,2.5);
\draw[kernels2] (-1,1) -- (-0.5,2.5);
\draw[kernels2] (1,1) -- (0.5,2.5);
\draw[kernels2] (1,1) -- (1.5,2.5);
\draw (-1,2.7) node[g] {};
\draw (1,2.7) node[g] {};
}

\DeclareSymbol{CX_rec}{0}{
\draw [black] (-0.3,1) to (-0.3,-0.3);
\draw [black] (0.3,1) to (0.3,-0.3);
\draw [black] (-0.3,1) to (-0.3,2.3);
\draw [black] (0.3,1) to (0.3,2.3);
\draw (0,1) node[rec] {};
}

\DeclareSymbol{CX_cerc}{0}{
\draw [black] (0,1) to (0,-0.3);
\draw (0,1) node[cerc] {};
}

\DeclareMathAlphabet{\mathpzc}{OT1}{pzc}{m}{it}

\let\eps\varepsilon

\def\eqref#1{(\ref{#1})}

\makeatletter % Stolen from the internet to make a fat \cdot which isn't as fat as a \bullet
\newcommand*{\bigcdot}{}% Check if undefined
\DeclareRobustCommand*{\bigcdot}{%
  \mathbin{\mathpalette\bigcdot@{}}%
}
\newcommand*{\bigcdot@scalefactor}{.5}
\newcommand*{\bigcdot@widthfactor}{1.15}
\newcommand*{\bigcdot@}[2]{%
  \sbox0{$#1\vcenter{}$}% math axis
  \sbox2{$#1\cdot\m@th$}%
  \hbox to \bigcdot@widthfactor\wd2{%
    \hfil
    \raise\ht0\hbox{%
      \scalebox{\bigcdot@scalefactor}{%
        \lower\ht0\hbox{$#1\bullet\m@th$}%
      }%
    }%
    \hfil
  }%
}
\makeatother

\tcbset
{colframe=boxcolor,colback=symbols!7!pagebackground,coltext=pageforeground,
fonttitle=\bfseries,nobeforeafter,center title,size=fbox,boxsep=1.5pt,
top=0mm,bottom=0mm,boxsep=0mm,tcbox raise base}

\def\two{{\<generic>\kern0.05em\<genericb>}}
\def\twoI{{\<Ito>\kern0.05em\<Itob>}}

\def\mail#1{\burlalt{#1}{mailto:#1}}

\usepackage{thmtools}

\begin{document}

\title{Renormalisation in the flow approach for singular SPDEs}

\author{Yvain Bruned and Aurélien Minguella}
\institute{ 
 Universite de Lorraine, CNRS, IECL, F-54000 Nancy, France
  \\
Email:\ \begin{minipage}[t]{\linewidth}
\mail{yvain.bruned@univ-lorraine.fr}
\\ \mail{aurelien.minguella@univ-lorraine.fr}
\end{minipage}}

\maketitle 

\begin{abstract}
In this work, we study the renormalisation of singular SPDEs in the flow approach recently developed by Duch. We introduce a general ansatz based on decorated trees for the solution of the flow equation. The ansatz is renormalised in a recursive way, in the sense of the trees, via local extractions introduced for regularity structures. We derive the renormalised equation from this ansatz and show that the renormalisation scheme is identical to that appearing in the context of regularity structures, thus matching the BPHZ renormalisation. 
\end{abstract}
\setcounter{tocdepth}{2}
\setcounter{secnumdepth}{4}
\tableofcontents
\newpage

\section{Introduction}

\quad \, The last decade has seen major progress in singular stochastic partial differential equations (SPDEs) with two solutions theories: regularity structures \cite{reg} and paracontrolled calculus \cite{GIP15}. They both allow to construct iterated stochastic integrals coming from the iteration of Duhamel's formula. The most involved part in these theories is to renormalise these terms in order to get well-defined distributions. In this direction, in the context of regularity structures, the work \cite{BHZ} has developed a heavy algebraic machinery based on Hopf algebras yielding black box theorems that allow to tackle very broad classes of equations. With \cite{BCCH, CH16}, the task of showing well-posedness for general subcritical SPDEs was completed. These methods are based on the renormalisation of Feynman diagrams arising from perturbative quantum field theory. The procedure is inspired by the famous BPHZ renormalisation \cite{BP57,KH69,WZ69}.\\

Recently, in \cite{Duc21}, Duch proposed a new solution theory to SPDEs. It is inspired by the continuous renormalisation group method introduced by Polchinski \cite{P84}, sometimes referred to as the Polchinski flow. See \cite{M03} for a review of this method in physics. It is based on a recursive construction of the stochastic terms using a truncated version of the Polchinski equation. By deriving a similar construction on the cumulants of these stochastic terms, Duch manages to prove inductively the so-called stochastic estimates, that are essential to invoke a fixed-point argument to show the well-posedness of SPDEs. More precisely, this is done by defining a continuous flow of scales, indexed by a parameter $\mu\in [0,1]$, that suppresses the small scales. Such an approach had already been used in the context of singular SPDEs by Kupiainen and Marcozzi in \cite{K16,KM17}. However, it was based on a discrete renormalisation group, \textit{i.e.} the scales where indexed by a discrete parameter $\mu_n$. The authors have been able to tackle the $\phi^4_3$ and the generalized KPZ equations. Nevertheless, this technique did not encompass equations closer to criticality, as the computations had to be done by hand and would have become intractable. The main contribution of \cite{Duc21} was to bring the continuity in the renormalisation group, opening the path to treating the full subcritical regime. For a smooth introduction to the method, we refer to the lecture notes \cite{Duc23}.\\

In this work, we take the opposite view of the works mentioned above. We choose to clarify the algebraic side of the method. The purpose of this is to compare the renormalisation procedure of the flow approach with a known one, namely the one in regularity structures. As a point of comparison, we choose the recursive framework first developed in \cite{BR18} and perfected in \cite{BB21b}. The word recursive in this context refers to the depth of the trees and not the Polchinski flow recursion, to remove any doubt. The renormalisation involves then local extractions at the root of the trees. We draw a clear link between the renormalisation in the two contexts. However, we do not show the well-posedness for SPDEs, as this result has already been claimed in very general cases in the papers cited in this introduction. We rather shed a new light on the hidden combinatorics behind the method.\\

Let us mention other recent works using the Polchinski flow. Duch reused his method in an elliptic context \cite{Duc22} (the first paper \cite{Duc21} being in the parabolic context), as well as in a non-commutative framework with \cite{Duc24} to construct the Gross-Neveu model. With other co-authors, they construct the $\phi^4_3$ field in the full subcritical regime using parabolic stochastic quantisation in \cite{DGR23}. Chandra and Ferdinand proposed to use elementary differentials in \cite{CF24a} in order to deal with equations having non-polynomial non-linearities. They applied this technique again in \cite{CF24b} to show how the flow approach can be used to prove the well-posedness of rough differential equations. Outside of singular SPDEs strictly speaking, the Polchinski flow has been used to show fast convergence to equilibrium of some spin systems arising from statistical physics and euclidean field theory, with log-Sobolev inequalities by Bauerschmidt and Bodineau \cite{BB21}. We refer to the useful lecture notes \cite{BBD24}. Before other contributions, this seminal work was dealing with the Sine-Gordon model. Very close techniques have also been used by Gubinelli and Meyer in \cite{Gub24} to propose a novel method of stochastic quantisation of this model based on forward-backward SDEs, allowing to show many relevant results on the constructed measure.\\

Let us be more specific about the techniques of the present work. We study the renormalisation of very general parabolic equations of the form
\begin{equation}\label{eqintro}
\begin{split}
(\D_t-\Delta)\psi&=g(\psi,\nabla\psi,\dots,\nabla^q\psi)+f(\psi,\nabla\psi,\dots,\nabla^p\psi)\xi,\\
\psi(0,\cdot)&=\psi^0,
\end{split}
\end{equation}
on $\R_+\times\T^d$, with fixed $p,q<2$, which is supposed to be subcritical. $\xi$ denotes a Gaussian noise. It covers a very large class of equations such as the $\phi^4$ equation
\begin{equation}
(\D_t-\Delta)\psi=-\psi^3+\xi,
\end{equation}
or the generalized KPZ equation
\begin{equation}
(\D_t-\Delta)\psi=g(\psi)(\D_x\psi)^2+f(\psi)\xi.
\end{equation}
Note that we only consider scalar-valued equations, but our work could be easily adapted, following the lines of \cite{BB21b}, to systems of equations, in order to tackle even more general equations, such as the geometric stochastic heat equation studied in \cite{BGHZ}. One has to carry more decorations in the decorated trees. We choose not to write explicitly such a framework in the interest of brevity. This remark is what leads us to choose trees as main combinatorial objects. If we denote $G$, the Green function associated to the parabolic operator in the equations above, the main idea of the flow approach is to introduce a family of kernels $(G_\mu)$ interpolated along the scale $\mu\in[0,1]$, and satisfying $G_0=G$, and $G_1$ is smooth, as well as some support properties \ref{cutscale}, that represent the fact that $G_\mu$ cuts the small scales. More details about the flow approach are given in Section \eqref{flownut} below. We also denote $\dot G_\mu=\D_\mu G_\mu$. Writing $S[\psi]$ for the force term on the right-hand side of the equation \eqref{eqintro}, the flow approach creates an interpolation $S_\mu[\psi]$ satisfying $S_0[\psi]=S[\psi]$. Note that in this introduction, we ignore the trick that is necessary to take the initial condition into account in the parabolic case. Thus the notations here might differ a bit from the rest of the paper, and will stay a bit vague to gain in readability. We make an ansatz of the general form
\begin{equation}
S_\mu[\psi]=\sum_{\tau\in T_0^*}\frac{1}{S(\tau)}(\Pi_\mu^R\tau)[\psi],
\end{equation}
where $T_0^*$ is a set of decorated trees (we give a comprehensive introduction to these objects in Section \ref{trees}), $S(\tau)$ is a combinatorial factor, and $\Pi_\mu^R$ is what we call an evaluation map. Before saying more about this object, let us continue with the flow method. The main idea is to define $S_\mu$ so that it solves the Polchinski equation (see Equation \eqref{Poleq} in the next section for a precise statement),
\begin{equation}
\D_\mu S_\mu+DS_\mu\cdot \dot G_\mu * S_\mu=0,
\end{equation}
or at least a truncated version of it. In the seminal works by Duch \cite{Duc21, Duc22}, the coefficients $\Pi_\mu^R\tau$ are constructed recursively by imposing that they solve this equation. This is done formally by grafting these coefficients one onto another. The renormalisation is then added at the beginning of the induction so that the following terms are also renormalised.\\
Our main contribution is actually to give an explicit definition of the evaluation map $\Pi_\mu^R$. This definition is strongly inspired by the recursive (in the sense of the trees) framework of \cite{BR18,BB21b}. The major difference is that unlike in regularity structures, the terms are not localized, \textit{i.e.} the function $\psi$ appears inside the tree, where in regularity structures it has been taken out (see Example \ref{exeval} for clarity). Once this definition, and the ansatz, are set, the subsequent properties are derived quite naturally. It is defined recursively in the following way (for a complete definition see \ref{eval})
\begin{equation}
\Pi_\mu^R\tau=(\ell\Upsilon\otimes\Pi_\mu^{R\times})(\Mloc\otimes\mathrm{id})\Delta_r.
\end{equation}
Let us take a moment to explain the objects appearing here. The map $\Delta_r$ is an extraction-contraction coproduct that acts on trees by extracting sub-trees at the root and contracting the remaining part. It is the local (in the sense of the trees) version of the general negative coproduct defined in \cite{BHZ}, that acts similarly but not only at the root. The map $\Mloc$ is what we call the localization map. It formally takes a tree and inserts polynomials everywhere. We use it to bring in the evaluation map the localisation of the procedure, and that is done in the previous works not in the beginning but at the very end. The linear map $\ell\Upsilon$ is composed of a character $\ell$ that vanishes on trees of non-negative degree, so that only the trees of negative degree (corresponding to the diverging part of the associated Feynman diagrams) are used in the renormalisation, and of the elementary differentials $\Upsilon$ that are functionals in $\psi$. We give a detailed definition of these objects in Section \ref{eldiff}. One of the main advantages of using elementary differentials is that it allows a lot of freedom in the definition of the sets of trees involved, since $\Upsilon$ is going to vanish on unnecessary trees, then permitting to not write a precise rule to construct the family of trees. Finally, the map $\Pi_\mu^{R\times}$ is a multiplicative evaluation map giving the rest of the construction to the induction. The functionals $\Upsilon$ are then being integrated over.
More technically, we introduce in Section \ref{abstractflow} a linear operator $\dmu$ acting on trees by changing the decoration of some edges representing the kernel $\D^a(G-G_\mu)$ to the one representing $\D^a\dot G_\mu$ (drawn in red). We show  a commutation relation in Proposition \ref{commutation}
\begin{equation}
\D_\mu\Pi_\mu^R=\Pi_\mu^R\dmu.
\end{equation}
Note that, according to our definitions, if no renormalisation was involved, this relation would be trivial. We also introduce, as it has already been done in previous works, a bilinear operator $B$ corresponding to the Fréchet term in the Polchinski equation (see Definition \ref{defb}). We prove a morphism property in Proposition \ref{graft} with the grafting operators
\begin{equation}
B(\Pi_\mu^R\sigma,\Pi_\mu^R\tau)=\sum_{a\in\N^{d+1}}\Pi_\mu^R(\sigma\rgraft_a\tau).
\end{equation}
Note that the last sum is in fact finite, as the decorations on the edges of the trees cannot be of order more that $q$, that was fixed above. It will be the case for most of the sums in this paper, but we prefer this lighter notation. The graft $\rgraft_a$ represents grafting with an edge that represents the kernel $\D^a\dot G_\mu$. A subtle phenomenon already occurs in this expression. It roughly says that when taking to renormalised terms, grafting one onto another gives the right renormalised tree, with in the middle a red edge. The renormalisation does not act on this edge. More precisely, during the extraction procedure of $\Delta_r$, the red edges are left untouched. What could be seen as a mere algebraic trick has in fact an analytical interpretation.\\

With the help of these two properties, we can show that our definition of $\Pi_\mu^R$ satisfies a flow equation on the coefficients (see Proposition \ref{flowcoeff}) that is a rewriting of the equation that defines the force coefficients. One can then make the following observation: we define explicitly coefficients through an evaluation map. These coefficients satisfy the same inductive relation and initial conditions as in the original method. Therefore they indeed match the other construction, but where in the original version the renormalisation was implicit, it is clearly identified in our case.\\

We also show some subsequent results. In Section \ref{coherence}, we show that our ansatz defined independently of the Polchinski equation, actually solves it, or at least a truncated version. This sanity check may be viewed as the counterpart of the coherence result proved in \cite{BCCH} in the context of regularity structures. The result is much simpler in the present context. We show in Theorem \ref{coherence} that the ansatz almost satisfies the Polchinski equation, in the sense that it satisfies the following truncated version:
\begin{equation}
\D_\mu S_\mu^\gamma+\mcP_{\gamma-\lambda}(D S_\mu^\gamma\cdot\dot G_\mu* S_\mu^\gamma)=0,
\end{equation}
where $S_\mu^\gamma$ is a truncated version of the ansatz, at a given level $\gamma$, and $\mcP_{\gamma-\lambda}$ is a projector on some space generated by the $\Pi_\mu^R\tau$'s. We give a quick comparison in Remark \ref{comparison} of our method with the diagram-free approach in regularity structures introduced first in \cite{LOTP} for multi-indices. For decorated trees, one can look at the references \cite{BN23,BB23,HS23,BH23}. We display strong similarities, both at the level of some algebraic relations as well as in the global idea of the method.\\

In Section \ref{backspde}, we explain how the renormalisation of the coefficients expresses at the level of the SPDE. We find back in Theorem \ref{renormalised_equation} the known result in the context of regularity structures. It is the counterpart of another important result of \cite{BCCH}, reproved with the recursive framework in \cite{BB21b}. Similar ideas have also been used in the context of multi-indices (first introduced for singular SPDEs in \cite{OSSW}) in \cite{BL24} relying on the algebraic construction proposed in \cite{LOT}. Our proof is once again surprisingly elementary. It is yet another way to check that the definition of the evaulation map is the right one. Formally, we show that the renormalised SPDE takes the form
\begin{equation}
(\D_t-\Delta)\phi=g(\phi,\dots,\nabla^q\phi)+f(\phi,\dots,\nabla^p\phi)\xi+\sum_{\sigma}\frac{\ell(\sigma)}{S(\sigma)}\Upsilon[\sigma][\phi],
\end{equation}
where the last sum is indexed by a set of trees of negative degree including polynomials.\\

In Section \ref{bphz}, we perform a localization of the evaluation map in Theorem \ref{localization}, as it is done classically in the initialization of the flow on trees of negative degree (also called relevant terms). Concretely speaking, we take out the functional parameters stuck in the convolutions to place them at the root of the trees, representing the stochastic terms, accordingly to what is done in regularity structures. This result allows us to make a link in Theorem \ref{bphztheorem} with the BPHZ renormalisation by choosing a precise character $\ell_{ \text{\tiny{BPHZ}}}$, being the same one as in regularity structures, thus concluding the characterization of the renormalization in the flow approach.\\

Our main results can be summarised as follows:
\begin{itemize}
	\item Definition \ref{eval}:  General decorated trees ansatz for the flow equation. 
	\item Proposition \ref{commutation}: Commutation relation.
	\item Proposition \ref{graft}: Identity connecting the operator $B$ and the grafting operator.
	\item Theorem \ref{coherence}:  Coherence property, the general ansatz solves the flow equation.
	\item Theorem \ref{renormalised_equation}: Renormalised equation.
	\item Theorems \ref{localization} and \ref{bphztheorem}: Localization of the evaluation map and selection of the character.
\end{itemize}

\subsection*{Acknowledgements}

{\small
	Y.B. and A.M. gratefully acknowledge funding support from the European Research Council (ERC) through the ERC Starting Grant Low Regularity Dynamics via Decorated Trees (LoRDeT), grant agreement No.\ 101075208. Views and opinions expressed are however those of the author(s) only and do not necessarily reflect those of the European Union or the European Research Council. Neither the European Union nor the granting authority can be held responsible for them.  Y.B. also thanks the ``Institut des Hautes Études Scientifiques" (IHES) for a long research stay from 7th of January to 21st of March 2025, where part of this work was written. We also thank the anonymous referee for useful remarks that helped to improve the clarity of the paper.
}

\section{The flow approach in a nutshell}\label{flownut}

We try to give a quick review of the flow method developed in \cite{Duc22}, and we will follow quite closely the introduction of \cite{CF24a}. We want to solve on $[0,1]\times\T^d$ a SPDE of the form 
\begin{equation}\label{maineq}
\begin{split}
L\psi=S[\psi],\quad
&\psi(0,\cdot)=\psi^0,
\end{split}
\end{equation}
with $L$ is a parabolic fractional operator of the form $\D_t-\Delta$ and $S$ has the general form
\begin{equation}
S[\psi] = 	g(x,\psi,\nabla\psi,\dots,\nabla^q\psi)+f(\psi,\nabla\psi,\dots,\nabla^p\psi)\xi.
\end{equation}
In the case of $S$ being a gKPZ-type functional, one has 
\begin{equation}
S[\psi]=b(\psi)+\sum_{i=1}^d d_i(\psi)\D_i\psi+\sum_{i,j=1}^dg_{ij}(\psi)\D_i\psi\D_j\psi+h(\psi)\xi,
\end{equation}
where $\xi$ is a mean-zero Gaussian noise with covariance given by
\begin{equation}
\E[\xi(t,x)\xi(0,0)]=\delta(t)(1-\Delta)^{1-d/2-\alpha}(x),
\end{equation}
where $\delta$ is the Dirac distribution on $\R$, and $\alpha\in(0,1]$ for $d\geq2$ or $\alpha\in(1/4,1]$ if $d=1$. With this choice of covariance, the noise is in the Hölder-Besov space $\mcC^{-2+\alpha-\kappa}(\R\times\T^d)$, for every $\kappa>0$.
\begin{remark}
We made the choice to use the same setting as in \cite{CF24a} for simplicity. Otherwise, one could have chosen to work with a fractional Laplacian and a white noise, as done in the original works by Duch \cite{Duc21, Duc22}, but it involves a number of technical difficulties. While these issues mainly arise in the proofs of analytic estimates and do not play a significant role in the present work which focuses on algebraic aspects, they are nonetheless relevant for the overall well-posedness of the framework.
\end{remark}
 We need to take into account the initial condition. For this purpose, we let, by overloading the notation, $\psi^0(t,x)=\delta(t)\psi^0(x)$. We consider a mollification of the noise $\xi_\eps=\rho_\eps *\xi$, for an $\eps>0$, and write a mollified version of our main equation \eqref{maineq}
\begin{equation}
\psi_\eps=G*(\one^0_{(0,\infty)}S_\eps[\psi_\eps]+\psi^0_\eps),
\end{equation}
where $G$ is the Green kernel associated with our differential operator $L$. We have also set $\psi_\eps^0=\rho_\eps*\psi^0$. Note that, in this paper, we will not consider the case $\eps\to0$. Indeed, our focus is on the algebraic part of the method and we want then to avoid analytical complications. We have denoted, for any $f:[0,1]\times\T^d\to\R$, $\one_{(0,\infty)}^0f(t,x)=\one_{(0,\infty)}(t)f(t,x)$. Introducing $\phi_\eps=\psi_\eps-G*\psi^0_\eps$, we can write
\begin{equation}
\phi_\eps=G*F_\eps[\phi_\eps],
\end{equation}
where we have used the functional $F_\eps[\cdot]=\one_{(0,\infty)}^0(S_\eps[\cdot+G*\psi^0])$. The main idea is then to introduce a flow of scales. For this, we let, for $\mu\in[0,1]$, 
\begin{equs}
G_\mu(t,x) =  \chi_{\mu}(t) G(t,x)
 \end{equs} 
where $\chi_{\mu}(t) = \chi(t/\mu^2)$ with the convention $ \chi_{0}(t) = 1 $. Here, we have supposed that the time $t$ belongs to $[0,1]$. The map $\chi : \R_+$ is a fixed smooth and non-decreasing function  such that $\chi |_{[0,1]} = 0$ and $\chi |_{[2,+\infty)} = 1$. In particular, $G_0=G$. We also denote $\dot G_\mu=\D_\mu G_\mu$. 
One has the following key properties on the support of $G_{\mu}$
\begin{equs}\label{cutscale}
	\mathrm{Supp }\, G_{\mu} \subset [\mu^2,1] \times \T^d, \quad \mathrm{Supp } \, \dot{G}_{\mu} \subset [\mu^2,2\mu^2] \times \T^d.
\end{equs}
This means $ \dot{G}_{\mu} $ is at much smaller scale than $  G_{\mu}$ and increasing the degree, or homogeneity, in $\mu$ of an iterated integral in comparison to $G_{\mu}$.
We can now define the so-called coarse-grained process
\begin{equation} \label{grained_process}
\pem=G_\mu*F_\eps[\phi_\eps].
\end{equation}
We further let
\begin{equation} \label{scale_decomposition}  F_\eps[\phi_\eps] = \Fem[\pem]+\Rem,
\end{equation}
where $\Fem$ is to be chosen wisely later with the so-called ansatz, and $\Rem$ is a remainder. We impose $F_{\eps,0}=F_\eps$. Setting such a remainder is not mandatory, but it will let more freedom in our analysis. One can notice that the original paper \cite{Duc22} does not make this choice. 

\begin{proposition}
	The maps $ \varphi_{\eps,\mu} $ and $R_{\eps,\mu}$ are solutions of \eqref{grained_process} and \eqref{scale_decomposition},
if they solve
\begin{equation}
	\displaystyle\begin{cases}
		\pem=-\displaystyle\int_\mu^1\dot{G}_\nu*(F_{\eps,\nu}[\phi_{\eps,\nu}]+R_{\eps,\nu})d\nu\\
		R_{\eps,\mu}=\displaystyle-\int_0^\mu(\D_\nu+DF_{\eps,\nu}\cdot\dot G_\nu)F_{\eps,\nu}[\phi_{\eps,\nu}]+DF_{\eps,\nu}\cdot\dot G_\nu* R_{\eps,\nu}d\nu.
	\end{cases}
\end{equation}
\end{proposition}

\begin{proof}
Indeed, one has
\begin{equs}
	\pem &=G_\mu*(F_{\eps}[\varphi_{\eps}]) =  (G_1 - G_\mu)*(F_{\eps}[\varphi_{\eps}])
	\\ &  = - \int_\mu^1\dot{G}_\nu*(F_{\eps}[\varphi_{\eps}])d\nu
	 = - \int_\mu^1\dot{G}_\nu*(F_{\eps,\nu}[\phi_{\eps,\nu}]+R_{\eps,\nu})d\nu,
\end{equs}
where we have used \eqref{scale_decomposition} for the last line. The second equation is obtained by deriving in $\mu$ the equation \eqref{scale_decomposition}, noting that $\D_\mu F_\eps[\phi_\eps]=0$,
\begin{equs}
\partial_{\mu}\Rem & = \partial_{\mu} \left(  	 F_\eps[\phi_\eps] - \Fem[\pem] \right)
\\ & =  - \partial_{\mu} \Fem[\pem] -  D \Fem[\pem]  \cdot \dot{\varphi}_{\eps,\mu}
\\ & =  - \partial_{\mu} \Fem[\pem] -  D \Fem[\pem]  \cdot \dot{G_\mu}*F_\eps[\phi_\eps]
\\ &  = - \partial_{\mu} \Fem[\pem] -  D \Fem[\pem]  \cdot \dot{G_\mu}*\Fem  -  D \Fem[\pem]  \cdot \dot{G_\mu}*\Rem
\end{equs}
where we have used \eqref{grained_process} for the third line and \eqref{scale_decomposition} for the fourth line.

\end{proof}
Note that, in particular, making the choice $\Rem=0$ leads to the flow, or Polchinski, equation
\begin{equation}\label{Poleq}
\D_\mu\Fem[\pem]+D\Fem[\pem]\cdot\dot G_\mu*\Fem[\pem]=0.
\end{equation}
In this work, we make an ansatz of the form 
\begin{equation}\label{ansatz}
\Fem^\gamma[\phi]=\sum_{\substack{\tau\in T_0^*\\ \deg(\tau)\leq\gamma}}\frac{1}{S(\tau)}(\eval^R\tau)[\phi],
\end{equation}
where $T_0^*$ is a combinatorial set of trees (see Section \ref{decorated _trees} for a precise definition of $ T_0^*$), the $S(\tau)$'s the symmetric factors associated to these trees, and $\eval^R$ a map of evaluation of these trees, defined in Section \ref{models}, taking into account the renormalisation of these terms. The parameter $\gamma>0$ is a fixed real number, that depends only on the regularity of the noise and of the non-linearity. We will not make precise statements for its choice in this work. The only fact to keep in mind is that the truncation parameter $\gamma$ keeps the terms that actually need to be renormalised. The attentive reader will notice that we therefore do not choose to have $\Rem=0$. We provide more information about the truncation in Section \eqref{coherence}. In the previous works using the flow approach \cite{Duc22, Duc23, CF24a, CF24b}, an ansatz is made and when plugged into the flow equation, gives a recursive definition of the coefficients appearing in the ansatz, that is also hierarchical in the order of the coefficients,  allowing to define them recursively. Let us already mention that our approach differs here. We will define these coefficients explicitly in order to highlight how the renormalisation is compatible with the flow equation, and compare it with the renormalisation scheme in regularity structures.

\begin{remark}
We have imposed that the equation is subcritical. Therefore the sum in the Ansatz \eqref{ansatz} contains a finite number of terms after truncation.
\end{remark}

\begin{remark}
We will not explain in much detail how the machinery of the flow approach works on the analytical side to get a meaning of solution but we can still exhibit the main steps. In order to construct a solution, Duch first writes a modified fixed point problem for any $\mu\in(0,1]$ using another kernel $K_\mu$ that we will not make explicit here, but which, as a kernel operator, regularizes by suppressing the spatial irregularities at scales below $\mu$. The key to show the contraction is to verify that the renormalised stochastic terms $\eval^R$ satisfy a proper stochastic estimate when tested against the kernels $K_\mu$, and mesured with a suitable norm. This is similar to the analytic conditions of a model in regularity structures. We will not make it explicit there as the construction we chose is rather adapted to the algebraic framework. Instead, we refer to \cite{CF24a}, Section 2.3. In a few words, one replaces the elementary differentials in Definition \ref{eval} by convolutions with the kernels $K_\mu$, and takes a $L^p$-type norm of the latter quantity. This precise choice of norm allowed Duch to prove these estimates inductively via the flow equation \eqref{Poleq}, which makes the whole interest of the flow approach. We refer the reader to Proposition \ref{flowcoeff} to see how this equation reads on the coefficients with our notations. The stochastic estimate is obtained thanks to a similar expression on the cumulants of the coefficients. More precisely, one defines the joint cumulant of the force terms by
\begin{equation}
\mcK_{\mu,\eps}(\tau_1,\dots,\tau_n)[\phi]=\E_c\left[(\eval^R\tau_1)[\phi],\dots,(\eval^R\tau_n)[\phi]\right],
\end{equation}
where $\tau_1,\dots,\tau_n$ are trees in $T_0^*$, and $\E_c$ is the joint cumulant of random variables. These cumulants satisfy themselves an enhanced version of the flow equation on coefficients of Proposition $\ref{flowcoeff}$. The terms $\D_\mu\mcK_{\mu,\eps}(\tau_1,\dots,\tau_n)[\phi]$ can be expressed as a linear combination of a linear operator of $\mcK_{\mu,\eps}(\tau_1,\dots,\tau_n)[\phi]$, as well as a bilinear operator of $\mcK_{\mu,\eps}(\tau_I)[\phi]$ and $\mcK_{\mu,\eps}(\tau_J)[\phi]$ with $I\sqcup J=\{1,\dots,n\}$. Moreover, when integrating this quantity in $\mu$, the initial condition  $\mcK_{0,\eps}(\tau_1,\dots,\tau_n)[\phi]$ brings in inductively the renormalisation. Note that the choice of norm is made so that only the cumulants of one tree have a non-vanishing initial condition. See equation (8.6) in \cite{Duc21} for a precise statement. This induction given by the flow allows to prove a bound on these cumulants (see for instance Lemma 4.15 in \cite{CF24a}) which, thanks to a Kolmogorov-type argument, gives in return the desired almost sure estimate on the force coefficients (see \cite{CF24a} Section 4.3). The parameters $\mu$ and $\eps$ are then successively taken to $0$ and $\pem$ converges in a Hölder-Besov space of negative regularity.

\end{remark}

\section{Algebraic tools on decorated trees}\label{trees}

\subsection{Decorated trees}
\label{decorated _trees}
We want to define a framework of decorated trees to write the involved distributions and formalize the algebraic operations on them. We define a set of decorations for the nodes $\mcD_\Labn=\N^{d+1}\times\{0,1\}$ and a set of decorations for the edges $\mcD_\Labe\coloneqq\{\mcI,\mcI',\CIg,\Xi,\Xig\}\times\N^{d+1}$. We will generally denote an element of $\mcD_\Labn$ as $(X^k,\mfv)$. We will call the symbols denoted with a calligraphic $\mcI$ of integration type and the symbols denoted with a $\Xi$ of noise type. We will denote $\mcI_a$, $\mcI'_a$ and $\CIg_a$ the symbols for edges with decoration $a\in\N^{d+1}$. The first node decoration in $\N^{d+1}$ has to be really interpreted as a polynomial. These will appear when performing Taylor expansions in the $d+1$ space-time variables on the analytical side. We refer the reader to \eqref{Mloc} for the algebraic counterpart.

We interpret these decorated trees in two different ways. First, they will encode iterated stochastic integrals. In this context, the abstract integration maps $\mcI_a$ and $\mcI_a'$ for $a \in \N^{d+1}$ will represent a convolution with respect to $\partial^a (G- G_\mu)$ or its derivative in $\mu$. Following this interpretation, we will overload the notation of type $\mcI_a$ to also denote the operator that grafts a tree onto a new root with edge of decoration $a$. The symbol $\Xi$ will represent the driving noise.
 
These trees will also be used as parameters in the elementary differentials, or the coefficients of the ansatz. The integration-type edge will represent, at the analytical level, derivatives in variables representing the solution and its derivatives. The grey symbols mean that these edges are used only for the elementary differentials but not for the iterated integrals. This is needed as the ansatz for the Polchinski flow is not local, on the contrary of what happens in regularity structures. These grey symbols could be connected to the extended decorations defined in \cite{BHZ}. We also mention that one of the central places where these grey decorations appear is with the coproduct introduced in Definition \ref{deltar}, where it is used as a memory of the edges that were grafted on a given extracted subtree.
  
 The  binary parameter $\mfv$ will be used as a marker at the level of the stochastic iterated integrals viewed as operators with one place-holder corresponding to this mark. Composition of operators is viewed at the combinatorial level via a grafting procedure: one connects a decorated tree to this marked decorated tree by putting an edge between its root and the marked node. In simpler words, the decoration $\mfv$ authorizes grafting on a given node if it is set to $1$, and forbids it if it is set to $0$. See the end of Section \ref{seccoeff} for further explanation. Let us give a precise definition of the involved combinatorial objects.

\begin{definition}[Decorated tree]
	A decorated tree over $(\mcD_\Labn,\mcD_\Labe)$ is a $3$-tuple of the form  $\tau_{\Labe}^{\Labn} =  (\tau,\Labn,\Labe)$ where $\tau$ is a non-planar rooted tree with node set $N_\tau$ and edge set $E_\tau$. The maps $\Labn : N_{\tau} \rightarrow \mcD_\Labn$ and $\Labe : E_\tau{\tiny } \rightarrow \mcD_\Labe$ are node and edge decorations, respectively. 
\end{definition}
We define a binary tree product by 
\begin{equation}  \label{treeproduct}
 	(\tau,\Labn,\Labe) \cdot  (\tau',\Labn',\Labe') 
 	= (\tau \cdot \tau',\Labn + \Labn', \Labe + \Labe')\;, 
\end{equation} 
where $\tau \cdot \tau'$ is the rooted tree obtained by identifying the roots of $ \tau$ and $\tau'$. The sums $ \Labn + \Labn'$ mean that decorations are added at the root and extended to the disjoint union by setting them to vanish on the other tree. Each edge and vertex of both trees keeps its decoration, except the roots which merge into a new root decorated by the sum of the previous two decorations. This operation is simply a way to formalize the natural operations used when trees represent distributions. \\

We already set $\mcI(Y^\mfv)=\mcI'(Y^\mfv)=0$, as well as $\mcI(\tau)=\mcI'(\tau)=0$ for a tree composed only of grey symbols or if there is at least one symbol $\Xig$.\\

\textit{Sets of trees}. We use $T$ to denote the set of decorated trees and $ \mathcal{T} $ is its linear span. We further impose on trees of $T$ that on any node, there cannot be more than one edge of noise type. We impose that the grey integration symbols $\CIg$ can only be final edges with potentially a decoration $X^k$ at the top but having $\mfv$ set to $0$. We also denote:
\begin{itemize}
\item  $T_0$ the set of trees containing only the symbols $\mcI$ and $\Xi$ and with no polynomial decoration on the nodes, \textit{i.e.} set to $X^0=\one$.
\item $T_0^*$, but we set the decoration $\mfv$ to be $1$ uniformly, which means that one can graft everywhere on a tree in $T_0^*$
\item $T_0^{\mfv=0}$ the same set as $T_0$ but with $\mfv=0$ everywhere.
\item $T_1$ the same way as $T_0$, but allowing  at most one edge with the symbol $\mcI'$.
\item  $T_0^X$ as $T_0$, but allowing polynomial decorations.
\item  $T_0^{X,*}$ as $T_0^*$, but allowing polynomial decorations.
\item  $T_1^X$ as $T_1$, but allowing polynomial decorations.
\item $T_0^\mfg$ as $T_0$, but allowing the decorations $\Xig$ and $\CIg$.
\item $T_0^{X,\mfg}$ as $T_0$, but allowing both polynomial and grey decorations. 
\item $T_0^0$ as $T_0$, but allowing only the decoration of the edges to be set to $0$.
\item $T_0^{X,-,*}$ the trees of $T_0^{X,*}$ of negative degree (see the definition of the degree below).
\end{itemize}   We further denote with a calligraphic $\mcT$ the span of any of the sets above, with the right set of subscript and superscripts. We also denote $\zeta_0=\one$ and $\zeta_1=\Xi$ and use generally the notation $\zeta_l$ with $l\in\{0,1\}$. One can then notice that a tree $\tau\in T_0^*$ can be written as
\begin{equation}
\tau=\zeta_l\prod_{i=1}^n\mcI_{a_i}(\tau_i).
\end{equation}
where the product $\prod_i$ has to be interpreted as a tree product.
We will write more general trees in $T$ as
\begin{equation}
\tau=\zeta_lX^kY^\mfv\prod_{i=1}^n\mcI_{a_i}(\tau_i)\prod_{j=1}^m\mcI'_{b_j}(\tau'_j)\prod_{q=1}^p\CIg_{c_q}(X^{k_q}).
\end{equation}
Note that we have introduced the notation $Y^\mfv$ that means that the second decoration at the root is $\mfv\in\{0,1\}$.

Let us mention that the decorated trees used in regularity structures \cite{reg,BHZ} do not use the grey objects for their symbolic notation and one does not have a derivative in $\mu$ which means no symbol $\CI'_a$. They are of the form
\begin{equation*}
	\tau=\zeta_lX^k\prod_{i=1}^n\mcI_{a_i}(\tau_i).
\end{equation*}

\begin{example}
Here is an example of a decorated tree in $T$.
\begin{equation*}
	 X^{\beta}Y^1 \Xi \mathcal{I}_{b}(X^{\gamma}Y^0{\CIg}_c(\one)\Xi) 
	 \mathcal{I}'_{a}(X^{\alpha}Y^0\Xi)
	 = 	\begin{tikzpicture}[scale=0.2,baseline=0.1cm]
			\node at (0,0)  [dot,fill=blue,label= {[label distance=-0.2em]below: \scriptsize  $  \beta    $} ] (root) {};
			\node at (0,5)  [dot,label= {[label distance=-0.2em]right: \scriptsize  $  \gamma   $} ] (center) {};
			\node at (0,9)  [dot,label= {[label distance=-0.2em]above: \scriptsize  $    $} ] (centerc) {};
			\node at (3,4)  [dot,label={[label distance=-0.2em]right: \scriptsize  $ \alpha $}] (right) {};
			\node at (-3,4)  [dot,label={[label distance=-0.2em]above: \scriptsize  $ $} ] (left) {};
			\node at (3,8)  [dot,label={[label distance=-0.2em]above: \scriptsize  $ $} ] (rightc) {};
			\node at (-3,8)  [dot,label={[label distance=-0.2em]above: \scriptsize  $ $} ] (leftc) {};
			\draw[kernel1,color=red] (right) to
			node [sloped,below] {\small }     (root); 
			\draw[kernel1] (center) to
			node [sloped,below] {\small }     (root);
			\draw[kernel1] (center) to
			node [sloped,below] {\small }     (centerc);
			\draw[kernel1] (right) to
			node [sloped,below] {\small }     (rightc); 
			\draw[kernel1] (left) to
			node [sloped,below] {\small }     (root);
			\draw[kernel1,color=gray] (leftc) to
			node [sloped,below] {\small }     (center);
			\node at (3,6) [fill=white,label={[label distance=0em]center: \scriptsize  $ \Xi $} ] () {};
			\node at (0,7) [fill=white,label={[label distance=0em]center: \scriptsize  $ \Xi $} ] () {};
			\node at (-2,2) [fill=white,label={[label distance=0em]center: \scriptsize  $ \Xi $} ] () {};
			\node at (2,2) [fill=white,label={[label distance=0em]center: \scriptsize   $ a  $} ] () {};
			\node at (0,2.5) [fill=white,label={[label distance=0em]center: \scriptsize  $ b  $} ] () {};
			\node at (-2,6.5) [fill=white,label={[label distance=0em]center: \scriptsize  $ c $} ] () {};
		\end{tikzpicture}
	\end{equation*}
	Let us briefly explain the graphical notations which are used in the figure above. A black node corresponds to $Y^0$ and a blue node to $Y^1$. A black edge decorated by $\Xi$ is an edge decorated by the noise. The black edges decorated by $a \in \N^{d+1}$ (roman letters) are edges with the decoration $ \mathcal{I}_a $. A red edge decorated by $a$ is for an edge decorated  by $\mathcal{I}_a'$. Node decorations in Greek letters correspond to monomials of the type $X^{\alpha}$.
\end{example}

\textit{Degree}. Before defining the degree, we fix a vector $\mfs=(2,1,\dots,1)\in\N^{d+1}$ corresponding to a parabolic-type scaling and we define for $q\in\R^{d+1}$, $|q|_\mfs=\sum_{i=1}^{d+1}\mfs_i q_i$. We then let inductively
\begin{align*}
&\deg(\Xi)=-2+\alpha-\kappa & &\deg(X^k)=|k|_\mfs\\
&\deg(\tau\sigma)=\deg(\tau)+\deg(\sigma)  & &\deg(\mcI_a(\tau))=\deg(\tau)+2-|a|_\mfs\\
&\deg(\mcI'_a(\tau))=\deg(\mcI_a(\tau))-1 & &\deg(\CIg_a(X^k))=0.
\end{align*}
$\kappa$ is a small positive integer. The degree map coincides with the degree map defined in \cite[Sec. 8.1]{reg} for the decorated trees that appear in regularity structures.Consequently, we define $T_0^{X,-,*}$ the subset of $T_0^{X,*}$ of trees of non-positive degree.
In general, the degree of a decorated tree has to be understood as the scaling exponent in $\mu$ of the stochastic terms, appropriately tested against a specific scaling kernel. That is why the grey symbols do not play any role in its computation as they play a role only when the considered tree is used as a parameter of an elementary differential.
The degree associated with $\mcI_a$ corresponds to a gain of regularity, provided by a Schauder estimate with respect to the convolution with the kernel $ \partial^{a}(G-G_\mu) $.
In the flow approach, this will also correspond to the behaviour in the parameter $\mu$ of $\dot G_\mu$, which is the reason why one has to subtract $1$ from the degree of  $ \mathcal{I}_a $  to get the degree of  $ \mathcal{I}'_a $.

Let us introduce some classical algebraic quantities on trees, which will not depend on the decoration $\mfv$, so that we will forget the symbol $Y^\mfv$ in this paragraph for clarity. We define the symmetry factor $S(\tau)$ of a decorated tree $\tau=X^k\zeta_l\mcI'_{a_0}(\tau_0)^{\beta_0}\prod_{i=1}^n\mcI_{a_i}(\tau_i)^{\beta_i}$ in $T_1^X$, with $(a_i,\tau_i)\neq(a_j,\tau_j)$ for $i\neq j$, and $\beta_0\in\{0,1\}$, inductively by 
\begin{equation}
S(\tau)=k!\prod_{i=0}^n S(\tau_i)^{\beta_i}\beta_i!,
\end{equation}
and with $S(\zeta_l)=1$ to initialize. We also define an inner product on $T_1^X$ for two trees $\sigma,\tau$ by
\begin{equation}
\scal{\sigma,\tau}=S(\tau)\one_{\sigma=\tau},
\end{equation}
and extended by bilinearity to $\mcT_1^X$. We also set
\begin{equation}
\scal{\sigma_1\otimes\sigma_2,\tau_1\otimes\tau_2}=\scal{\sigma_1,\tau_1}\scal{\sigma_2,\tau_2}.
\end{equation}
The inner product and the symmetry factor are of the same flavour as in \cite[Sec. 2.7]{BCCH}, with the difference that more decorations are considered.

\par{} \textit{Grafting operations}. We choose to define a grafting operation $\rgraft_a$ recursively. Analytically, this will correspond to a composition of two iterated integrals one being convolved with the kernel $ \partial^a  \dot{G}_{\mu}$ encoded by the decorated edge type $\mathcal{I}_a'$ which is associated with the grafting $\rgraft_a$. Note that we choose to define these operations in an inductive way. This choice is motivated by the fact that the vast majority of our proofs use inductive techniques, thus making these definitions easier to manipulate. We set, on a tree in $T_0^\mfg$ of the form $\zeta_lY^\mfv\prod_{i=1}^n\mcI_{a_i}(\tau_i)\prod_{j=1}^m\CIg_{b_j}(X^{k_j})$, and $\sigma\in T$,
\begin{multline}
\sigma\rgraft_a\left(\zeta_lY^1\prod_{i=1}^n\mcI_{a_i}(\tau_i)\prod_{j=1}^m\CIg_{b_j}(X^{k_j})\right)=\mcI'_a(\sigma)\zeta_lY^0\prod_{i=1}^n\mcI_{a_i}(\tau_i)\prod_{j=1}^m\CIg_{b_j}(X^{k_j})\\+ \zeta_lY^1 \sum_{i=1}^n\mcI_{a_i}(\sigma\rgraft_a\tau_i)\prod_{k\neq i}\mcI_{a_k}(\tau_k)\prod_{j=1}^m\CIg_{b_j}(X^{k_j}),
\end{multline}
and
\begin{multline}
\sigma  \rgraft_a\left(\zeta_lY^0\prod_{i=1}^n\mcI_{a_i}(\tau_i)\prod_{j=1}^m\CIg_{b_j}(X^{k_j})\right) \\  = \zeta_lY^0\sum_{i=1}^n\mcI_{a_i}(\sigma\rgraft_a\tau_i)\prod_{k\neq i}\mcI_{a_k}(\tau_k)\prod_{j=1}^m\CIg_{b_j}(X^{k_j}).
\end{multline}
This latter definition makes the use of the parameter $\mfv\in\{0,1\}$ as an indicator of the nodes we can graft on clear. The rest is given to the recursion. It will later be used to write the flow equation on coefficients in an abstract way.  In the case where the parameter $\mfv$ is zero at the root, the graft does not happen there, the rest of the definition is given to the induction on the superior nodes. Note that the graft does not touch the grey edges, since they are only artefacts allowing to keep tract of additional derivatives in the elementary differentials. In many proofs of Section \ref{mainresults}, $\mfv$ will be set to $1$ on every node, and thus will not play any role. We subsequently introduce the operation $\rgraft_a^{\text{\tiny{root}}}$ where the grafting only occurs at the root of the tree, if it is of course allowed. We also set $\rgraft_a^{\text{\tiny{non-root}}}=\rgraft_a-\rgraft_a^{\text{\tiny{root}}}$. We consequently define the following insertion product $\graft:\mcI'(T_0)\times T_0\to T_1$ by
\begin{equation}\label{insert}
\mcI'_a(\sigma)\graft\tau=\sigma\rgraft_a\tau,
\end{equation}
and vanishing on other sets.
\begin{example}
Let us give a simple example of the operation we just introduced.
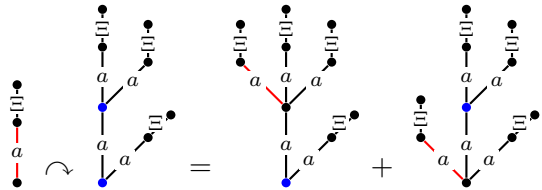
\begin{equation*}
\begin{tikzpicture}[scale=0.2,baseline=0.1cm]
			\node at (0,0)  [dot,label= {[label distance=-0.2em]below: \scriptsize  $      $} ] (root) {};
			\node at (0,4)  [dot,label= {[label distance=-0.2em]right: \scriptsize  $     $} ] (center) {};
			\node at (0,6.5)  [dot,label= {[label distance=-0.2em]right: \scriptsize  $     $} ] (centerc) {};
			\draw[kernel1,color=red] (center) to
			node [sloped,below] {\small }     (root);
			\draw[kernel1] (centerc) to
			node [sloped,below] {\small }     (center);
			\node at (0,2) [fill=white,label={[label distance=0em]center: \scriptsize  $ a  $} ] () {};
			\node at (0,5.25) [fill=white,label={[label distance=0em]center: \scriptsize  $ \Xi  $} ] () {};
		\end{tikzpicture}\graft\begin{tikzpicture}[scale=0.2,baseline=0.1cm]
			\node at (0,0)  [dot,fill=blue,label= {[label distance=-0.2em]below: \scriptsize  $      $} ] (root) {};
			\node at (0,5)  [dot,fill=blue,label= {[label distance=-0.2em]right: \scriptsize  $     $} ] (center) {};
			\node at (0,9)  [dot,label= {[label distance=-0.2em]above: \scriptsize  $    $} ] (centerc) {};
			\node at (3,3)  [dot,label={[label distance=-0.2em]right: \scriptsize  $  $}] (right) {};
			\node at (4.5,4.5)  [dot,label={[label distance=-0.2em]above: \scriptsize  $ $} ] (rightc) {};
			\node at (0,11.5)  [dot,label={[label distance=-0.2em]above: \scriptsize  $ $} ] (top) {};
			\node at (3,10.5)  [dot,label={[label distance=-0.2em]above: \scriptsize  $ $} ] (top2) {};
			\node at (3,8)  [dot,label={[label distance=-0.2em]above: \scriptsize  $ $} ] (mid1) {};
			\draw[kernel1] (right) to
			node [sloped,below] {\small }     (root); 
			\draw[kernel1] (center) to
			node [sloped,below] {\small }     (root);
			\draw[kernel1] (center) to
			node [sloped,below] {\small }     (centerc);
			\draw[kernel1] (right) to
			node [sloped,below] {\small }     (rightc); 
			\draw[kernel1] (top2) to
			node [sloped,below] {\small }     (mid1);
			\draw[kernel1] (top) to
			node [sloped,below] {\small }     (centerc);
			\draw[kernel1] (mid1) to
			node [sloped,below] {\small }     (center);
			\node at (3.5,4) [fill=white,label={[label distance=0em]center: \scriptsize  $ \Xi $} ] () {};
			\node at (0,7) [fill=white,label={[label distance=0em]center: \scriptsize  $ a $} ] () {};
			\node at (1.5,1.5) [fill=white,label={[label distance=0em]center: \scriptsize   $ a  $} ] () {};
			\node at (0,2.5) [fill=white,label={[label distance=0em]center: \scriptsize  $ a  $} ] () {};
			\node at (2,6.75) [fill=white,label={[label distance=0em]center: \scriptsize  $ a $} ] () {};
			\node at (0,10.25) [fill=white,label={[label distance=0em]center: \scriptsize  $ \Xi $} ] () {};
			\node at (3,9.25) [fill=white,label={[label distance=0em]center: \scriptsize  $ \Xi $} ] () {};
		\end{tikzpicture}=\begin{tikzpicture}[scale=0.2,baseline=0.1cm]
			\node at (0,0)  [dot,fill=blue,label= {[label distance=-0.2em]below: \scriptsize  $      $} ] (root) {};
			\node at (0,5)  [dot,label= {[label distance=-0.2em]right: \scriptsize  $     $} ] (center) {};
			\node at (0,9)  [dot,label= {[label distance=-0.2em]above: \scriptsize  $    $} ] (centerc) {};
			\node at (3,3)  [dot,label={[label distance=-0.2em]right: \scriptsize  $  $}] (right) {};
			\node at (4.5,4.5)  [dot,label={[label distance=-0.2em]above: \scriptsize  $ $} ] (rightc) {};
			\node at (0,11.5)  [dot,label={[label distance=-0.2em]above: \scriptsize  $ $} ] (top) {};
			\node at (-3,10.5)  [dot,label={[label distance=-0.2em]above: \scriptsize  $ $} ] (top1) {};
			\node at (3,10.5)  [dot,label={[label distance=-0.2em]above: \scriptsize  $ $} ] (top2) {};
			\node at (3,8)  [dot,label={[label distance=-0.2em]above: \scriptsize  $ $} ] (mid1) {};
			\node at (-3,8)  [dot,label={[label distance=-0.2em]above: \scriptsize  $ $} ] (mid2) {};
			\draw[kernel1] (right) to
			node [sloped,below] {\small }     (root); 
			\draw[kernel1] (center) to
			node [sloped,below] {\small }     (root);
			\draw[kernel1] (center) to
			node [sloped,below] {\small }     (centerc);
			\draw[kernel1] (right) to
			node [sloped,below] {\small }     (rightc); 
			\draw[kernel1] (top1) to
			node [sloped,below] {\small }     (mid2);
			\draw[kernel1] (top2) to
			node [sloped,below] {\small }     (mid1);
			\draw[kernel1] (top) to
			node [sloped,below] {\small }     (centerc);
			\draw[kernel1] (mid1) to
			node [sloped,below] {\small }     (center);
			\draw[kernel1,color=red] (mid2) to
			node [sloped,below] {\small }     (center);
			\node at (3.5,4) [fill=white,label={[label distance=0em]center: \scriptsize  $ \Xi $} ] () {};
			\node at (0,7) [fill=white,label={[label distance=0em]center: \scriptsize  $ a $} ] () {};
			\node at (1.5,1.5) [fill=white,label={[label distance=0em]center: \scriptsize   $ a  $} ] () {};
			\node at (0,2.5) [fill=white,label={[label distance=0em]center: \scriptsize  $ a  $} ] () {};
			\node at (1.5,1.5) [fill=white,label={[label distance=0em]center: \scriptsize   $ a  $} ] () {};
			\node at (-2,6.75) [fill=white,label={[label distance=0em]center: \scriptsize  $ a $} ] () {};
			\node at (2,6.75) [fill=white,label={[label distance=0em]center: \scriptsize  $ a $} ] () {};
			\node at (0,10.25) [fill=white,label={[label distance=0em]center: \scriptsize  $ \Xi $} ] () {};
			\node at (-3,9.25) [fill=white,label={[label distance=0em]center: \scriptsize  $ \Xi $} ] () {};
			\node at (3,9.25) [fill=white,label={[label distance=0em]center: \scriptsize  $ \Xi $} ] () {};
		\end{tikzpicture}+\begin{tikzpicture}[scale=0.2,baseline=0.1cm]
			\node at (0,0)  [dot,label= {[label distance=-0.2em]below: \scriptsize  $      $} ] (root) {};
			\node at (0,5)  [dot,fill=blue,label= {[label distance=-0.2em]right: \scriptsize  $     $} ] (center) {};
			\node at (0,9)  [dot,label= {[label distance=-0.2em]above: \scriptsize  $    $} ] (centerc) {};
			\node at (3,3)  [dot,label={[label distance=-0.2em]right: \scriptsize  $  $}] (right) {};
			\node at (-3,3)  [dot,label={[label distance=-0.2em]above: \scriptsize  $ $} ] (left) {};
			\node at (4.5,4.5)  [dot,label={[label distance=-0.2em]above: \scriptsize  $ $} ] (rightc) {};
			\node at (-3,5.5)  [dot,label={[label distance=-0.2em]above: \scriptsize  $ $} ] (leftc) {};
			\node at (0,11.5)  [dot,label={[label distance=-0.2em]above: \scriptsize  $ $} ] (top) {};
			\node at (3,10.5)  [dot,label={[label distance=-0.2em]above: \scriptsize  $ $} ] (top2) {};
			\node at (3,8)  [dot,label={[label distance=-0.2em]above: \scriptsize  $ $} ] (mid1) {};
			\draw[kernel1] (right) to
			node [sloped,below] {\small }     (root); 
			\draw[kernel1] (center) to
			node [sloped,below] {\small }     (root);
			\draw[kernel1] (center) to
			node [sloped,below] {\small }     (centerc);
			\draw[kernel1] (right) to
			node [sloped,below] {\small }     (rightc); 
			\draw[kernel1,color=red] (left) to
			node [sloped,below] {\small }     (root);
			\draw[kernel1] (leftc) to
			node [sloped,below] {\small }     (left);
			\draw[kernel1] (top2) to
			node [sloped,below] {\small }     (mid1);
			\draw[kernel1] (top) to
			node [sloped,below] {\small }     (centerc);
			\draw[kernel1] (mid1) to
			node [sloped,below] {\small }     (center);
			\node at (3.5,4) [fill=white,label={[label distance=0em]center: \scriptsize  $ \Xi $} ] () {};
			\node at (0,7) [fill=white,label={[label distance=0em]center: \scriptsize  $ a $} ] () {};
			\node at (-1.5,1.5) [fill=white,label={[label distance=0em]center: \scriptsize  $ a $} ] () {};
			\node at (1.5,1.5) [fill=white,label={[label distance=0em]center: \scriptsize   $ a  $} ] () {};
			\node at (0,2.5) [fill=white,label={[label distance=0em]center: \scriptsize  $ a  $} ] () {};
			\node at (-3,4.25) [fill=white,label={[label distance=0em]center: \scriptsize  $ \Xi $} ] () {};
			\node at (2,6.75) [fill=white,label={[label distance=0em]center: \scriptsize  $ a $} ] () {};
			\node at (0,10.25) [fill=white,label={[label distance=0em]center: \scriptsize  $ \Xi $} ] () {};
			\node at (3,9.25) [fill=white,label={[label distance=0em]center: \scriptsize  $ \Xi $} ] () {};
		\end{tikzpicture}
\end{equation*}
We recall that the blue nodes are the ones for which the parameter $\mfv$ is set to one, thus meaning that we can only graft on these. In this example, it is forbidden to graft on the black nodes, whereas it would be the case with the classical definition found in the literature. Note that it is also always forbidden to graft on the noise-type edges.
\end{example}

Before continuing, we need to introduce two grafting-type operations $\ggraft_a$ and $\ggraft_{a,k}$ on $T_0^X$ and $T_0$ respectively, defined inductively in the following way, and extended by linearity to $\mcT_0^X$ and $\mcT_0$.
\begin{multline}
\ggraft_{a,k}\left(\zeta_lY^\mfv\prod_{i=1}^n\mcI_{a_i}(\tau_i)\right)=\mfv\CIg_{a}(X^k)\zeta_lY^\mfv\prod_{i=1}^n\mcI_{a_i}(\tau_i)\\ +\sum_{i=1}^n\zeta_lY^\mfv\mcI_{a_i}(\ggraft_{a,k}\tau_i)\prod_{j\neq i}\mcI_{a_j}(\tau_j),
\end{multline}
and
\begin{multline}
\ggraft_{a}\left(\zeta_lX^kY^\mfv\prod_{i=1}^n\mcI_{a_i}(\tau_i)\right)=\left(\mfv\sum_{j\in\N^{d+1}}\binom{k}{j}X^{k-j}\CIg_{a-j}(X^j)\right)\times\\ \zeta_lY^\mfv\prod_{i=1}^n\mcI_{a_i}(\tau_i)
+\sum_{i=1}^n\zeta_lX^kY^\mfv\mcI_{a_i}(\ggraft_{a}\tau_i)\prod_{j\neq i}\mcI_{a_j}(\tau_j).
\end{multline}
Note that $\ggraft_a$ and $\ggraft_{a,k}$ are \textit{not} binary operations. As the parameter $\mfv$ is in $\{0,1\}$, placing it at the front of the expressions allows us to use it as an indicator function, so that, when $\mfv=0$, the graft at the root does not occur. $\ggraft_{a,k}$ will be used in interaction with the coproduct $\Delta_r$ defined in Definition \eqref{deltar}. The second operation $\ggraft_a$ will be used in interaction with the elementary differentials in Lemma \eqref{lemmaupsilon} and in the proof of Proposition \ref{graft} as the counterpart of an analytical operation, namely a Fréchet-type derivative. Note its similarity with the deformed grafts intensively used in \cite{BB21b} in regularity structures, thus making it a natural object.

\begin{example}
We give two concrete examples of how these two operations act.
\begin{equs}
\ggraft_{a,k}\begin{tikzpicture}[scale=0.2,baseline=0.1cm]
			\node at (0,0)  [dot,color=blue,label= {[label distance=-0.2em]below: \scriptsize  $      $} ] (root) {};
			\node at (2,3)  [dot,color=blue,label={[label distance=-0.2em]right: \scriptsize  $  $}] (right) {};
			\node at (-2,3)  [dot,label={[label distance=-0.2em]above: \scriptsize  $ $} ] (left) {};
			\node at (2,5.5)  [dot,label={[label distance=-0.2em]above: \scriptsize  $ $} ] (rightc) {};
			\node at (-2,5.5)  [dot,label={[label distance=-0.2em]above: \scriptsize  $ $} ] (leftc) {};
			\draw[kernel1] (right) to
			node [sloped,below] {\small }     (root); 
			\draw[kernel1] (right) to
			node [sloped,below] {\small }     (rightc); 
			\draw[kernel1] (left) to
			node [sloped,below] {\small }     (root);
			\draw[kernel1] (leftc) to
			node [sloped,below] {\small }     (left);
			\node at (2,4.25) [fill=white,label={[label distance=0em]center: \scriptsize  $ \Xi $} ] () {};
			\node at (-1.25,1.5) [fill=white,label={[label distance=0em]center: \scriptsize  $ b $} ] () {};
			\node at (1.25,1.5) [fill=white,label={[label distance=0em]center: \scriptsize   $ c  $} ] () {};
			\node at (-2,4.25) [fill=white,label={[label distance=0em]center: \scriptsize  $ \Xi $} ] () {};
		\end{tikzpicture}=\begin{tikzpicture}[scale=0.2,baseline=0.1cm]
			\node at (0,0)  [dot,color=blue,label= {[label distance=-0.2em]below: \scriptsize  $      $} ] (root) {};
			\node at (2.5,3)  [dot,color=blue,label={[label distance=-0.2em]right: \scriptsize  $  $}] (right) {};
			\node at (-2.5,3)  [dot,label={[label distance=-0.2em]above: \scriptsize  $ $} ] (left) {};
			\node at (0,3.5) [dot,label={[label distance=-0.2em]above: \scriptsize  $ $} ] (middle) {};
			\node at (2.5,5.5)  [dot,label={[label distance=-0.2em]above: \scriptsize  $ $} ] (rightc) {};
			\node at (-2.5,5.5)  [dot,label={[label distance=-0.2em]above: \scriptsize  $ $} ] (leftc) {};
			\draw[kernel1] (right) to
			node [sloped,below] {\small }     (root); 
			\draw[kernel1] (right) to
			node [sloped,below] {\small }     (rightc); 
			\draw[kernel1] (left) to
			node [sloped,below] {\small }     (root);
			\draw[kernel1] (leftc) to
			node [sloped,below] {\small }     (left);
			\draw[kernel1, color=gray] (root) to
			node [sloped,below] {\small }     (middle);
			\node at (2.5,4.25) [fill=white,label={[label distance=0em]center: \scriptsize  $ \Xi $} ] () {};
			\node at (-1.5,1.5) [fill=white,label={[label distance=0em]center: \scriptsize  $ b $} ] () {};
			\node at (1.5,1.5) [fill=white,label={[label distance=0em]center: \scriptsize   $ c  $} ] () {};
			\node at (0,1.75) [fill=white,label={[label distance=0em]center: \scriptsize   $ a  $} ] () {};
			\node at (-2.5,4.25) [fill=white,label={[label distance=0em]center: \scriptsize  $ \Xi $} ] () {};
			\node at (1,3.5) [fill=white,label={[label distance=0em]center: \scriptsize   $ k  $} ] () {};
		\end{tikzpicture}+\begin{tikzpicture}[scale=0.2,baseline=0.1cm]
			\node at (0,0)  [dot,color=blue,label= {[label distance=-0.2em]below: \scriptsize  $      $} ] (root) {};
			\node at (2,3)  [dot,color=blue,label={[label distance=-0.2em]right: \scriptsize  $  $}] (right) {};
			\node at (-2,3)  [dot,label={[label distance=-0.2em]above: \scriptsize  $ $} ] (left) {};
			\node at (0,6) [dot,label={[label distance=-0.2em]above: \scriptsize  $ $} ] (middle) {};
			\node at (3.5,5.5)  [dot,label={[label distance=-0.2em]above: \scriptsize  $ $} ] (rightc) {};
			\node at (-2,5.5)  [dot,label={[label distance=-0.2em]above: \scriptsize  $ $} ] (leftc) {};
			\draw[kernel1] (right) to
			node [sloped,below] {\small }     (root); 
			\draw[kernel1] (right) to
			node [sloped,below] {\small }     (rightc); 
			\draw[kernel1] (left) to
			node [sloped,below] {\small }     (root);
			\draw[kernel1] (leftc) to
			node [sloped,below] {\small }     (left);
			\draw[kernel1, color=gray] (right) to
			node [sloped,below] {\small }     (middle);
			\node at (2.75,4.25) [fill=white,label={[label distance=0em]center: \scriptsize  $ \Xi $} ] () {};
			\node at (-1.25,1.5) [fill=white,label={[label distance=0em]center: \scriptsize  $ b $} ] () {};
			\node at (1.25,1.5) [fill=white,label={[label distance=0em]center: \scriptsize   $ c  $} ] () {};
			\node at (-2,4.25) [fill=white,label={[label distance=0em]center: \scriptsize  $ \Xi $} ] () {};
			\node at (1,6.5) [fill=white,label={[label distance=0em]center: \scriptsize   $ k  $} ] () {};
			\node at (1,4.5) [fill=white,label={[label distance=0em]center: \scriptsize   $ a  $} ] () {};
		\end{tikzpicture},
\end{equs}
and
\begin{equs}
\ggraft_{a}\begin{tikzpicture}[scale=0.2,baseline=0.1cm]
			\node at (0,0)  [dot,color=blue,label= {[label distance=-0.2em]below: \scriptsize  $      $} ] (root) {};
			\node at (2,3)  [dot,color=blue,label={[label distance=-0.2em]right: \scriptsize  $  $}] (right) {};
			\node at (-2,3)  [dot,label={[label distance=-0.2em]above: \scriptsize  $ $} ] (left) {};
			\node at (2,5.5)  [dot,label={[label distance=-0.2em]above: \scriptsize  $ $} ] (rightc) {};
			\node at (-2,5.5)  [dot,label={[label distance=-0.2em]above: \scriptsize  $ $} ] (leftc) {};
			\draw[kernel1] (right) to
			node [sloped,below] {\small }     (root); 
			\draw[kernel1] (right) to
			node [sloped,below] {\small }     (rightc); 
			\draw[kernel1] (left) to
			node [sloped,below] {\small }     (root);
			\draw[kernel1] (leftc) to
			node [sloped,below] {\small }     (left);
			\node at (2,4.25) [fill=white,label={[label distance=0em]center: \scriptsize  $ \Xi $} ] () {};
			\node at (-1.25,1.5) [fill=white,label={[label distance=0em]center: \scriptsize  $ b $} ] () {};
			\node at (1.25,1.5) [fill=white,label={[label distance=0em]center: \scriptsize   $ c  $} ] () {};
			\node at (-2,4.25) [fill=white,label={[label distance=0em]center: \scriptsize  $ \Xi $} ] () {};
			\node at (1.5,0) [fill=white,label={[label distance=0em]center: \scriptsize  $ k_1  $} ] () {};
			\node at (3.5,3) [fill=white,label={[label distance=0em]center: \scriptsize  $ k_2  $} ] () {};
		\end{tikzpicture}=\sum_{j_1\in\N^{d+1}}\binom{k_1}{j_1}\begin{tikzpicture}[scale=0.2,baseline=0.1cm]
			\node at (0,0)  [dot,color=blue,label= {[label distance=-0.2em]below: \scriptsize  $      $} ] (root) {};
			\node at (2.5,3)  [dot,label={[label distance=-0.2em]right: \scriptsize  $  $}] (right) {};
			\node at (-2.5,3)  [dot,label={[label distance=-0.2em]above: \scriptsize  $ $} ] (left) {};
			\node at (0,3.5) [dot,color=blue,label={[label distance=-0.2em]above: \scriptsize  $ $} ] (middle) {};
			\node at (0,6)  [dot,label={[label distance=-0.2em]above: \scriptsize  $ $} ] (rightc) {};
			\node at (-2.5,5.5)  [dot,label={[label distance=-0.2em]above: \scriptsize  $ $} ] (leftc) {};
			\draw[kernel1, color=gray] (right) to
			node [sloped,below] {\small }     (root); 
			\draw[kernel1] (middle) to
			node [sloped,below] {\small }     (rightc); 
			\draw[kernel1] (left) to
			node [sloped,below] {\small }     (root);
			\draw[kernel1] (leftc) to
			node [sloped,below] {\small }     (left);
			\draw[kernel1] (root) to
			node [sloped,below] {\small }     (middle);
			\node at (0,4.75) [fill=white,label={[label distance=0em]center: \scriptsize  $ \Xi $} ] () {};
			\node at (-1.5,1.5) [fill=white,label={[label distance=0em]center: \scriptsize  $ b $} ] () {};
			\node at (0,1.75) [fill=white,label={[label distance=0em]center: \scriptsize   $ c  $} ] () {};
			\node at (1.25,1.5) [fill=white,label={[label distance=0em]center: \scriptsize   $   $} ] () {};
			\node at (3,1.5) [fill=white,label={[label distance=0em]center: \scriptsize   $ a-j_1  $} ] () {};
			\node at (-2.5,4.25) [fill=white,label={[label distance=0em]center: \scriptsize  $ \Xi $} ] () {};
			\node at (1.25,3.5) [fill=white,label={[label distance=0em]center: \scriptsize   $ k_2  $} ] () {};
			\node at (3,0) [fill=white,label={[label distance=0em]center: \scriptsize  $ k_1-j_1  $} ] () {};
			\node at (3.75,3) [fill=white,label={[label distance=0em]center: \scriptsize  $ j_1  $} ] () {};
		\end{tikzpicture}+\sum_{j_2\in\N^{d+1}}\binom{k_2}{j_2}\begin{tikzpicture}[scale=0.2,baseline=0.1cm]
			\node at (0,0)  [dot,color=blue,label= {[label distance=-0.2em]below: \scriptsize  $      $} ] (root) {};
			\node at (2,3)  [dot,color=blue,label={[label distance=-0.2em]right: \scriptsize  $  $}] (right) {};
			\node at (-2,3)  [dot,label={[label distance=-0.2em]above: \scriptsize  $ $} ] (left) {};
			\node at (4,6) [dot,label={[label distance=-0.2em]above: \scriptsize  $ $} ] (middle) {};
			\node at (0.5,5.5)  [dot,label={[label distance=-0.2em]above: \scriptsize  $ $} ] (rightc) {};
			\node at (-2,5.5)  [dot,label={[label distance=-0.2em]above: \scriptsize  $ $} ] (leftc) {};
			\draw[kernel1] (right) to
			node [sloped,below] {\small }     (root); 
			\draw[kernel1] (right) to
			node [sloped,below] {\small }     (rightc); 
			\draw[kernel1] (left) to
			node [sloped,below] {\small }     (root);
			\draw[kernel1] (leftc) to
			node [sloped,below] {\small }     (left);
			\draw[kernel1, color=gray] (right) to
			node [sloped,below] {\small }     (middle);
			\node at (1.25,4.25) [fill=white,label={[label distance=0em]center: \scriptsize  $ \Xi $} ] () {};
			\node at (-1.25,1.5) [fill=white,label={[label distance=0em]center: \scriptsize  $ b $} ] () {};
			\node at (1.25,1.5) [fill=white,label={[label distance=0em]center: \scriptsize   $ c  $} ] () {};
			\node at (-2,4.25) [fill=white,label={[label distance=0em]center: \scriptsize  $ \Xi $} ] () {};
			\node at (5.25,6) [fill=white,label={[label distance=0em]center: \scriptsize   $ j_2  $} ] () {};
			\node at (3,4.5) [fill=white,label={[label distance=0em]center: \scriptsize   $   $} ] () {};
			\node at (4.25,4.5) [fill=white,label={[label distance=0em]center: \scriptsize   $ a-j_2  $} ] () {};
			\node at (1.5,0) [fill=white,label={[label distance=0em]center: \scriptsize  $ k_1  $} ] () {};
			\node at (5,3) [fill=white,label={[label distance=0em]center: \scriptsize  $ k_2-j_2  $} ] () {};
		\end{tikzpicture}
\end{equs}
\end{example}

\textit{Localisation map}. Eventually, we consider the following map $ \Mloc :  \mcT_0\to\mcT_0^X $ given recursively  by
\begin{equation}\label{Mloc}
	\Mloc \zeta_lY^\mfv\prod_{i=1}^n\mcI_{a_i}(\tau_i) =  \sum_{k \in \N^{d+1}} \frac{X^k}{k!} \zeta_lY^\mfv\prod_{i=1}^n\mcI_{a_i}( \Mloc \tau_i).
\end{equation}
This map, that will play an important role throughout the paper, can be viewed as an abstract way to perform Taylor expansions on the nodes of an input tree. This latter expansion will be later key in what we call the localisation step of the procedure. This inductive definition applies this idea to the root of a tree and lets the rest of it to the inductive part in such a way that $\Mloc$ is multiplicative. 
$ \Mloc $ produces an infinite sum but, in practice, they are finite as the polynomial decorations cannot be of arbitrarily large degree. One can make also sense of this infinite sum by using a bigrading introduced in \cite[Sec. 2.3]{BHZ} where and the first component is a sum of the size of all node and edge decorations and the second component is the size of a tree counting the number of edges and nodes. 

Let us give an example to make this intuition crystal clear.

\begin{example}\label{exMloc}
\begin{equation*}
\Mloc\begin{tikzpicture}[scale=0.2,baseline=0.1cm]
			\node at (0,0)  [dot,label= {[label distance=-0.2em]below: \scriptsize  $      $} ] (root) {};
			\node at (2,3)  [dot,label={[label distance=-0.2em]right: \scriptsize  $  $}] (right) {};
			\node at (-2,3)  [dot,label={[label distance=-0.2em]above: \scriptsize  $ $} ] (left) {};
			\node at (2,5.5)  [dot,label={[label distance=-0.2em]above: \scriptsize  $ $} ] (rightc) {};
			\node at (-2,5.5)  [dot,label={[label distance=-0.2em]above: \scriptsize  $ $} ] (leftc) {};
			\draw[kernel1] (right) to
			node [sloped,below] {\small }     (root); 
			\draw[kernel1] (right) to
			node [sloped,below] {\small }     (rightc); 
			\draw[kernel1] (left) to
			node [sloped,below] {\small }     (root);
			\draw[kernel1] (leftc) to
			node [sloped,below] {\small }     (left);
			\node at (2,4.25) [fill=white,label={[label distance=0em]center: \scriptsize  $ \Xi $} ] () {};
			\node at (-1.25,1.5) [fill=white,label={[label distance=0em]center: \scriptsize  $ a $} ] () {};
			\node at (1.25,1.5) [fill=white,label={[label distance=0em]center: \scriptsize   $ a  $} ] () {};
			\node at (-2,4.25) [fill=white,label={[label distance=0em]center: \scriptsize  $ \Xi $} ] () {};
		\end{tikzpicture}=\sum_{i,j,k\in\N^{d+1}}\frac{1}{i!j!k!}\begin{tikzpicture}[scale=0.2,baseline=0.1cm]
			\node at (0,0)  [dot,label= {[label distance=-0.2em]below: \scriptsize  $   k   $} ] (root) {};
			\node at (2,3)  [dot,label={[label distance=-0.2em]right: \scriptsize  $j$}] (right) {};
			\node at (-2,3)  [dot,label={[label distance=-0.2em]left: \scriptsize  $i$} ] (left) {};
			\node at (2,5.5)  [dot,label={[label distance=-0.2em]above: \scriptsize  $ $} ] (rightc) {};
			\node at (-2,5.5)  [dot,label={[label distance=-0.2em]above: \scriptsize  $  $} ] (leftc) {};
			\draw[kernel1] (right) to
			node [sloped,below] {\small }     (root); 
			\draw[kernel1] (right) to
			node [sloped,below] {\small }     (rightc); 
			\draw[kernel1] (left) to
			node [sloped,below] {\small }     (root);
			\draw[kernel1] (leftc) to
			node [sloped,below] {\small }     (left);
			\node at (2,4.25) [fill=white,label={[label distance=0em]center: \scriptsize  $ \Xi $} ] () {};
			\node at (-1.25,1.5) [fill=white,label={[label distance=0em]center: \scriptsize  $ a $} ] () {};
			\node at (1.25,1.5) [fill=white,label={[label distance=0em]center: \scriptsize   $ a  $} ] () {};
			\node at (-2,4.25) [fill=white,label={[label distance=0em]center: \scriptsize  $ \Xi $} ] () {};
		\end{tikzpicture}
\end{equation*}
We recall that the labels in Roman letters $i,j$, and $k$ represent the polynomial decorations.
\end{example}

We prove the following lemma, that is, in some way, a commutation property between the map $\Mloc$ and the two grafting operations we defined above. It will be crucial in the proof of Proposition \ref{graft}, combined with Lemma \ref{lemmaupsilon}.

\begin{lemma}\label{commutggraft}
The following holds on $\mcT_0$.
\begin{equation}
\ggraft_a\Mloc=\Mloc\sum_{k\in\N^{d+1}}\frac{1}{k!}\ggraft_{a-k,k}.
\end{equation}
\end{lemma}

\begin{proof}
For conciseness, we write the proof in the case where $\mfv=1$ at the root, the proof in the other case being the same but easier. We then voluntarily forget the symbol $Y^1$.
\begin{equs}
\ggraft_a\Mloc\left(\zeta_l\prod_{i=1}^n\mcI_{a_i}(\tau_i)\right)&=\sum_{k\in\N^{d+1}}\frac{1}{k!}\sum_{j\in\N^{d+1}}\binom{k}{j}X^{k-j}\CIg_{a-j}(X^j)\zeta_l\prod_{i=1}^n\mcI_{a_i}(\Mloc\tau_i)\\
&~~~~+\sum_{i=1}^n\zeta_l\frac{X^k}{k!}\mcI_{a_i}(\ggraft_a\Mloc\tau_i)\prod_{m\neq i}\mcI_{a_m}(\Mloc\tau_m)\\
&=\sum_{k\in\N^{d+1}}\frac{X^k}{k!}\sum_{j\in\N^{d+1}}\frac{1}{j!}\CIg_{a-j}(X^j)\zeta_l\prod_{i=1}^n\mcI_{a_i}(\Mloc\tau_i)\\
&~~~~+\sum_{i=1}^n\zeta_l\frac{X^k}{k!}\sum_{j\in\N^{d+1}}\frac{1}{j!}\mcI_{a_i}(\Mloc\ggraft_{a-j,j}\tau_i)\prod_{m\neq i}\mcI_{a_m}(\Mloc\tau_m)\\
&=\Mloc\Biggl(\sum_{j\in\N^{d+1}}\frac{1}{j!}\CIg_{a-j}(X^j)\zeta_l\prod_{i=1}^n\mcI_{a_i}(\tau_i)\\
&~~~~+\sum_{i=1}^n\zeta_l\sum_{j\in\N^{d+1}}\frac{1}{j!}\mcI_{a_i}(\ggraft_{a-j,j}\tau_i)\prod_{m\neq i}\mcI_{a_m}(\tau_m)\Biggl)\\
&=\left(\Mloc\sum_{j\in\N^{d+1}}\frac{1}{j!}\ggraft_{a-j,j}\right)\left(\zeta_l\prod_{i=1}^n\mcI_{a_i}(\tau_i)\right).
\end{equs}
\end{proof}

\textit{Dual coproduct of $\rgraft_a$}. We define a family of coproducts $(\Delta_a)_{a\in\N^{d+1}}$ that satisfies a duality property that will be key in Section \eqref{coherence}.

\begin{definition}
For a given $a\in\N^{d+1}$, we define a coproduct $\Delta_a:\mcT_1\to\mcT_0\otimes\mcT_0$ inductively by
\begin{equation}
\begin{split}
&\Delta_a Y^0=Y^0\otimes Y^0,~~~~\Delta_a\Xi Y^0=\Xi Y^0\otimes Y^1,\\
&\Delta_a\mcI'_b(\tau)=(\mathrm{id}\otimes\mcI_b)\Delta_a\tau+\one_{a=b}\tau\otimes Y^1,\\
&\Delta_a\mcI_b(\tau)=(\mathrm{id}\otimes\mcI_b)\Delta_a\tau,\\
&\Delta_a(\sigma\tau)=\Delta_a\sigma\times(\one\otimes\tau)+(\one\otimes\sigma)\times\Delta_a\tau.
\end{split}
\end{equation}
\end{definition}

\begin{example}
Let $a\neq b$, we have
\begin{align*}
\Delta_a\begin{tikzpicture}[scale=0.2,baseline=0.1cm]
			\node at (0,0)  [dot,label= {[label distance=-0.2em]below: \scriptsize  $      $} ] (root) {};
			\node at (2,3)  [dot,label={[label distance=-0.2em]right: \scriptsize  $  $}] (right) {};
			\node at (-2,3)  [dot,label={[label distance=-0.2em]above: \scriptsize  $ $} ] (left) {};
			\node at (2,5.5)  [dot,label={[label distance=-0.2em]above: \scriptsize  $ $} ] (rightc) {};
			\node at (-2,5.5)  [dot,label={[label distance=-0.2em]above: \scriptsize  $ $} ] (leftc) {};
			\draw[kernel1] (right) to
			node [sloped,below] {\small }     (root); 
			\draw[kernel1] (right) to
			node [sloped,below] {\small }     (rightc); 
			\draw[kernel1,color=red] (left) to
			node [sloped,below] {\small }     (root);
			\draw[kernel1] (leftc) to
			node [sloped,below] {\small }     (left);
			\node at (2,4.25) [fill=white,label={[label distance=0em]center: \scriptsize  $ \Xi $} ] () {};
			\node at (-1.25,1.5) [fill=white,label={[label distance=0em]center: \scriptsize  $ a $} ] () {};
			\node at (1.25,1.5) [fill=white,label={[label distance=0em]center: \scriptsize   $ b  $} ] () {};
			\node at (-2,4.25) [fill=white,label={[label distance=0em]center: \scriptsize  $ \Xi $} ] () {};
		\end{tikzpicture}=\Xi\otimes\begin{tikzpicture}[scale=0.2,baseline=0.1cm]
			\node at (0,0)  [dot,fill=blue,label= {[label distance=-0.2em]below: \scriptsize  $      $} ] (root) {};
			\node at (0,4)  [dot,label= {[label distance=-0.2em]right: \scriptsize  $     $} ] (center) {};
			\node at (0,6.5)  [dot,label= {[label distance=-0.2em]right: \scriptsize  $     $} ] (centerc) {};
			\draw[kernel1] (center) to
			node [sloped,below] {\small }     (root);
			\draw[kernel1] (centerc) to
			node [sloped,below] {\small }     (center);
			\node at (0,2) [fill=white,label={[label distance=0em]center: \scriptsize  $ b  $} ] () {};
			\node at (0,5.25) [fill=white,label={[label distance=0em]center: \scriptsize  $ \Xi  $} ] () {};
		\end{tikzpicture}~~\mathrm{and}~~\Delta_a\begin{tikzpicture}[scale=0.2,baseline=0.1cm]
			\node at (0,0)  [dot,label= {[label distance=-0.2em]below: \scriptsize  $      $} ] (root) {};
			\node at (2,3)  [dot,label={[label distance=-0.2em]right: \scriptsize  $  $}] (right) {};
			\node at (-2,3)  [dot,label={[label distance=-0.2em]above: \scriptsize  $ $} ] (left) {};
			\node at (2,5.5)  [dot,label={[label distance=-0.2em]above: \scriptsize  $ $} ] (rightc) {};
			\node at (-2,5.5)  [dot,label={[label distance=-0.2em]above: \scriptsize  $ $} ] (leftc) {};
			\draw[kernel1,color=red] (right) to
			node [sloped,below] {\small }     (root); 
			\draw[kernel1] (right) to
			node [sloped,below] {\small }     (rightc); 
			\draw[kernel1] (left) to
			node [sloped,below] {\small }     (root);
			\draw[kernel1] (leftc) to
			node [sloped,below] {\small }     (left);
			\node at (2,4.25) [fill=white,label={[label distance=0em]center: \scriptsize  $ \Xi $} ] () {};
			\node at (-1.25,1.5) [fill=white,label={[label distance=0em]center: \scriptsize  $ a $} ] () {};
			\node at (1.25,1.5) [fill=white,label={[label distance=0em]center: \scriptsize   $ b  $} ] () {};
			\node at (-2,4.25) [fill=white,label={[label distance=0em]center: \scriptsize  $ \Xi $} ] () {};
		\end{tikzpicture}=0.
\end{align*}
These two equations make quite clear the use of the indicator function in the definition above. An edge is cut by $\Delta_a$ if and only if it is red and its decoration is $a$. When an edge is taken, the node on which it was attached it turned blue, meaning that it will later be possible to graft on it.
\end{example}

From this definition follows the easy lemma, noting that, on $T_1$, $\Delta_a$ is a coproduct that acts by cutting trees at the edges of type $\mcI'$, and the trees of $T_1$ contain at most one edge of this type.

\begin{lemma}\label{duality}
We have the duality relation, for $\tau,\sigma\in\mcT_0$, $\mu\in \mcT_1$,
\begin{equation}
\langle\tau\rgraft_a\sigma,\mu\rangle=\langle\tau\otimes\sigma,\Delta_a\mu\rangle.
\end{equation}
\end{lemma}

\subsection{The abstract flow equation}\label{abstractflow}

With these definitions at hand, we are ready to show how the flow equation on the coefficients can be represented at the level of trees. We start with the abstract counterpart of the derivative $\D_\mu$. We choose again to give an inductive definition. Note that this definition is designed so that the abstract derivative will satisfy the same algebraic properties as the classical derivation. 

\begin{definition}[Abstract derivative]
We define a map $\dmu:\mcT_0\to\mcT_1$ by setting $\dmu\zeta_l=\zeta_l$ and inductively
\begin{equation}
\dmu\left(\zeta_lY^\mfv\prod_{i=1}^n\mcI_{a_i}(\tau_i)\right)=\zeta_lY^\mfv\sum_{i=1}^n\big(\mcI'_{a_i}(\tau_i)+\mcI_{a_i}(\dmu\tau_i)\big)\prod_{j\neq i}\mcI_{a_j}(\tau_j).
\end{equation}
We have not indicated explicitly the decoration $\mfv$ as $\dmu$ does not depend on it. 
\end{definition}

This definition being recursive and not so intuitive, we present a non-inductive version that will give a better intuition at many steps throughout the paper, even though the proofs will mostly be based on the inductive definition.

\begin{proposition}
It is easy to see that $\dmu$ also satisfies the non-inductive relation
\begin{equation}
\dmu\tau=\sum_{e\in E_\tau}\uparrow_\mu^e\tau,
\end{equation}
where the symbol $\uparrow_\mu^e$ changes the symbol of the edge $e$ from $\mcI$ to $\mcI'$.
\end{proposition}

\begin{proof}
We proceed by induction on the depth. The proposition is trivial on elementary trees since $\dmu$ does not act on them. We have for a planted tree $\mcI_a(\tau)$, denoting $e_0$ the edge corresponding to $\mcI_a$,
\begin{align*}
\dmu\mcI_a(\tau)&=\mcI'_a(\tau)+\mcI_a(\dmu\tau)=\dmu^{e_0}\mcI_a(\tau)+\sum_{e\in E_\tau}\mcI(\uparrow_\mu^e\tau)=\sum_{e\in E_{\mcI_a(\tau)}}\uparrow_\mu^e\mcI_a(\tau)
\end{align*}
For the product of two trees $\sigma$ and $\tau$, it is easy to see that $\dmu$ satisfies a Leibniz rule. It yields
\begin{align*}
\dmu(\sigma\tau)&=\dmu\sigma\times\tau+\sigma\times\dmu\tau=\sum_{e\in E_\sigma}\uparrow_\mu^e\sigma\times\tau+\sum_{e\in E_\tau}\sigma\times\uparrow_\mu^e\tau\\
&=\sum_{e\in E_\sigma}\uparrow_\mu^e(\sigma\tau)+\sum_{e\in E_\tau}\uparrow_\mu^e(\sigma\tau)=\sum_{e\in E_{\sigma\tau}}\dmu^e(\sigma\tau),
\end{align*}
which concludes the proof.
\end{proof}

\begin{example}\label{exdmu}
Let us give an example of how $\dmu$ acts for clarity.
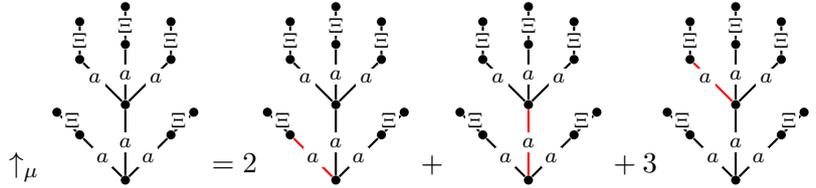
\begin{equation*}
	 \dmu 	\begin{tikzpicture}[scale=0.2,baseline=0.1cm]
			\node at (0,0)  [dot,label= {[label distance=-0.2em]below: \scriptsize  $      $} ] (root) {};
			\node at (0,5)  [dot,label= {[label distance=-0.2em]right: \scriptsize  $     $} ] (center) {};
			\node at (0,9)  [dot,label= {[label distance=-0.2em]above: \scriptsize  $    $} ] (centerc) {};
			\node at (3,3)  [dot,label={[label distance=-0.2em]right: \scriptsize  $  $}] (right) {};
			\node at (-3,3)  [dot,label={[label distance=-0.2em]above: \scriptsize  $ $} ] (left) {};
			\node at (4.5,4.5)  [dot,label={[label distance=-0.2em]above: \scriptsize  $ $} ] (rightc) {};
			\node at (-4.5,4.5)  [dot,label={[label distance=-0.2em]above: \scriptsize  $ $} ] (leftc) {};
			\node at (0,11.5)  [dot,label={[label distance=-0.2em]above: \scriptsize  $ $} ] (top) {};
			\node at (-3,10.5)  [dot,label={[label distance=-0.2em]above: \scriptsize  $ $} ] (top1) {};
			\node at (3,10.5)  [dot,label={[label distance=-0.2em]above: \scriptsize  $ $} ] (top2) {};
			\node at (3,8)  [dot,label={[label distance=-0.2em]above: \scriptsize  $ $} ] (mid1) {};
			\node at (-3,8)  [dot,label={[label distance=-0.2em]above: \scriptsize  $ $} ] (mid2) {};
			\draw[kernel1] (right) to
			node [sloped,below] {\small }     (root); 
			\draw[kernel1] (center) to
			node [sloped,below] {\small }     (root);
			\draw[kernel1] (center) to
			node [sloped,below] {\small }     (centerc);
			\draw[kernel1] (right) to
			node [sloped,below] {\small }     (rightc); 
			\draw[kernel1] (left) to
			node [sloped,below] {\small }     (root);
			\draw[kernel1] (leftc) to
			node [sloped,below] {\small }     (left);
			\draw[kernel1] (top1) to
			node [sloped,below] {\small }     (mid2);
			\draw[kernel1] (top2) to
			node [sloped,below] {\small }     (mid1);
			\draw[kernel1] (top) to
			node [sloped,below] {\small }     (centerc);
			\draw[kernel1] (mid1) to
			node [sloped,below] {\small }     (center);
			\draw[kernel1] (mid2) to
			node [sloped,below] {\small }     (center);
			\node at (3.5,4) [fill=white,label={[label distance=0em]center: \scriptsize  $ \Xi $} ] () {};
			\node at (0,7) [fill=white,label={[label distance=0em]center: \scriptsize  $ a $} ] () {};
			\node at (-1.5,1.5) [fill=white,label={[label distance=0em]center: \scriptsize  $ a $} ] () {};
			\node at (1.5,1.5) [fill=white,label={[label distance=0em]center: \scriptsize   $ a  $} ] () {};
			\node at (0,2.5) [fill=white,label={[label distance=0em]center: \scriptsize  $ a  $} ] () {};
			\node at (-3.5,4) [fill=white,label={[label distance=0em]center: \scriptsize  $ \Xi $} ] () {};
			\node at (1.5,1.5) [fill=white,label={[label distance=0em]center: \scriptsize   $ a  $} ] () {};
			\node at (-2,6.75) [fill=white,label={[label distance=0em]center: \scriptsize  $ a $} ] () {};
			\node at (2,6.75) [fill=white,label={[label distance=0em]center: \scriptsize  $ a $} ] () {};
			\node at (0,10.25) [fill=white,label={[label distance=0em]center: \scriptsize  $ \Xi $} ] () {};
			\node at (-3,9.25) [fill=white,label={[label distance=0em]center: \scriptsize  $ \Xi $} ] () {};
			\node at (3,9.25) [fill=white,label={[label distance=0em]center: \scriptsize  $ \Xi $} ] () {};
		\end{tikzpicture}
		=
		2	\begin{tikzpicture}[scale=0.2,baseline=0.1cm]
			\node at (0,0)  [dot,label= {[label distance=-0.2em]below: \scriptsize  $      $} ] (root) {};
			\node at (0,5)  [dot,label= {[label distance=-0.2em]right: \scriptsize  $     $} ] (center) {};
			\node at (0,9)  [dot,label= {[label distance=-0.2em]above: \scriptsize  $    $} ] (centerc) {};
			\node at (3,3)  [dot,label={[label distance=-0.2em]right: \scriptsize  $  $}] (right) {};
			\node at (-3,3)  [dot,label={[label distance=-0.2em]above: \scriptsize  $ $} ] (left) {};
			\node at (4.5,4.5)  [dot,label={[label distance=-0.2em]above: \scriptsize  $ $} ] (rightc) {};
			\node at (-4.5,4.5)  [dot,label={[label distance=-0.2em]above: \scriptsize  $ $} ] (leftc) {};
			\node at (0,11.5)  [dot,label={[label distance=-0.2em]above: \scriptsize  $ $} ] (top) {};
			\node at (-3,10.5)  [dot,label={[label distance=-0.2em]above: \scriptsize  $ $} ] (top1) {};
			\node at (3,10.5)  [dot,label={[label distance=-0.2em]above: \scriptsize  $ $} ] (top2) {};
			\node at (3,8)  [dot,label={[label distance=-0.2em]above: \scriptsize  $ $} ] (mid1) {};
			\node at (-3,8)  [dot,label={[label distance=-0.2em]above: \scriptsize  $ $} ] (mid2) {};
			\draw[kernel1] (right) to
			node [sloped,below] {\small }     (root); 
			\draw[kernel1] (center) to
			node [sloped,below] {\small }     (root);
			\draw[kernel1] (center) to
			node [sloped,below] {\small }     (centerc);
			\draw[kernel1] (right) to
			node [sloped,below] {\small }     (rightc); 
			\draw[kernel1,color=red] (left) to
			node [sloped,below] {\small }     (root);
			\draw[kernel1] (leftc) to
			node [sloped,below] {\small }     (left);
			\draw[kernel1] (top1) to
			node [sloped,below] {\small }     (mid2);
			\draw[kernel1] (top2) to
			node [sloped,below] {\small }     (mid1);
			\draw[kernel1] (top) to
			node [sloped,below] {\small }     (centerc);
			\draw[kernel1] (mid1) to
			node [sloped,below] {\small }     (center);
			\draw[kernel1] (mid2) to
			node [sloped,below] {\small }     (center);
			\node at (3.5,4) [fill=white,label={[label distance=0em]center: \scriptsize  $ \Xi $} ] () {};
			\node at (0,7) [fill=white,label={[label distance=0em]center: \scriptsize  $ a $} ] () {};
			\node at (-1.5,1.5) [fill=white,label={[label distance=0em]center: \scriptsize  $ a $} ] () {};
			\node at (1.5,1.5) [fill=white,label={[label distance=0em]center: \scriptsize   $ a  $} ] () {};
			\node at (0,2.5) [fill=white,label={[label distance=0em]center: \scriptsize  $ a  $} ] () {};
			\node at (-3.5,4) [fill=white,label={[label distance=0em]center: \scriptsize  $ \Xi $} ] () {};
			\node at (1.5,1.5) [fill=white,label={[label distance=0em]center: \scriptsize   $ a  $} ] () {};
			\node at (-2,6.75) [fill=white,label={[label distance=0em]center: \scriptsize  $ a $} ] () {};
			\node at (2,6.75) [fill=white,label={[label distance=0em]center: \scriptsize  $ a $} ] () {};
			\node at (0,10.25) [fill=white,label={[label distance=0em]center: \scriptsize  $ \Xi $} ] () {};
			\node at (-3,9.25) [fill=white,label={[label distance=0em]center: \scriptsize  $ \Xi $} ] () {};
			\node at (3,9.25) [fill=white,label={[label distance=0em]center: \scriptsize  $ \Xi $} ] () {};
		\end{tikzpicture}
		+	\begin{tikzpicture}[scale=0.2,baseline=0.1cm]
			\node at (0,0)  [dot,label= {[label distance=-0.2em]below: \scriptsize  $      $} ] (root) {};
			\node at (0,5)  [dot,label= {[label distance=-0.2em]right: \scriptsize  $     $} ] (center) {};
			\node at (0,9)  [dot,label= {[label distance=-0.2em]above: \scriptsize  $    $} ] (centerc) {};
			\node at (3,3)  [dot,label={[label distance=-0.2em]right: \scriptsize  $  $}] (right) {};
			\node at (-3,3)  [dot,label={[label distance=-0.2em]above: \scriptsize  $ $} ] (left) {};
			\node at (4.5,4.5)  [dot,label={[label distance=-0.2em]above: \scriptsize  $ $} ] (rightc) {};
			\node at (-4.5,4.5)  [dot,label={[label distance=-0.2em]above: \scriptsize  $ $} ] (leftc) {};
			\node at (0,11.5)  [dot,label={[label distance=-0.2em]above: \scriptsize  $ $} ] (top) {};
			\node at (-3,10.5)  [dot,label={[label distance=-0.2em]above: \scriptsize  $ $} ] (top1) {};
			\node at (3,10.5)  [dot,label={[label distance=-0.2em]above: \scriptsize  $ $} ] (top2) {};
			\node at (3,8)  [dot,label={[label distance=-0.2em]above: \scriptsize  $ $} ] (mid1) {};
			\node at (-3,8)  [dot,label={[label distance=-0.2em]above: \scriptsize  $ $} ] (mid2) {};
			\draw[kernel1] (right) to
			node [sloped,below] {\small }     (root); 
			\draw[kernel1,color=red] (center) to
			node [sloped,below] {\small }     (root);
			\draw[kernel1] (center) to
			node [sloped,below] {\small }     (centerc);
			\draw[kernel1] (right) to
			node [sloped,below] {\small }     (rightc); 
			\draw[kernel1] (left) to
			node [sloped,below] {\small }     (root);
			\draw[kernel1] (leftc) to
			node [sloped,below] {\small }     (left);
			\draw[kernel1] (top1) to
			node [sloped,below] {\small }     (mid2);
			\draw[kernel1] (top2) to
			node [sloped,below] {\small }     (mid1);
			\draw[kernel1] (top) to
			node [sloped,below] {\small }     (centerc);
			\draw[kernel1] (mid1) to
			node [sloped,below] {\small }     (center);
			\draw[kernel1] (mid2) to
			node [sloped,below] {\small }     (center);
			\node at (3.5,4) [fill=white,label={[label distance=0em]center: \scriptsize  $ \Xi $} ] () {};
			\node at (0,7) [fill=white,label={[label distance=0em]center: \scriptsize  $ a $} ] () {};
			\node at (-1.5,1.5) [fill=white,label={[label distance=0em]center: \scriptsize  $ a $} ] () {};
			\node at (1.5,1.5) [fill=white,label={[label distance=0em]center: \scriptsize   $ a  $} ] () {};
			\node at (0,2.5) [fill=white,label={[label distance=0em]center: \scriptsize  $ a  $} ] () {};
			\node at (-3.5,4) [fill=white,label={[label distance=0em]center: \scriptsize  $ \Xi $} ] () {};
			\node at (1.5,1.5) [fill=white,label={[label distance=0em]center: \scriptsize   $ a  $} ] () {};
			\node at (-2,6.75) [fill=white,label={[label distance=0em]center: \scriptsize  $ a $} ] () {};
			\node at (2,6.75) [fill=white,label={[label distance=0em]center: \scriptsize  $ a $} ] () {};
			\node at (0,10.25) [fill=white,label={[label distance=0em]center: \scriptsize  $ \Xi $} ] () {};
			\node at (-3,9.25) [fill=white,label={[label distance=0em]center: \scriptsize  $ \Xi $} ] () {};
			\node at (3,9.25) [fill=white,label={[label distance=0em]center: \scriptsize  $ \Xi $} ] () {};
		\end{tikzpicture}
		+3	\begin{tikzpicture}[scale=0.2,baseline=0.1cm]
			\node at (0,0)  [dot,label= {[label distance=-0.2em]below: \scriptsize  $      $} ] (root) {};
			\node at (0,5)  [dot,label= {[label distance=-0.2em]right: \scriptsize  $     $} ] (center) {};
			\node at (0,9)  [dot,label= {[label distance=-0.2em]above: \scriptsize  $    $} ] (centerc) {};
			\node at (3,3)  [dot,label={[label distance=-0.2em]right: \scriptsize  $  $}] (right) {};
			\node at (-3,3)  [dot,label={[label distance=-0.2em]above: \scriptsize  $ $} ] (left) {};
			\node at (4.5,4.5)  [dot,label={[label distance=-0.2em]above: \scriptsize  $ $} ] (rightc) {};
			\node at (-4.5,4.5)  [dot,label={[label distance=-0.2em]above: \scriptsize  $ $} ] (leftc) {};
			\node at (0,11.5)  [dot,label={[label distance=-0.2em]above: \scriptsize  $ $} ] (top) {};
			\node at (-3,10.5)  [dot,label={[label distance=-0.2em]above: \scriptsize  $ $} ] (top1) {};
			\node at (3,10.5)  [dot,label={[label distance=-0.2em]above: \scriptsize  $ $} ] (top2) {};
			\node at (3,8)  [dot,label={[label distance=-0.2em]above: \scriptsize  $ $} ] (mid1) {};
			\node at (-3,8)  [dot,label={[label distance=-0.2em]above: \scriptsize  $ $} ] (mid2) {};
			\draw[kernel1] (right) to
			node [sloped,below] {\small }     (root); 
			\draw[kernel1] (center) to
			node [sloped,below] {\small }     (root);
			\draw[kernel1] (center) to
			node [sloped,below] {\small }     (centerc);
			\draw[kernel1] (right) to
			node [sloped,below] {\small }     (rightc); 
			\draw[kernel1] (left) to
			node [sloped,below] {\small }     (root);
			\draw[kernel1] (leftc) to
			node [sloped,below] {\small }     (left);
			\draw[kernel1] (top1) to
			node [sloped,below] {\small }     (mid2);
			\draw[kernel1] (top2) to
			node [sloped,below] {\small }     (mid1);
			\draw[kernel1] (top) to
			node [sloped,below] {\small }     (centerc);
			\draw[kernel1] (mid1) to
			node [sloped,below] {\small }     (center);
			\draw[kernel1,color=red] (mid2) to
			node [sloped,below] {\small }     (center);
			\node at (3.5,4) [fill=white,label={[label distance=0em]center: \scriptsize  $ \Xi $} ] () {};
			\node at (0,7) [fill=white,label={[label distance=0em]center: \scriptsize  $ a $} ] () {};
			\node at (-1.5,1.5) [fill=white,label={[label distance=0em]center: \scriptsize  $ a $} ] () {};
			\node at (1.5,1.5) [fill=white,label={[label distance=0em]center: \scriptsize   $ a  $} ] () {};
			\node at (0,2.5) [fill=white,label={[label distance=0em]center: \scriptsize  $ a  $} ] () {};
			\node at (-3.5,4) [fill=white,label={[label distance=0em]center: \scriptsize  $ \Xi $} ] () {};
			\node at (1.5,1.5) [fill=white,label={[label distance=0em]center: \scriptsize   $ a  $} ] () {};
			\node at (-2,6.75) [fill=white,label={[label distance=0em]center: \scriptsize  $ a $} ] () {};
			\node at (2,6.75) [fill=white,label={[label distance=0em]center: \scriptsize  $ a $} ] () {};
			\node at (0,10.25) [fill=white,label={[label distance=0em]center: \scriptsize  $ \Xi $} ] () {};
			\node at (-3,9.25) [fill=white,label={[label distance=0em]center: \scriptsize  $ \Xi $} ] () {};
			\node at (3,9.25) [fill=white,label={[label distance=0em]center: \scriptsize  $ \Xi $} ] () {};
		\end{tikzpicture}.
\end{equation*}
We simply sum over all the edges, turning them red one by one. Since the tree chosen in the present example contains some symmetries, we regroup the terms.
\end{example}

The following lemma, that will be used in Section \eqref{coherence}, gives an easy expression of the combinatorial factors appearing in the example above with symmetry factors. This will later make it interact well with the inner product on trees.

\begin{lemma}\label{combidmu}
One can write, for $\tau\in T_0$,
\begin{equation}
\dmu\tau=\sum_i\frac{S(\tau)}{S(\tau_i)}\tau_i,
\end{equation}
where the $\tau_i$'s are pairwise distinct trees in $T_1$ of the form $\dmu^e\tau$, for $e\in E_\tau$.
\end{lemma}

\begin{proof}
As usual, we proceed by induction on the depth. Let us write
\begin{equation*}
\tau=\zeta_l\prod_{i=1}^n\mcI_{a_i}(\tau_i)^{\beta_i},
\end{equation*}
where $(a_i,\tau_i)\neq(a_j,\tau_j)$ for $i\neq j$, and $\beta_i\neq 0$. We proceed by induction, writing for every $i$ $\dmu\tau_i=\sum_k\frac{S(\tau_i)}{S(\tau_{i,k})}\tau_{i,k}$,
\begin{align*}
\dmu\tau&=\sum_{i=1}^n\beta_i\zeta_l\mcI'_{a_i}(\tau_i)\mcI_{a_i}(\tau_i)^{\beta_i-1}\prod_{j\neq i}\mcI_{a_j}(\tau_j)^{\beta_j}\\
&~~~~+\sum_{i=1}^n\beta_i\zeta_l\mcI_{a_i}(\dmu\tau_i)\mcI_{a_i}(\tau_i)^{\beta_i-1}\prod_{j\neq i}\mcI_{a_j}(\tau_j)^{\beta_j}\\
&=\sum_{i=1}^n\beta_i\zeta_l\mcI'_{a_i}(\tau_i)\mcI_{a_i}(\tau_i)^{\beta_i-1}\prod_{j\neq i}\mcI_{a_j}(\tau_j)^{\beta_j}\\
&~~~~+\sum_{i=1}^n\sum_k\frac{S(\tau_i)}{S(\tau_{i,k})}\beta_i\zeta_l\mcI_{a_i}(\tau_{i,k})\mcI_{a_i}(\tau_i)^{\beta_i-1}\prod_{j\neq i}\mcI_{a_j}(\tau_j)^{\beta_j}.
\end{align*}
To finish the proof, one can see the following equalities:
\begin{align*}
\beta_i=\frac{\prod_{j=1}^n S(\tau_j)^{\beta_j}\beta_j!}{(\beta_i-1)!S(\tau_i)^{\beta_i}\prod_{j\neq i} S(\tau_j)^{\beta_j}\beta_j!}
=\frac{S(\tau)}{S\left(\zeta_l\mcI'_{a_i}(\tau_i)\mcI_{a_i}(\tau_i)^{\beta_i-1}\prod_{j\neq i}\mcI_{a_j}(\tau_j)^{\beta_j}\right)}
\end{align*}
and
\begin{align*}
\frac{S(\tau_i)\beta_i}{S(\tau_{i,k})}&=\frac{\prod_{j=1}^n S(\tau_j)^{\beta_j}\beta_j!}{(\beta_i-1)!S(\tau_{i,k})S(\tau_i)^{\beta_i-1}\prod_{j\neq i} S(\tau_j)^{\beta_j}\beta_j!}\\
&=\frac{S(\tau)}{S\left(\zeta_l\mcI_{a_i}(\tau_{i,k})\mcI_{a_i}(\tau_i)^{\beta_i-1}\prod_{j\neq i}\mcI_{a_j}(\tau_j)^{\beta_j}\right)}.
\end{align*}
\end{proof}

\begin{definition}[Elementary cuts coproduct]
We define a coproduct $\Delta_1:\mcT_0^{\mfv=0}\to\mcT_1^{\mfv=0}\otimes\mcT_0$ inductively by
\begin{equation}
\begin{split}
\Delta_1 Y^0=Y^0\otimes Y^0,~~~~~~~~~\Delta_1\Xi Y^0=\Xi Y^0\otimes Y^1,\\
\Delta_1Y^0\mcI_a(\tau)=(\mathrm{id}\otimes\mcI_a)\Delta_1\tau+Y^0\mcI'_a(\tau)\otimes Y^1,\\
\Delta_1(\sigma\tau)=\Delta_1\sigma\times(Y^0\otimes\tau)+(Y^0\otimes\sigma)\times\Delta_1\tau.
\end{split}
\end{equation}
\end{definition}

\begin{example}\label{exdelta1}
We have
\begin{equation*}
\Delta_1\begin{tikzpicture}[scale=0.2,baseline=0.1cm]
			\node at (0,0)  [dot,label= {[label distance=-0.2em]below: \scriptsize  $      $} ] (root) {};
			\node at (0,5)  [dot,label= {[label distance=-0.2em]right: \scriptsize  $     $} ] (center) {};
			\node at (0,9)  [dot,label= {[label distance=-0.2em]above: \scriptsize  $    $} ] (centerc) {};
			\node at (3,3)  [dot,label={[label distance=-0.2em]right: \scriptsize  $  $}] (right) {};
			\node at (-3,3)  [dot,label={[label distance=-0.2em]above: \scriptsize  $ $} ] (left) {};
			\node at (4.5,4.5)  [dot,label={[label distance=-0.2em]above: \scriptsize  $ $} ] (rightc) {};
			\node at (-4.5,4.5)  [dot,label={[label distance=-0.2em]above: \scriptsize  $ $} ] (leftc) {};
			\node at (0,11.5)  [dot,label={[label distance=-0.2em]above: \scriptsize  $ $} ] (top) {};
			\node at (-3,10.5)  [dot,label={[label distance=-0.2em]above: \scriptsize  $ $} ] (top1) {};
			\node at (3,10.5)  [dot,label={[label distance=-0.2em]above: \scriptsize  $ $} ] (top2) {};
			\node at (3,8)  [dot,label={[label distance=-0.2em]above: \scriptsize  $ $} ] (mid1) {};
			\node at (-3,8)  [dot,label={[label distance=-0.2em]above: \scriptsize  $ $} ] (mid2) {};
			\draw[kernel1] (right) to
			node [sloped,below] {\small }     (root); 
			\draw[kernel1] (center) to
			node [sloped,below] {\small }     (root);
			\draw[kernel1] (center) to
			node [sloped,below] {\small }     (centerc);
			\draw[kernel1] (right) to
			node [sloped,below] {\small }     (rightc); 
			\draw[kernel1] (left) to
			node [sloped,below] {\small }     (root);
			\draw[kernel1] (leftc) to
			node [sloped,below] {\small }     (left);
			\draw[kernel1] (top1) to
			node [sloped,below] {\small }     (mid2);
			\draw[kernel1] (top2) to
			node [sloped,below] {\small }     (mid1);
			\draw[kernel1] (top) to
			node [sloped,below] {\small }     (centerc);
			\draw[kernel1] (mid1) to
			node [sloped,below] {\small }     (center);
			\draw[kernel1] (mid2) to
			node [sloped,below] {\small }     (center);
			\node at (3.5,4) [fill=white,label={[label distance=0em]center: \scriptsize  $ \Xi $} ] () {};
			\node at (0,7) [fill=white,label={[label distance=0em]center: \scriptsize  $ a $} ] () {};
			\node at (-1.5,1.5) [fill=white,label={[label distance=0em]center: \scriptsize  $ a $} ] () {};
			\node at (1.5,1.5) [fill=white,label={[label distance=0em]center: \scriptsize   $ a  $} ] () {};
			\node at (0,2.5) [fill=white,label={[label distance=0em]center: \scriptsize  $ a  $} ] () {};
			\node at (-3.5,4) [fill=white,label={[label distance=0em]center: \scriptsize  $ \Xi $} ] () {};
			\node at (-2,6.75) [fill=white,label={[label distance=0em]center: \scriptsize  $ a $} ] () {};
			\node at (2,6.75) [fill=white,label={[label distance=0em]center: \scriptsize  $ a $} ] () {};
			\node at (0,10.25) [fill=white,label={[label distance=0em]center: \scriptsize  $ \Xi $} ] () {};
			\node at (-3,9.25) [fill=white,label={[label distance=0em]center: \scriptsize  $ \Xi $} ] () {};
			\node at (3,9.25) [fill=white,label={[label distance=0em]center: \scriptsize  $ \Xi $} ] () {};
		\end{tikzpicture}=2\begin{tikzpicture}[scale=0.2,baseline=0.1cm]
			\node at (0,0)  [dot,label= {[label distance=-0.2em]below: \scriptsize  $      $} ] (root) {};
			\node at (0,4)  [dot,label= {[label distance=-0.2em]right: \scriptsize  $     $} ] (center) {};
			\node at (0,6.5)  [dot,label= {[label distance=-0.2em]right: \scriptsize  $     $} ] (centerc) {};
			\draw[kernel1,color=red] (center) to
			node [sloped,below] {\small }     (root);
			\draw[kernel1] (centerc) to
			node [sloped,below] {\small }     (center);
			\node at (0,2) [fill=white,label={[label distance=0em]center: \scriptsize  $ a  $} ] () {};
			\node at (0,5.25) [fill=white,label={[label distance=0em]center: \scriptsize  $ \Xi  $} ] () {};
		\end{tikzpicture}\otimes\begin{tikzpicture}[scale=0.2,baseline=0.1cm]
			\node at (0,0)  [dot,fill=blue,label= {[label distance=-0.2em]below: \scriptsize  $      $} ] (root) {};
			\node at (0,5)  [dot,label= {[label distance=-0.2em]right: \scriptsize  $     $} ] (center) {};
			\node at (0,9)  [dot,label= {[label distance=-0.2em]above: \scriptsize  $    $} ] (centerc) {};
			\node at (3,3)  [dot,label={[label distance=-0.2em]right: \scriptsize  $  $}] (right) {};
			\node at (4.5,4.5)  [dot,label={[label distance=-0.2em]above: \scriptsize  $ $} ] (rightc) {};
			\node at (0,11.5)  [dot,label={[label distance=-0.2em]above: \scriptsize  $ $} ] (top) {};
			\node at (-3,10.5)  [dot,label={[label distance=-0.2em]above: \scriptsize  $ $} ] (top1) {};
			\node at (3,10.5)  [dot,label={[label distance=-0.2em]above: \scriptsize  $ $} ] (top2) {};
			\node at (3,8)  [dot,label={[label distance=-0.2em]above: \scriptsize  $ $} ] (mid1) {};
			\node at (-3,8)  [dot,label={[label distance=-0.2em]above: \scriptsize  $ $} ] (mid2) {};
			\draw[kernel1] (right) to
			node [sloped,below] {\small }     (root); 
			\draw[kernel1] (center) to
			node [sloped,below] {\small }     (root);
			\draw[kernel1] (center) to
			node [sloped,below] {\small }     (centerc);
			\draw[kernel1] (right) to
			node [sloped,below] {\small }     (rightc); 
			\draw[kernel1] (top1) to
			node [sloped,below] {\small }     (mid2);
			\draw[kernel1] (top2) to
			node [sloped,below] {\small }     (mid1);
			\draw[kernel1] (top) to
			node [sloped,below] {\small }     (centerc);
			\draw[kernel1] (mid1) to
			node [sloped,below] {\small }     (center);
			\draw[kernel1] (mid2) to
			node [sloped,below] {\small }     (center);
			\node at (3.5,4) [fill=white,label={[label distance=0em]center: \scriptsize  $ \Xi $} ] () {};
			\node at (0,7) [fill=white,label={[label distance=0em]center: \scriptsize  $ a $} ] () {};
			\node at (1.5,1.5) [fill=white,label={[label distance=0em]center: \scriptsize   $ a  $} ] () {};
			\node at (0,2.5) [fill=white,label={[label distance=0em]center: \scriptsize  $ a  $} ] () {};
			\node at (1.5,1.5) [fill=white,label={[label distance=0em]center: \scriptsize   $ a  $} ] () {};
			\node at (-2,6.75) [fill=white,label={[label distance=0em]center: \scriptsize  $ a $} ] () {};
			\node at (2,6.75) [fill=white,label={[label distance=0em]center: \scriptsize  $ a $} ] () {};
			\node at (0,10.25) [fill=white,label={[label distance=0em]center: \scriptsize  $ \Xi $} ] () {};
			\node at (-3,9.25) [fill=white,label={[label distance=0em]center: \scriptsize  $ \Xi $} ] () {};
			\node at (3,9.25) [fill=white,label={[label distance=0em]center: \scriptsize  $ \Xi $} ] () {};
		\end{tikzpicture}+\begin{tikzpicture}[scale=0.2,baseline=0.1cm]
			\node at (0,0)  [dot,label= {[label distance=-0.2em]below: \scriptsize  $      $} ] (root) {};
			\node at (0,5)  [dot,label= {[label distance=-0.2em]right: \scriptsize  $     $} ] (center) {};
			\node at (0,9)  [dot,label= {[label distance=-0.2em]above: \scriptsize  $    $} ] (centerc) {};
			\node at (0,11.5)  [dot,label={[label distance=-0.2em]above: \scriptsize  $ $} ] (top) {};
			\node at (-3,10.5)  [dot,label={[label distance=-0.2em]above: \scriptsize  $ $} ] (top1) {};
			\node at (3,10.5)  [dot,label={[label distance=-0.2em]above: \scriptsize  $ $} ] (top2) {};
			\node at (3,8)  [dot,label={[label distance=-0.2em]above: \scriptsize  $ $} ] (mid1) {};
			\node at (-3,8)  [dot,label={[label distance=-0.2em]above: \scriptsize  $ $} ] (mid2) {};
			\draw[kernel1,color=red] (center) to
			node [sloped,below] {\small }     (root);
			\draw[kernel1] (center) to
			node [sloped,below] {\small }     (centerc);
			\draw[kernel1] (top1) to
			node [sloped,below] {\small }     (mid2);
			\draw[kernel1] (top2) to
			node [sloped,below] {\small }     (mid1);
			\draw[kernel1] (top) to
			node [sloped,below] {\small }     (centerc);
			\draw[kernel1] (mid1) to
			node [sloped,below] {\small }     (center);
			\draw[kernel1] (mid2) to
			node [sloped,below] {\small }     (center);
			\node at (0,7) [fill=white,label={[label distance=0em]center: \scriptsize  $ a $} ] () {};
			\node at (0,2.5) [fill=white,label={[label distance=0em]center: \scriptsize   $ a  $} ] () {};
			\node at (-2,6.75) [fill=white,label={[label distance=0em]center: \scriptsize  $ a $} ] () {};
			\node at (2,6.75) [fill=white,label={[label distance=0em]center: \scriptsize  $ a $} ] () {};
			\node at (0,10.25) [fill=white,label={[label distance=0em]center: \scriptsize  $ \Xi $} ] () {};
			\node at (-3,9.25) [fill=white,label={[label distance=0em]center: \scriptsize  $ \Xi $} ] () {};
			\node at (3,9.25) [fill=white,label={[label distance=0em]center: \scriptsize  $ \Xi $} ] () {};
		\end{tikzpicture}\otimes\begin{tikzpicture}[scale=0.2,baseline=0.1cm]
			\node at (0,0)  [dot,fill=blue,label= {[label distance=-0.2em]below: \scriptsize  $      $} ] (root) {};
			\node at (2,3)  [dot,label={[label distance=-0.2em]right: \scriptsize  $  $}] (right) {};
			\node at (-2,3)  [dot,label={[label distance=-0.2em]above: \scriptsize  $ $} ] (left) {};
			\node at (2,5.5)  [dot,label={[label distance=-0.2em]above: \scriptsize  $ $} ] (rightc) {};
			\node at (-2,5.5)  [dot,label={[label distance=-0.2em]above: \scriptsize  $ $} ] (leftc) {};
			\draw[kernel1] (right) to
			node [sloped,below] {\small }     (root); 
			\draw[kernel1] (right) to
			node [sloped,below] {\small }     (rightc); 
			\draw[kernel1] (left) to
			node [sloped,below] {\small }     (root);
			\draw[kernel1] (leftc) to
			node [sloped,below] {\small }     (left);
			\node at (2,4.25) [fill=white,label={[label distance=0em]center: \scriptsize  $ \Xi $} ] () {};
			\node at (-1.25,1.5) [fill=white,label={[label distance=0em]center: \scriptsize  $ a $} ] () {};
			\node at (1.25,1.5) [fill=white,label={[label distance=0em]center: \scriptsize   $ a  $} ] () {};
			\node at (-2,4.25) [fill=white,label={[label distance=0em]center: \scriptsize  $ \Xi $} ] () {};
		\end{tikzpicture}+3\begin{tikzpicture}[scale=0.2,baseline=0.1cm]
			\node at (0,0)  [dot,label= {[label distance=-0.2em]below: \scriptsize  $      $} ] (root) {};
			\node at (0,4)  [dot,label= {[label distance=-0.2em]right: \scriptsize  $     $} ] (center) {};
			\node at (0,6.5)  [dot,label= {[label distance=-0.2em]right: \scriptsize  $     $} ] (centerc) {};
			\draw[kernel1,color=red] (center) to
			node [sloped,below] {\small }     (root);
			\draw[kernel1] (centerc) to
			node [sloped,below] {\small }     (center);
			\node at (0,2) [fill=white,label={[label distance=0em]center: \scriptsize  $ a  $} ] () {};
			\node at (0,5.25) [fill=white,label={[label distance=0em]center: \scriptsize  $ \Xi  $} ] () {};
		\end{tikzpicture}\otimes\begin{tikzpicture}[scale=0.2,baseline=0.1cm]
			\node at (0,0)  [dot,label= {[label distance=-0.2em]below: \scriptsize  $      $} ] (root) {};
			\node at (0,5)  [dot,label= {[label distance=-0.2em]right: \scriptsize  $     $} ] (center) {};
			\node at (0,9)  [dot,label= {[label distance=-0.2em]above: \scriptsize  $    $} ] (centerc) {};
			\node at (3,3)  [dot,label={[label distance=-0.2em]right: \scriptsize  $  $}] (right) {};
			\node at (-3,3)  [dot,label={[label distance=-0.2em]above: \scriptsize  $ $} ] (left) {};
			\node at (4.5,4.5)  [dot,label={[label distance=-0.2em]above: \scriptsize  $ $} ] (rightc) {};
			\node at (-3,5.5)  [dot,label={[label distance=-0.2em]above: \scriptsize  $ $} ] (leftc) {};
			\node at (0,11.5)  [dot,label={[label distance=-0.2em]above: \scriptsize  $ $} ] (top) {};
			\node at (3,10.5)  [dot,label={[label distance=-0.2em]above: \scriptsize  $ $} ] (top2) {};
			\node at (3,8)  [dot,label={[label distance=-0.2em]above: \scriptsize  $ $} ] (mid1) {};
			\draw[kernel1] (right) to
			node [sloped,below] {\small }     (root); 
			\draw[kernel1] (center) to
			node [sloped,below] {\small }     (root);
			\draw[kernel1] (center) to
			node [sloped,below] {\small }     (centerc);
			\draw[kernel1] (right) to
			node [sloped,below] {\small }     (rightc); 
			\draw[kernel1] (left) to
			node [sloped,below] {\small }     (root);
			\draw[kernel1] (leftc) to
			node [sloped,below] {\small }     (left);
			\draw[kernel1] (top2) to
			node [sloped,below] {\small }     (mid1);
			\draw[kernel1] (top) to
			node [sloped,below] {\small }     (centerc);
			\draw[kernel1] (mid1) to
			node [sloped,below] {\small }     (center);
			\node at (3.5,4) [fill=white,label={[label distance=0em]center: \scriptsize  $ \Xi $} ] () {};
			\node at (0,7) [fill=white,label={[label distance=0em]center: \scriptsize  $ a $} ] () {};
			\node at (-1.5,1.5) [fill=white,label={[label distance=0em]center: \scriptsize  $ a $} ] () {};
			\node at (1.5,1.5) [fill=white,label={[label distance=0em]center: \scriptsize   $ a  $} ] () {};
			\node at (0,2.5) [fill=white,label={[label distance=0em]center: \scriptsize  $ a  $} ] () {};
			\node at (-3,4.25) [fill=white,label={[label distance=0em]center: \scriptsize  $ \Xi $} ] () {};
			\node at (2,6.75) [fill=white,label={[label distance=0em]center: \scriptsize  $ a $} ] () {};
			\node at (0,10.25) [fill=white,label={[label distance=0em]center: \scriptsize  $ \Xi $} ] () {};
			\node at (3,9.25) [fill=white,label={[label distance=0em]center: \scriptsize  $ \Xi $} ] () {};
		\end{tikzpicture}.
\end{equation*}
Here, $\Delta_1$ acts by choosing one by one all the edges of the tree on which it is applied, removing them from it. The removed edges are then turned red and the node from which it has been removed turned blue (\textit{i.e.} the parameter $\mfv$ is set to $1$). This will imply that when we will graft the tree on the left on the tree on the right, we will get the original tree back. This is made concrete just below.
\end{example}

An abstract formulation of the flow equation reads
\begin{proposition}[Flow equation]\label{flow}
	We have, for $\tau\in\mcT_0^{\mfv=0}$,
	\begin{equation}
	\dmu\tau=(.\graft.)\Delta_1\tau
	\end{equation}
	where $\graft$ is defined in Equation \eqref{insert}.
\end{proposition}
Let us  mention that the previous proposition is essential in the proof of Proposition \ref{flowcoeff}.

\begin{remark}
Note that the composition of the insertion $\graft$ with the coproduct $\Delta_1$ is well defined since the trees on the left side of the coproduct is always a planted tree with base edge of type $\mcI'$. See Example \eqref{exdelta1} for a representation on a concrete case.
\end{remark}

\begin{proof}
We show the result by induction of the depth. The two base cases are trivial as both sides of the equality are $0$. For the inductive part, we have, keeping in mind that it is forbidden to graft on elements of $\mcT_0^{\mfv=0}$,
\begin{align*}
\dmu Y^0\mcI_a(\tau)&=Y^0\mcI'_a(\tau)+Y^0\mcI_a(\dmu\tau)\\
&=Y^0\mcI'_a(\tau)+Y^0\mcI_a((.\graft.)\Delta_1\tau)\\
&=Y^0\mcI'_a(\tau)\graft Y^1+(.\graft .)(\mathrm{id}\otimes\mcI_a)\Delta_1\tau\\
&=(.\graft.)\Delta_1 Y^0\mcI_a(\tau),
\end{align*}
and
\begin{align*}
\dmu(\sigma\tau)&=(\dmu\sigma)\times\tau+\sigma\times(\dmu\tau)\\
&=((.\graft.)\Delta_1\sigma)\times\tau+\sigma\times((.\graft.)\Delta_1\tau)\\
&=(.\graft.)(\Delta_1\sigma\times(Y^0\otimes\sigma))+(.\graft.)((Y^0\otimes\tau)\times\Delta_1\tau)\\
&=(.\graft.)\Delta_1(\sigma\tau).
\end{align*}
\end{proof}

\subsection{Elementary differentials}\label{eldiff}

We define recursively the so-called elementary differentials depending on trees in $T_0^{X,\mfg}$ as a family of functionals. We start with two functions $\Upsilon^0$ and $\Upsilon^1$ depending on abstract variables $Z_a$ for $a\in\N^{d+1}$. Namely,
\begin{equation}
\Upsilon^0(Z)=g(Z_1,\dots,Z_q)~~\mathrm{and}~~\Upsilon^1(Z)=f(Z_1,\dots,Z_p).
\end{equation}
With this definition at hand, it is clear that the $Z_a$'s represent the functions $\D^a\phi$, which will be evaluated at the same point, say $x$. We further introduce a new set of abstract variables $\Tilde{Z}_a$, for $a\in\N^{d+1}$, that will have the same use at the $Z_a$'s, but representing another dummy function $\Tilde{\phi}$. We define the derivatives $D_a:=D_{Z_a}$ and $\Tilde{D}_a:=D_{\Tilde{Z}_a}$, as well as the differential operators
\begin{equation*}
\D^{\eps_i}=\sum_a Z_{a+\eps_i}\miD_a,
\end{equation*}
where we have denoted $(\eps_i)$ the canonical basis of $\N^{d+1}$. It is a common fact that for any $i$ and $j$, $\D^{\eps_i}\D^{\eps_j}=\D^{\eps_j}\D^{\eps_i}$, which allows to define, for $k\in\N^{d+1}$,
\begin{equation*}
\D^k:=\prod_{i=1}^{d+1}(\D^{\eps_i})^{k_i},
\end{equation*}
as well as its counterpart $\Tilde{\D}^k$. We define a family of local functionals, indexed by trees of $T_0^X$, in the variables $(Z_a)$ and $(\Tilde{Z})_a$ that represent $\D^a\phi$ and $\D^a\Tilde{\phi}$, all evaluated at the same point $x$. Compared to the usual definition, the presence of a second argument depending on $\Tilde{\phi}$ might seem puzzling. In Section \ref{models}, we will define an evaluation map on the trees on which a Fréchet derivative in $\phi$ is going to act. This Fréchet derivative will, on the combinatorial side, have the effect of a grafting. However, we will not want it to act on every vertex, according to the decoration $\mfv$. That is why we need another dummy function the Fréchet derivative in $\phi$ will be blind to. Let us first set 
\begin{equation}
\TUpsilon[\zeta_lY^\mfv][\phi,\Tilde{\phi}]=\mfv\Upsilon^l[\phi]+(1-\mfv)\Upsilon^l[\Tilde{\phi}].
\end{equation}
We can then define inductively a linear map $\TUpsilon$ of functions in the variables $Z_a$ and $\Tilde{Z}_a$ on a tree $\tau$ of the form
\begin{equation}
 X^kY^\mfv\zeta_l\prod_{i=1}^n\mcI_{a_i}(\tau_i)\prod_{j=1}^m\CIg_{b_j}(X^{k_j}), 
\end{equation}
with $\tau_i\in T_0^X$ for every $i$ by
\begin{equation}\label{defupsilon}
\begin{split}
(\mfv=1)~~\TUpsilon[\tau]&=\left(\D^k \miD_{a_1}\dots \miD_{a_n}\miD_{b_1}\dots \miD_{b_m}\TUpsilon[\zeta_lY^1]\right)\prod_{i=1}^n\TUpsilon[\tau_i].\\
(\mfv=0)~~\TUpsilon[\tau]&=\left(\Tilde{\D}^k \Tilde{D}_{a_1}\dots \Tilde{D}_{a_n}\Tilde{D}_{b_1}\dots \Tilde{D}_{b_m}\TUpsilon[\zeta_lY^0]\right)\prod_{i=1}^n\TUpsilon[\tau_i].
\end{split}
\end{equation}

If in addition in this definition there is a factor $\Xig$, it overrides the term $\zeta_l$ and we take $\TUpsilon^1$ instead of the general term $\TUpsilon[\zeta_l]$. In addition, we let
\begin{equation}
\Upsilon[\tau][\phi]=\TUpsilon[\tau][\phi,\phi],
\end{equation}
that corresponds to the usual definition of elementary differential in regularity structures and is therefore independent of the decoration $\mfv$. These elementary differentials can be seen as an extension of the elementary differentials introduced in \cite[Sec. 2.7]{BCCH}.

Let us define inductively a linear map $\downg:\mcT_0^{X,\mfg}\to \mcT_0^X$ defined on $T_0^{X,\mfg}$ by
\begin{equation}
\downg X^k\zeta_l\Xig^p\prod_{i=1}^n\mcI_{a_i}(\tau_i)\prod_{j=1}^m\CIg_{b_j}(X^{k_j})=X^{k+\sum_{j=1}^m k_j}\zeta_l\prod_{i=1}^n\mcI_{a_i}(\downg\tau_i).
\end{equation}
Now, given a character $\ell:\mcT_0^X\to\R$, we define on $T_0^{X,\mfg}$, $\ell\TUpsilon[\tau]=\ell(\downg\tau)\TUpsilon[\tau]$ and extend it by linearity on $\mcT_0^{X,\mfg}$. We impose that it vanishes on trees of non-negative degree. We prove the following lemma, that will be crucial in the proof of Proposition \ref{graft}.

\begin{lemma}\label{lemmaupsilon}
We have, for any $\tau\in \mcT_0^X$,
\begin{equation}
D_a\ell\TUpsilon[\tau]=\ell\TUpsilon[\ggraft_a\tau].
\end{equation}
It further gives
\begin{equation}
D_{\D^a\phi}\ell\TUpsilon[\tau][\phi,\Tilde{\phi}]=\ell\TUpsilon[\ggraft_a\tau][\phi,\Tilde{\phi}],
\end{equation}
where we have denoted $D_{\D^a\phi}$ the Fréchet derivative in $\D^a\phi$, and $\D^a\phi$ only.
\end{lemma}

\begin{proof}
We start with the two base cases. The case $\tau=\zeta_lY^0$ is easy as both sides of the equality are $0$. For the other one, we have
\begin{align*}
\ell\TUpsilon[\ggraft_a\zeta_lY^1]&=\ell\TUpsilon\left[\CIg_{a}(\one)\zeta_lY^1\right]=\ell(\zeta_l)D_{a}\TUpsilon[\zeta_lY^1]=D_a\ell(\zeta_l)\TUpsilon[\zeta_lY^1]\\
&=D_a\ell\TUpsilon[\zeta_lY^1].
\end{align*}

Before starting the proof, let us recall from \cite{BB21b}, Lemma 2.1, that
\begin{equation*}
D_a\D^k=\sum_{j\in\N^{d+1}}\binom{k}{j}\partial^{k-j}D_{a-j}.
\end{equation*}
We thus have, for the case $\mfv=1$, noting that for $\tau\in\mcT_0^X$, $\downg\tau=\tau$,
\begin{equs}
~&\ell\TUpsilon\left[\ggraft_a\zeta_lX^kY^1\prod_{i=1}^n\mcI_{b_i}(\tau_i)\right]=\ell\TUpsilon\left[\sum_{j\in\N^{d+1}}\binom{k}{j}X^{k-j}\CIg_{a-j}(X^j)\zeta_lY^1\prod_{i=1}^n\mcI_{b_i}(\tau_i)\right]\\
&~~~~+\ell\TUpsilon\left[\zeta_lX^kY^1\sum_{i=1}^n\mcI_{b_i}(\ggraft_a\tau_i)\prod_{j\neq i}\mcI_{b_j}(\tau_j)\right]\\
&=\sum_{j\in\N^{d+1}}\binom{k}{j}\ell\left(\zeta_lX^k\prod_{i=1}^n\mcI_{b_i}(\tau_i)\right)\D^{k-j}D_{a-j}D_{b_1}\dots D_{b_n}\TUpsilon[\zeta_lY^1]\prod_{i=1}^n\TUpsilon[\tau_i]\\
&~~~~+\ell\left(\zeta_lX^k\prod_{i=1}^n\mcI_{b_i}(\tau_i)\right)\D^kD_{b_1}\dots D_{b_n}\TUpsilon[\zeta_lY^1]\sum_{i=1}^nD_a\TUpsilon[\tau_i]\prod_{j\neq i}\TUpsilon[\tau_j]\\
&=\ell\left(\zeta_lX^k\prod_{i=1}^n\mcI_{b_i}(\tau_i)\right)\Biggl(D_{a}D_{b_1}\dots D_{b_n}\TUpsilon[\zeta_lY^1]\prod_{i=1}^n\TUpsilon[\tau_i]\\
&~~~~+\D^kD_{b_1}\dots D_{b_n}\TUpsilon[\zeta_lY^1]\sum_{i=1}^nD_a\TUpsilon[\tau_i]\prod_{j\neq i}\TUpsilon[\tau_j]\Biggl)\\
&=\ell\left(\zeta_lX^k\prod_{i=1}^n\mcI_{b_i}(\tau_i)\right)D_a\left(\D^kD_{b_1}\dots D_{b_n}\TUpsilon[\zeta_lY^1]\prod_{i=1}^n\TUpsilon[\tau_i]\right)\\
&=D_a\left(\ell\left(\zeta_lX^k\prod_{i=1}^n\mcI_{b_i}(\tau_i)\right)\D^kD_{b_1}\dots D_{b_n}\TUpsilon[\zeta_lY^1]\prod_{i=1}^n\TUpsilon[\tau_i]\right)\\
&=D_a\ell\TUpsilon\left[\zeta_lX^kY^1\prod_{i=1}^n\mcI_{b_i}(\tau_i)\right].
\end{equs}
For the case $\mfv=0$, we get
\begin{equs}
~&\ell\TUpsilon\left[\ggraft_a\zeta_lX^kY^0\prod_{i=1}^n\mcI_{b_i}(\tau_i)\right]\\
&=\ell\TUpsilon\left[\zeta_lX^kY^0\sum_{i=1}^n\mcI_{b_i}(\ggraft_a\tau_i)\prod_{j\neq i}\mcI_{b_j}(\tau_j)\right]\\
&=\ell\left(\zeta_lX^k\prod_{i=1}^n\mcI_{b_i}(\tau_i)\right)\Tilde{\D}^k\Tilde{D}_{b_1}\dots \Tilde{D}_{b_n}\TUpsilon[\zeta_lY^0]\sum_{i=1}^nD_a\TUpsilon[\tau_i]\prod_{j\neq i}\TUpsilon[\tau_j]\\
&=D_a\left(\ell\left(\zeta_lX^k\prod_{i=1}^n\mcI_{b_i}(\tau_i)\right)\Tilde{\D}^k\Tilde{D}_{b_1}\dots \Tilde{D}_{b_n}\TUpsilon[\zeta_lY^0]\prod_{i=1}^n\TUpsilon[\tau_i]\right)\\
&=D_a\ell\TUpsilon\left[\zeta_lX^kY^0\prod_{i=1}^n\mcI_{b_i}(\tau_i)\right].
\end{equs}
The second statement of the lemma follows directly.
\end{proof}

\subsection{Extraction at the root}

In this subsection, we define an extraction-contraction coproduct that acts by extracting subtrees from the root. It will be later a paramount object in defining the renormalisation of the probabilistic quantities. We prove two algebraic properties that will be a key in the proof of Propositions \ref{commutation} and \ref{graft}.

\begin{definition}\label{deltar}
We define a coproduct $\Delta_r:\mcT_1\to \mcT_0^\mfg\otimes\mcT_1$ recursively by
\begin{equation}
\begin{split}
&\Delta_rY^\mfv=Y^\mfv\otimes\one,~~~~~~~~ \Delta_r\Xi=\Xig\otimes\Xi+\Xi\otimes\one,\\
&\Delta_r\mcI_a(\tau)=(\mcI_a\otimes\mathrm{id})\Delta_r\tau+ \sum_{k \in \N^{d+1}} \frac{1}{k!}\CIg_{a}(X^k)\otimes\mcI_{a+k}(\tau),\\
&\Delta_r\mcI'_a(\tau)= \sum_{k \in \N^{d+1}} \frac{1}{k!} \CIg_{a}(X^k) \otimes\mcI'_{a+k}(\tau),\\
&\Delta_r(\sigma\tau)=\Delta_r\sigma\Delta_r\tau.
\end{split}
\end{equation}
Moreover, in Section \ref{bphz}, we will need to define $\Delta_r$ on $T_0^X$. We do this by letting for every $i\in\{1,\cdots,d+1\}$,
\begin{equation}
\Delta_r X_i = \one\otimes X_i+X_i\otimes\one.
\end{equation}
$X_i$ is the $i$-th indeterminate of the polynomials in $d+1$ variables.
\end{definition}

We remind that we supposed that the SPDE \eqref{maineq} is subcritical, so that the sums appearing in the definition above are actually finite.

\begin{remark}
The expert will have noticed that this definition is very similar to the root-extraction coproduct in the inductive framework in regularity structures. We recall that it is given by $\delta_r:\mcT_0^X\to\mcT_0^X\otimes\mcT_0^X$ which reads
\begin{equation}\label{deltarrs}
\begin{split}
&\delta_r\one=\one\otimes\one,~~~~~~~~ \delta_r\Xi=\one\otimes\Xi+\Xi\otimes\one,~~~~~~~~\delta_r X_i = \one\otimes X_i+X_i\otimes\one,\\
&\delta_r\mcI_a(\tau)=(\mcI_a\otimes\mathrm{id})\delta_r\tau+ \sum_{k \in \N^{d+1}} \frac{X^k}{k!}\otimes\mcI_{a+k}(\tau),\\
&\delta_r(\sigma\tau)=\delta_r\sigma\delta_r\tau.
\end{split}
\end{equation}
A similar expression is given in \cite[Proposition 4.17]{BHZ} but used for the recentering and it appears for the first time in \cite[Sec. 8]{reg} expressed for a different decorated trees basis.

In addition to the Taylor expansion on the unextracted part, some unusual grey symbols appear. They will be of use in Section \ref{models} in the definition of the evaluation map \ref{eval}. These additional symbols allow to keep track of the information of what was grafted on the extracted tree, that otherwise could be lost during the procedure. The $X^k$ inside the symbol $\CIg$ will later only be read by the character used in the definition of the evaluation map.
\end{remark}

 Let us clarify this definition with the following example. 

\begin{example}\label{exdeltar}
We provide below two examples of computation of $\Delta_r$.
\begin{equs}
\Delta_r		\begin{tikzpicture}[scale=0.2,baseline=0.1cm]
			\node at (0,0)  [dot,label= {[label distance=-0.2em]below: \scriptsize  $      $} ] (root) {};
			\node at (0,4)  [dot,label= {[label distance=-0.2em]right: \scriptsize  $     $} ] (center) {};
			\node at (0,8)  [dot,label= {[label distance=-0.2em]above: \scriptsize  $    $} ] (centerc) {};
			\node at (0,10.5)  [dot,label={[label distance=-0.2em]above: \scriptsize  $ $} ] (top) {};
			\node at (3,9.5)  [dot,label={[label distance=-0.2em]above: \scriptsize  $ $} ] (top2) {};
			\node at (3,7)  [dot,label={[label distance=-0.2em]above: \scriptsize  $ $} ] (mid1) {};
			\node at (-2,6)  [dot,label={[label distance=-0.2em]above: \scriptsize  $ $} ] (mid2) {};
			\draw[kernel1] (center) to
			node [sloped,below] {\small }     (root);
			\draw[kernel1,color=red] (center) to
			node [sloped,below] {\small }     (centerc);
			\draw[kernel1] (top2) to
			node [sloped,below] {\small }     (mid1);
			\draw[kernel1] (top) to
			node [sloped,below] {\small }     (centerc);
			\draw[kernel1] (mid1) to
			node [sloped,below] {\small }     (center);
			\draw[kernel1] (mid2) to
			node [sloped,below] {\small }     (center);
			\node at (0,6.5) [fill=white,label={[label distance=0em]center: \scriptsize  $ a $} ] () {};
			\node at (0,2) [fill=white,label={[label distance=0em]center: \scriptsize  $ a  $} ] () {};
			\node at (0,9.25) [fill=white,label={[label distance=0em]center: \scriptsize  $ \Xi $} ] () {};
			\node at (3,8.25) [fill=white,label={[label distance=0em]center: \scriptsize  $ \Xi $} ] () {};
			\node at (2,5.75) [fill=white,label={[label distance=0em]center: \scriptsize  $ a $} ] () {};
			\node at (-1,5) [fill=white,label={[label distance=0em]center: \scriptsize  $ \Xi $} ] () {};
		\end{tikzpicture}&=\sum_{k\in\N^{d+1}}\frac{1}{k!}\begin{tikzpicture}[scale=0.2,baseline=0.1cm]
			\node at (0,0)  [dot,label= {[label distance=-0.2em]below: \scriptsize  $      $} ] (root) {};
			\node at (0,4)  [dot,label= {[label distance=-0.2em]below: \scriptsize  $      $} ] (center) {};
			\node at (0,4)  [dot,label={[label distance=-0.2em]right: \scriptsize  $ k $}] (right) {};
			\draw[kernel1,color=gray] (center) to
			node [sloped,below] {\small }     (root);
			\node at (0,2) [fill=white,label={[label distance=0em]center: \scriptsize  $ a $} ] () {};
			\end{tikzpicture}\otimes\begin{tikzpicture}[scale=0.2,baseline=0.1cm]
			\node at (0,0)  [dot,label= {[label distance=-0.2em]below: \scriptsize  $      $} ] (root) {};
			\node at (0,4)  [dot,label= {[label distance=-0.2em]right: \scriptsize  $     $} ] (center) {};
			\node at (0,8)  [dot,label= {[label distance=-0.2em]above: \scriptsize  $    $} ] (centerc) {};
			\node at (0,10.5)  [dot,label={[label distance=-0.2em]above: \scriptsize  $ $} ] (top) {};
			\node at (3,9.5)  [dot,label={[label distance=-0.2em]above: \scriptsize  $ $} ] (top2) {};
			\node at (3,7)  [dot,label={[label distance=-0.2em]above: \scriptsize  $ $} ] (mid1) {};
			\node at (-2,6)  [dot,label={[label distance=-0.2em]above: \scriptsize  $ $} ] (mid2) {};
			\draw[kernel1] (center) to
			node [sloped,below] {\small }     (root);
			\draw[kernel1,color=red] (center) to
			node [sloped,below] {\small }     (centerc);
			\draw[kernel1] (top2) to
			node [sloped,below] {\small }     (mid1);
			\draw[kernel1] (top) to
			node [sloped,below] {\small }     (centerc);
			\draw[kernel1] (mid1) to
			node [sloped,below] {\small }     (center);
			\draw[kernel1] (mid2) to
			node [sloped,below] {\small }     (center);
			\node at (0,6.5) [fill=white,label={[label distance=0em]center: \scriptsize  $ a $} ] () {};
			\node at (0,2) [fill=white,label={[label distance=0em]center: \scriptsize  $ a+k  $} ] () {};
			\node at (0,9.25) [fill=white,label={[label distance=0em]center: \scriptsize  $ \Xi $} ] () {};
			\node at (3,8.25) [fill=white,label={[label distance=0em]center: \scriptsize  $ \Xi $} ] () {};
			\node at (2,5.75) [fill=white,label={[label distance=0em]center: \scriptsize  $ a $} ] () {};
			\node at (-1,5) [fill=white,label={[label distance=0em]center: \scriptsize  $ \Xi $} ] () {};
		\end{tikzpicture}+\sum_{i,j\in\N^{d+1}}\frac{1}{i!j!}\begin{tikzpicture}[scale=0.2,baseline=0.1cm]
			\node at (0,0)  [dot,label= {[label distance=-0.2em]below: \scriptsize  $      $} ] (root) {};
			\node at (0,4)  [dot,label= {[label distance=-0.2em]right: \scriptsize  $     $} ] (center) {};
			\node at (0,8)  [dot,label= {[label distance=-0.2em]above: \scriptsize  $    $} ] (centerc) {};
			\node at (0,8)  [dot,label={[label distance=-0.2em]right: \scriptsize  $ i $}] (right) {};
			\node at (3,7)  [dot,label={[label distance=-0.2em]above: \scriptsize  $ $} ] (mid1) {};
			\node at (3,7)  [dot,label={[label distance=-0.2em]right: \scriptsize  $ j $}] (right) {};
			\node at (-2,6)  [dot,label={[label distance=-0.2em]above: \scriptsize  $ $} ] (mid2) {};
			\draw[kernel1] (center) to
			node [sloped,below] {\small }     (root);
			\draw[kernel1,color=gray] (center) to
			node [sloped,below] {\small }     (centerc);
			\draw[kernel1,color=gray] (mid1) to
			node [sloped,below] {\small }     (center);
			\draw[kernel1] (mid2) to
			node [sloped,below] {\small }     (center);
			\node at (0,6.5) [fill=white,label={[label distance=0em]center: \scriptsize  $ a $} ] () {};
			\node at (0,2) [fill=white,label={[label distance=0em]center: \scriptsize  $ a  $} ] () {};
			\node at (2,5.75) [fill=white,label={[label distance=0em]center: \scriptsize  $ a $} ] () {};
			\node at (-1,5) [fill=white,label={[label distance=0em]center: \scriptsize  $ \Xi $} ] () {};
		\end{tikzpicture}\otimes\begin{tikzpicture}[scale=0.2,baseline=0.1cm]
			\node at (0,0)  [dot,label= {[label distance=-0.2em]below: \scriptsize  $      $} ] (root) {};
			\node at (2,3)  [dot,label={[label distance=-0.2em]right: \scriptsize  $  $}] (right) {};
			\node at (-2,3)  [dot,label={[label distance=-0.2em]above: \scriptsize  $ $} ] (left) {};
			\node at (2,5.5)  [dot,label={[label distance=-0.2em]above: \scriptsize  $ $} ] (rightc) {};
			\node at (-2,5.5)  [dot,label={[label distance=-0.2em]above: \scriptsize  $ $} ] (leftc) {};
			\draw[kernel1] (right) to
			node [sloped,below] {\small }     (root); 
			\draw[kernel1] (right) to
			node [sloped,below] {\small }     (rightc); 
			\draw[kernel1,color=red] (left) to
			node [sloped,below] {\small }     (root);
			\draw[kernel1] (leftc) to
			node [sloped,below] {\small }     (left);
			\node at (2,4.25) [fill=white,label={[label distance=0em]center: \scriptsize  $ \Xi $} ] () {};
			\node at (-1.75,1.5) [fill=white,label={[label distance=0em]center: \scriptsize  $ a+i $} ] () {};
			\node at (1.75,1.5) [fill=white,label={[label distance=0em]center: \scriptsize   $ a+j  $} ] () {};
			\node at (-2,4.25) [fill=white,label={[label distance=0em]center: \scriptsize  $ \Xi $} ] () {};
		\end{tikzpicture}\\
		&~~~~+\sum_{k\in\N^{d+1}}\frac{1}{k!}\begin{tikzpicture}[scale=0.2,baseline=0.1cm]
			\node at (0,0)  [dot,label= {[label distance=-0.2em]below: \scriptsize  $      $} ] (root) {};
			\node at (0,4)  [dot,label= {[label distance=-0.2em]right: \scriptsize  $     $} ] (center) {};
			\node at (0,8)  [dot,label= {[label distance=-0.2em]above: \scriptsize  $    $} ] (centerc) {};
				\node at (0,8)  [dot,label={[label distance=-0.2em]left: \scriptsize  $ k $}] (right) {};
			\node at (3,9.5)  [dot,label={[label distance=-0.2em]above: \scriptsize  $ $} ] (top2) {};
			\node at (3,7)  [dot,label={[label distance=-0.2em]above: \scriptsize  $ $} ] (mid1) {};
			\node at (-2,6)  [dot,label={[label distance=-0.2em]above: \scriptsize  $ $} ] (mid2) {};
			\draw[kernel1] (center) to
			node [sloped,below] {\small }     (root);
			\draw[kernel1,color=gray] (center) to
			node [sloped,below] {\small }     (centerc);
			\draw[kernel1] (top2) to
			node [sloped,below] {\small }     (mid1);
			\draw[kernel1] (mid1) to
			node [sloped,below] {\small }     (center);
			\draw[kernel1] (mid2) to
			node [sloped,below] {\small }     (center);
			\node at (0,6.5) [fill=white,label={[label distance=0em]center: \scriptsize  $ a $} ] () {};
			\node at (0,2) [fill=white,label={[label distance=0em]center: \scriptsize  $ a  $} ] () {};
			\node at (3,8.25) [fill=white,label={[label distance=0em]center: \scriptsize  $ \Xi $} ] () {};
			\node at (2,5.75) [fill=white,label={[label distance=0em]center: \scriptsize  $ a $} ] () {};
			\node at (-1,5) [fill=white,label={[label distance=0em]center: \scriptsize  $ \Xi $} ] () {};
		\end{tikzpicture}\otimes\begin{tikzpicture}[scale=0.2,baseline=0.1cm]
			\node at (0,0)  [dot,label= {[label distance=-0.2em]below: \scriptsize  $      $} ] (root) {};
			\node at (0,4)  [dot,label= {[label distance=-0.2em]right: \scriptsize  $     $} ] (center) {};
			\node at (0,6.5)  [dot,label= {[label distance=-0.2em]right: \scriptsize  $     $} ] (centerc) {};
			\draw[kernel1,color=red] (center) to
			node [sloped,below] {\small }     (root);
			\draw[kernel1] (centerc) to
			node [sloped,below] {\small }     (center);
			\node at (0,2) [fill=white,label={[label distance=0em]center: \scriptsize  $ a+k  $} ] () {};
			\node at (0,5.25) [fill=white,label={[label distance=0em]center: \scriptsize  $ \Xi  $} ] () {};
		\end{tikzpicture},
\end{equs}
And
\begin{equs}
\Delta_r\begin{tikzpicture}[scale=0.2,baseline=0.1cm]
			\node at (0,0)  [dot,label= {[label distance=-0.2em]below: \scriptsize  $      $} ] (root) {};
			\node at (0,4)  [dot,label= {[label distance=-0.2em]right: \scriptsize  $     $} ] (center) {};
			\node at (0,6.5)  [dot,label= {[label distance=-0.2em]right: \scriptsize  $     $} ] (centerc) {};
			\node at (-2,2)  [dot,label= {[label distance=-0.2em]right: \scriptsize  $     $} ] (left) {};
			\draw[kernel1] (center) to
			node [sloped,below] {\small }     (root);
			\draw[kernel1] (centerc) to
			node [sloped,below] {\small }     (center);
			\draw[kernel1] (left) to
			node [sloped,below] {\small }     (root);
			\node at (0,2) [fill=white,label={[label distance=0em]center: \scriptsize  $ a  $} ] () {};
			\node at (0,5.25) [fill=white,label={[label distance=0em]center: \scriptsize  $ \Xi  $} ] () {};
			\node at (-1,1) [fill=white,label={[label distance=0em]center: \scriptsize  $ \Xi $} ] () {};
		\end{tikzpicture}=\sum_{k\in\N^{d+1}}\frac{1}{k!}\begin{tikzpicture}[scale=0.2,baseline=0.1cm]
			\node at (0,0)  [dot,label= {[label distance=-0.2em]below: \scriptsize  $      $} ] (root) {};
			\node at (0,4)  [dot,label= {[label distance=-0.2em]right: \scriptsize  $     $} ] (center) {};
			\node at (0,4)  [dot,label={[label distance=-0.2em]right: \scriptsize  $ k $}] (right) {};
			\node at (-2,2)  [dot,label= {[label distance=-0.2em]right: \scriptsize  $     $} ] (left) {};
			\draw[kernel1,color=gray] (center) to
			node [sloped,below] {\small }     (root);
			\draw[kernel1] (left) to
			node [sloped,below] {\small }     (root);
			\node at (0,2) [fill=white,label={[label distance=0em]center: \scriptsize  $ a  $} ] () {};
			\node at (-1,1) [fill=white,label={[label distance=0em]center: \scriptsize  $ \Xig $} ] () {};
		\end{tikzpicture}\otimes\begin{tikzpicture}[scale=0.2,baseline=0.1cm]
			\node at (0,0)  [dot,label= {[label distance=-0.2em]below: \scriptsize  $      $} ] (root) {};
			\node at (0,4)  [dot,label= {[label distance=-0.2em]right: \scriptsize  $     $} ] (center) {};
			\node at (0,6.5)  [dot,label= {[label distance=-0.2em]right: \scriptsize  $     $} ] (centerc) {};
			\node at (-2,2)  [dot,label= {[label distance=-0.2em]right: \scriptsize  $     $} ] (left) {};
			\draw[kernel1] (center) to
			node [sloped,below] {\small }     (root);
			\draw[kernel1] (centerc) to
			node [sloped,below] {\small }     (center);
			\draw[kernel1] (left) to
			node [sloped,below] {\small }     (root);
			\node at (0,3) [fill=white,label={[label distance=0em]center: \scriptsize  $ ~~~~~~a+k  $} ] () {};
			\node at (0,5.25) [fill=white,label={[label distance=0em]center: \scriptsize  $ \Xi  $} ] () {};
			\node at (-1,1) [fill=white,label={[label distance=0em]center: \scriptsize  $ \Xi $} ] () {};
		\end{tikzpicture}+\sum_{k\in\N^{d+1}}\frac{1}{k!}\begin{tikzpicture}[scale=0.2,baseline=0.1cm]
			\node at (0,0)  [dot,label= {[label distance=-0.2em]below: \scriptsize  $      $} ] (root) {};
			\node at (0,4)  [dot,label= {[label distance=-0.2em]right: \scriptsize  $     $} ] (center) {};
			\node at (0,4)  [dot,label={[label distance=-0.2em]right: \scriptsize  $ k $}] (right) {};
			\node at (-2,2)  [dot,label= {[label distance=-0.2em]right: \scriptsize  $     $} ] (left) {};
			\draw[kernel1,color=gray] (center) to
			node [sloped,below] {\small }     (root);
			\draw[kernel1] (left) to
			node [sloped,below] {\small }     (root);
			\node at (0,2) [fill=white,label={[label distance=0em]center: \scriptsize  $a  $} ] () {};
			\node at (-1,1) [fill=white,label={[label distance=0em]center: \scriptsize  $ \Xi $} ] () {};
		\end{tikzpicture}\otimes\begin{tikzpicture}[scale=0.2,baseline=0.1cm]
			\node at (0,0)  [dot,label= {[label distance=-0.2em]below: \scriptsize  $      $} ] (root) {};
			\node at (0,4)  [dot,label= {[label distance=-0.2em]right: \scriptsize  $     $} ] (center) {};
			\node at (0,6.5)  [dot,label= {[label distance=-0.2em]right: \scriptsize  $     $} ] (centerc) {};		
			\draw[kernel1] (center) to
			node [sloped,below] {\small }     (root);
			\draw[kernel1] (centerc) to
			node [sloped,below] {\small }     (center);
			\node at (0,2) [fill=white,label={[label distance=0em]center: \scriptsize  $ a+k  $} ] () {};
			\node at (0,5.25) [fill=white,label={[label distance=0em]center: \scriptsize  $ \Xi  $} ] () {};
		\end{tikzpicture}+\begin{tikzpicture}[scale=0.2,baseline=0.1cm]
			\node at (0,0)  [dot,label= {[label distance=-0.2em]below: \scriptsize  $      $} ] (root) {};
			\node at (0,4)  [dot,label= {[label distance=-0.2em]right: \scriptsize  $     $} ] (center) {};
			\node at (0,6.5)  [dot,label= {[label distance=-0.2em]right: \scriptsize  $     $} ] (centerc) {};
			\node at (-2,2)  [dot,label= {[label distance=-0.2em]right: \scriptsize  $     $} ] (left) {};
			\draw[kernel1] (center) to
			node [sloped,below] {\small }     (root);
			\draw[kernel1] (centerc) to
			node [sloped,below] {\small }     (center);
			\draw[kernel1] (left) to
			node [sloped,below] {\small }     (root);
			\node at (0,2) [fill=white,label={[label distance=0em]center: \scriptsize  $ a  $} ] () {};
			\node at (0,5.25) [fill=white,label={[label distance=0em]center: \scriptsize  $ \Xi  $} ] () {};
			\node at (-1,1) [fill=white,label={[label distance=0em]center: \scriptsize  $ \Xi $} ] () {};
		\end{tikzpicture}\otimes\one.
\end{equs}
\end{example}

\begin{proposition} 
We have the following commutation property on $\mcT_0$.
\begin{equation} \label{delta_r_mu}
(\mathrm{id}\otimes\dmu)\Delta_r=\Delta_r\dmu.
\end{equation}
\end{proposition}

\begin{proof}
We proceed by induction on the depth of the tree. The statement is obvious for elementary trees, as both sides of the equality are zero. We have for any planted tree
\begin{equs}
\Delta_r\dmu\mcI_a(\tau)&=\Delta_r(\mcI_a'(\tau)+\mcI_a(\dmu\tau))\\
&=\sum_{k\in\N^{d+1}}\frac{1}{k!}\CIg_{a}(X^k) \otimes\mcI'_{a+k}(\tau)+(\mcI_a\otimes\mathrm{id})\Delta_r\dmu\tau\\&~~~~+\sum_{k\in\N^{d+1}}\frac{1}{k!}\CIg_{a}(X^k)\otimes\mcI_{a+k}(\dmu\tau)\\
&=\sum_{k\in\N^{d+1}}\frac{1}{k!}\CIg_{a}(X^k) \otimes\mcI'_{a+k}(\tau)+(\mcI_a\otimes\dmu)\Delta_r\tau\\&~~~~+\sum_{k\in\N^{d+1}}\frac{1}{k!}\CIg_{a}(X^k)\otimes\mcI_{a+k}(\dmu\tau)\\
&=(\mcI_a\otimes\dmu)\Delta_r\tau+\sum_{k\in\N^{d+1}}\frac{1}{k!}\CIg_{a}(X^k) \otimes(\mcI'_{a+k}(\tau)+\mcI_{a+k}(\dmu\tau))\\
&=(\mathrm{id}\otimes\dmu)\Delta_r\mcI_a(\tau).
\end{equs}

We have then, for the product of two trees,
\begin{equs}
\Delta_r\dmu(\sigma\tau)&=\Delta_r(\dmu\sigma\tau+\sigma\dmu\tau)\\
&=\Delta_r\dmu\sigma\Delta_r\tau+\Delta_r\sigma\Delta_r\dmu\tau\\
&=\bigl((\mathrm{id}\otimes\dmu)\Delta_r\sigma\bigl)\Delta_r\tau+\Delta_r\sigma\bigl((\mathrm{id}\otimes\dmu)\Delta_r\tau\bigl)\\
&=(\mathrm{id}\otimes\dmu)(\Delta_r\sigma\Delta_r\tau)\\
&=(\mathrm{id}\otimes\dmu)\Delta_r\sigma\tau.
\end{equs}
\end{proof}

\begin{proposition}\label{lemmagraft}
The following holds for $\sigma,\tau\in \mcT_0$.
\begin{equation}
\Delta_r(\sigma\rgraft_a\tau)=
\sum_{k\in\N^{d+1}}\frac{1}{k!}\bigl(\ggraft_{a,k}\otimes(\sigma\rgraft_{a+k}^{\text{\tiny{root}}}\cdot)\bigl)\Delta_r\tau+\bigl(\mathrm{id}\otimes(\sigma\rgraft_a^{\text{\tiny{non-root}}}\cdot)\bigl)\Delta_r\tau.
\end{equation}
\end{proposition}

\begin{proof}
We proceed by induction on the depth of $\tau$. Let us start with elementary trees. We have $\sigma\rgraft_aY^1\one=Y^1\mcI'(\sigma)$ so that
\begin{equs}
\Delta_r(\sigma\rgraft_aY^1\one)& =\sum_{k\in\N^{d+1}}\frac{1}{k!}\CIg_{a}(X^k) \otimes\mcI'_{a+k}(\sigma)
\\ &=\sum_{k \in \N^{d+1}} \frac{1}{k!}\bigl(\ggraft_{a,k}\otimes(\sigma
\rgraft_{a+k}^{\text{\tiny{root}}}\cdot)\bigl)\Delta_r\one
\end{equs}
Noticing that the non-root grafting does not happen in this case, the proposition is verified for $\tau=Y^1\one$. The case $\tau=Y^0\one$ is trivial since both sides are $0$. It is proved the same way for $\tau=Y^\mfv\Xi$.\\
 Suppose now it is satisfied for some $\tau$, we want to prove it for $Y^\mfv\mcI_b(\tau)$. We write explicitly the case $\mfv=1$, the proof for the other case being rigorously the same but easier. We choose not to write it for the sake of conciseness. We forget the symbol $Y^1$ for the rest of the proof. First note that $\sigma\rgraft_a\mcI_b(\tau)=\mcI'_a(\sigma)\mcI_b(\tau)+\mcI_b(\sigma\rgraft_a\tau)$. It gives on one hand
\begin{equs}
~&\Delta_r(\sigma\rgraft_a\mcI_b(\tau))=\Delta_r\mcI'_a(\sigma)\Delta_r\mcI_b(\tau)+\Delta_r\mcI_b(\sigma\rgraft_a\tau)\\
&=\left(\sum_{k\in\N^{d+1}}\frac{1}{k!}\CIg_{a}(X^k)\otimes\mcI'_{a+k}(\sigma)\right)\Biggl((\mcI_b\otimes\mathrm{id})\Delta_r\tau\\
&~~~~+\sum_{k\in\N^{d+1}}\frac{1}{k!}\CIg_{b}(X^k)\otimes\mcI_{b+k}(\tau)\Biggl)+(\mcI_b\otimes\mathrm{id})\Delta_r(\sigma\rgraft_a\tau)\\
&~~~~+\sum_{k\in\N^{d+1}}\frac{1}{k!}\CIg_{b}(X^k)\otimes\mcI_{b+k}(\sigma\rgraft_a\tau).
\end{equs}
At this point, we use the induction hypothesis on $\Delta_r(\sigma\rgraft_a\tau)$. We have on the other hand
\begin{multline*}
\bigl(\mathrm{id}\otimes(\sigma
\rgraft_a^{\text{\tiny{non-root}}}\cdot)\bigl)\Delta_r\mcI_b(\tau)=\bigl(\mcI_b\otimes(\sigma
\rgraft_a^{\text{\tiny{non-root}}}\cdot)\bigl)\Delta_r\tau\\+\sum_{k\in\N^{d+1}}\frac{1}{k!}\CIg_{b}(X^k)\otimes\mcI_{b+k}(\sigma\rgraft_a\tau),
\end{multline*}
as well as, using Sweedler's notation $\Delta_r\tau=\sum_{(\tau)}\tau_1\otimes\tau_2$,
\begin{equs}
~&\sum_{k\in\N^{d+1}}\frac{1}{k!}\bigl(\ggraft_{a,k}\otimes(\sigma
\rgraft_{a+k}^{\text{\tiny{root}}}\cdot)\bigl)\Delta_r\mcI_b(\tau)\\
&=\sum_{k\in\N^{d+1}}\frac{1}{k!}\bigl(\ggraft_{a,k}\otimes(\sigma
\rgraft_{a+k}^{\text{\tiny{root}}}\cdot)\bigl)\left((\mcI_b\otimes\mathrm{id})\Delta_r\tau+\sum_{i\in\N^{d+1}}\frac{1}{i!}\CIg_{b}(X^i)\otimes\mcI_{b+i}(\tau)\right)\\
&=\sum_{k\in\N^{d+1}}\sum_{(\tau)}\frac{1}{k!}\ggraft_{a,k}\mcI_b(\tau_1)\otimes\mcI'_{a+k}(\sigma)\tau_2\\
&~~~~+\sum_{i,k\in\N^{d+1}}\frac{1}{i!k!}\CIg_{a}(X^i)\CIg_{b}(X^k)\otimes\mcI'_{a+k}(\sigma)\mcI_{b+k}(\tau).
\end{equs}
The first term of the last line reads, using the definition of $\ggraft_{a,k}$,
\begin{equs}
~&\sum_{(\tau)}\sum_{k\in\N^{d+1}}\left(\frac{1}{k!}\CIg_{a}(X^k)\mcI_b(\tau_1)+\mcI_b(\ggraft_{a,k}\tau_1)\right)\otimes\mcI'_{a+k}(\sigma)\tau_2\\
&=\left(\sum_{k\in\N^{d+1}}\frac{1}{k!}\CIg_{a}(X^k)\otimes\mcI'_{a+k}(\sigma)\right)(\mcI_b\otimes\mathrm{id})\Delta_r\tau\\
&~~~~+(\mcI_b\otimes\mathrm{id})\sum_{k\in\N^{d+1}}\frac{1}{k!}\bigl(\ggraft_{a,k}\otimes(\sigma
\rgraft_{a+k}^{\text{\tiny{root}}}\cdot)\bigl)\Delta_r\tau.
\end{equs}
Putting it all together gives the result for $\mcI_b(\tau)$. We now assume the result true for two trees $\mu$ and $\tau$ and prove it for $\sigma\tau$. To begin, we have the relation \begin{equation*}
	\sigma\rgraft_a\tau\mu=\mcI'_a(\sigma)\tau\mu+\mu(\sigma\rgraft^{\text{\tiny{non-root}}}\tau)+\tau(\sigma\rgraft^{\text{\tiny{non-root}}}\mu).
\end{equation*}
 One then gets
\begin{equs}
\Delta_r(\sigma\rgraft_a\tau\mu)&=\left(\sum_{k\in\N^{d+1}}\frac{1}{k!}\CIg_{a}(X^k)\otimes\mcI'_{a+k}(\sigma)\right)\Delta_r\tau\Delta_r\mu+\Delta_r\mu\Delta_r(\sigma\rgraft^{\text{\tiny{non-root}}}\tau)\\
&~~~~+\Delta_r\tau\Delta_r(\sigma\rgraft^{\text{\tiny{non-root}}}\mu)\\
&=\left(\sum_{k\in\N^{d+1}}\frac{1}{k!}\CIg_{a}(X^k)\otimes\mcI'_{a+k}(\sigma)\right)\Delta_r\tau\Delta_r\mu-\Delta_r\mu\Delta_r(\mcI_a'(\sigma)\tau)\\
&~~~~-\Delta_r\tau\Delta_r(\mcI_a'(\sigma)\mu)+\Delta_r\mu\Delta_r(\sigma\rgraft_a\tau)+\Delta_r\tau\Delta_r(\sigma\rgraft_a\mu)\\
&=-\left(\sum_{k\in\N^{d+1}}\frac{1}{k!}\CIg_{a}(X^k)\otimes\mcI'_{a+k}(\sigma)\right)\Delta_r\tau\Delta_r\mu\\
&~~~~+\Delta_r\mu\sum_{k\in\N^{d+1}}\frac{1}{k!}\bigl(\ggraft_{a,k}\otimes(\sigma
\rgraft_{a+k}^{\text{\tiny{root}}}\cdot)\bigl)\Delta_r\tau\\
&~~~~+\Delta_r\tau\sum_{k\in\N^{d+1}}\frac{1}{k!}\bigl(\ggraft_{a,k}\otimes(\sigma
\rgraft_{a+k}^{\text{\tiny{root}}}\cdot)\bigl)\Delta_r\mu\\
&~~~~+\Delta_r\tau\bigl(\mathrm{id}\otimes(\sigma\rgraft^{\text{\tiny{non-root}}}\cdot)\bigl)\Delta_r\mu+\Delta_r\mu\bigl(\mathrm{id}\otimes(\sigma\rgraft^{\text{\tiny{non-root}}}\cdot)\bigl)\Delta_r\tau\\
&=\sum_{k\in\N^{d+1}}\frac{1}{k!}\bigl(\ggraft_{a,k}\otimes(\sigma
\rgraft_{a+k}^{\text{\tiny{root}}}\cdot)\bigl)\Delta_r(\tau\mu)+\Delta_r\tau\bigl(\mathrm{id}\otimes(\sigma\rgraft^{\text{\tiny{non-root}}}\cdot)\bigl)\Delta_r\mu\\
&~~~~+\Delta_r\mu\bigl(\mathrm{id}\otimes(\sigma\rgraft^{\text{\tiny{non-root}}}\cdot)\bigl)\Delta_r\tau
\end{equs}
where we have used to write the last line the equality $\ggraft_{a}(\tau'\mu')=\mu'\ggraft_{a}\tau'+\tau'\ggraft_{a}\mu'-\CIg_{a}(\one)\tau'\sigma'$. To finish the proof, one can see that, in Sweedler's notation
\begin{equs}
\bigl(\ggraft_a\otimes(\sigma\rgraft_a^{\text{\tiny{non-root}}}\cdot)\bigl)&\Delta_r(\tau\mu)=\sum_{(\tau),(\mu)}\tau_1\mu_1\otimes(\sigma\rgraft_a^{\text{\tiny{non-root}}}(\tau_2\mu_2))\\
&=\sum_{(\tau),(\mu)}\tau_1\mu_1\otimes(\tau_2(\sigma\rgraft_a^{\text{\tiny{non-root}}}\mu_2)+\mu_2(\sigma\rgraft_a^{\text{\tiny{non-root}}}\tau_2))\\
&=\Delta_r\tau\bigl(\mathrm{id}\otimes(\sigma\rgraft^{\text{\tiny{non-root}}}\cdot)\bigl)\Delta_r\mu+\Delta_r\mu\bigl(\mathrm{id}\otimes(\sigma\rgraft^{\text{\tiny{non-root}}}\cdot)\bigl)\Delta_r\tau.
\end{equs}
\end{proof}

\section{Evaluation map and renormalisation}\label{models}

We want to define, as in regularity structures, a realization map that transforms an abstract tree into a distribution. We want this map to also contain the renormalisation of the given tree. We choose to use a formalism that is close to the one introduced in \cite{BR18}, and perfected in \cite{BB21b}, using the so-called preparation maps.  We recall these definitions here, \textit{but the reader must keep in mind that there are given only for comparison, and will not be reused in the rest of the paper.}

One first defines the preparation map on $\mcT_0^X$ by $R=(\ell\otimes\mathrm{id})\delta_r$, where $\ell$ is a character on trees (typically the BHZ character introduced in \cite{BHZ}), and $\delta_r$ has been defined in equation \eqref{deltarrs}. The model  is then set inductively in the depth of the tree by alternating between a multiplicative model $\Pi_x^{R\times}$ and the preparation map $R$, precisely:

\begin{equation}\label{modelrs}
\begin{split}
&(\Pi_{x,\eps}^{R\times}\one)(y)=1,~~(\Pi_{x,\eps}^{R\times}\Xi)(y)=\xi_\eps(y)~~\mathrm{and}~~(\Pi_{x,\eps}^{R\times}X_i)=y_i-x_i,\\
&(\Pi_{x,\eps}^{R\times}\sigma\tau)(y)=(\Pi_{x,\eps}^{R\times}\sigma)(y)\times(\Pi_{x,\eps}^{R\times}\tau)(y),\\
&(\Pi_{x,\eps}^{R\times}\mcI_a(\tau))(y)=(\D^a G*\Pi_{x,\eps}^{R\times}R\tau)(y)
 -\sum_{\tiny{|k|_\mfs\leq \deg(\mcI_a\tau)}}\frac{(y-x)^k}{k!}(\D^{a+k} G*\Pi_{x,\eps}^{R\times}R\tau)(x),\\
&(\Pi_{x,\eps}^R\tau)(y)=(\Pi_{x,\eps}^{R\times}R\tau)(y).
\end{split}
\end{equation}

The formalism we introduce in the context of the flow approach involves the same concept. However, a main difference appears compared to regularity structures since in the flow approach, the iterated stochastic integrals are not localized, \textit{i.e.} some of the implicated parameters of the functional ($\varphi$ or $\Tilde{\varphi}$) are being convolved by the singular kernels, where in regularity structures, the terms in the ansatz consist of a purely stochastic term and a functional side to side. A localisation step still has to be applied at the end of the procedure, on the extracted sub-trees. This action is represented by the localisation map $\Mloc$, applied on the left-hand side of the coproduct. The work of this section also highlights why the top-down renormalisation is the same as the bottom-up one, the usual scheme adopted in the previous works being top-down where ours is bottom-up. In the regularity structures literature, top-down means that the form of the renormalisation counterterm is assumed and one identifies the renormalisation coefficients afterwards, where bottom-up is the opposite, \ie no assumption is made but the counterterm is found at the end of the procedure.

\begin{definition}[Evaluation map]\label{eval}
Given $\eps>0$ and $\mu\in[0,1]$, we define a map $\teval$ on $\mcT_0$ inductively. We first set
\begin{equation}
(\heval\one)[\phi,\Tilde{\phi}]=1~~\mathrm{and}~~(\heval\Xi)[\phi,\Tilde{\phi}]=\xi_\eps.
\end{equation}
Then
\begin{equation}
\heval(\tau\sigma)=\heval\tau\times\heval\sigma,
\end{equation}
and
\begin{equation}
\heval\mcI_a(\tau)=\D^a(G-G_\mu)*\eval^R\tau,~~~~\heval\mcI'_a(\tau)=-\D^a\dot G_\mu*\eval^R\tau.
\end{equation}
Finally,
\begin{equation} \label{pi_R}
\teval\tau=(\ell\TUpsilon\otimes\heval)(\Mloc\otimes\mathrm{id})\Delta_r.
\end{equation}
We also define another evaluation map $\eval^R $ that is, in some way, blind to the decoration $\mfv$ by $(\eval^R\tau)[\phi]=(\teval\tau)[\phi,\phi]$. We recall that $\ell$ vanishes on trees of non-negative degree. We recall that then $\ell$ leaves only a finite number of terms, as the equation is assumed subcritical.
\end{definition}

This definition is not recursive in the sense of the flow (but it is, of course, recursive in the depth of the tree). It is where we clearly differ from the previous works. We will explain with several arguments below why this definition is compatible with the flow recursion and thus that it matches what has been done until now.

\begin{example}\label{exeval}
Let us give a very simple example of computation of $\teval$ and $\eval^R$. We take the very special case where $g=g(\phi)$ and $f=f(\phi)$, in dimension $1$.
\begin{equs}
~&\Biggl(\teval\begin{tikzpicture}[scale=0.2,baseline=0.1cm]
			\node at (0,0)  [dot,fill=blue,label= {[label distance=-0.2em]below: \scriptsize  $      $} ] (root) {};
			\node at (2,3)  [dot,label={[label distance=-0.2em]right: \scriptsize  $  $}] (right) {};
			\node at (-2,3)  [dot,label={[label distance=-0.2em]above: \scriptsize  $ $} ] (left) {};
			\node at (2,5.5)  [dot,label={[label distance=-0.2em]above: \scriptsize  $ $} ] (rightc) {};
			\node at (-2,5.5)  [dot,label={[label distance=-0.2em]above: \scriptsize  $ $} ] (leftc) {};
			\draw[kernel1] (right) to
			node [sloped,below] {\small }     (root); 
			\draw[kernel1] (right) to
			node [sloped,below] {\small }     (rightc); 
			\draw[kernel1] (left) to
			node [sloped,below] {\small }     (root);
			\draw[kernel1] (leftc) to
			node [sloped,below] {\small }     (left);
			\node at (2,4.25) [fill=white,label={[label distance=0em]center: \scriptsize  $ \Xi $} ] () {};
			\node at (-1.25,1.5) [fill=white,label={[label distance=0em]center: \scriptsize  $ 0 $} ] () {};
			\node at (1.25,1.5) [fill=white,label={[label distance=0em]center: \scriptsize   $ 0  $} ] () {};
			\node at (-2,4.25) [fill=white,label={[label distance=0em]center: \scriptsize  $ \Xi $} ] () {};
		\end{tikzpicture}\Biggl)[\phi,\Tilde{\phi}]=\D^2g(\phi)\bigl((G-G_\mu)*(f(\Tilde{\phi})\xi_\eps)\bigl)^2+\ell\Biggl(\begin{tikzpicture}[scale=0.15,baseline=0.1cm]
			\node at (0,0)  [dot,label= {[label distance=-0.2em]below: \scriptsize  $      $} ] (root) {};
			\node at (2,3)  [dot,label={[label distance=-0.2em]right: \scriptsize  $  $}] (right) {};
			\node at (-2,3)  [dot,label={[label distance=-0.2em]above: \scriptsize  $ $} ] (left) {};
			\node at (2,5.5)  [dot,label={[label distance=-0.2em]above: \scriptsize  $ $} ] (rightc) {};
			\node at (-2,5.5)  [dot,label={[label distance=-0.2em]above: \scriptsize  $ $} ] (leftc) {};
			\draw[kernel1] (right) to
			node [sloped,below] {\small }     (root); 
			\draw[kernel1] (right) to
			node [sloped,below] {\small }     (rightc); 
			\draw[kernel1] (left) to
			node [sloped,below] {\small }     (root);
			\draw[kernel1] (leftc) to
			node [sloped,below] {\small }     (left);
			\node at (2,4.25) [fill=white,label={[label distance=0em]center: \scriptsize  $ \Xi $} ] () {};
			\node at (-1.25,1.5) [fill=white,label={[label distance=0em]center: \scriptsize  $ 0 $} ] () {};
			\node at (1.25,1.5) [fill=white,label={[label distance=0em]center: \scriptsize   $ 0  $} ] () {};
			\node at (-2,4.25) [fill=white,label={[label distance=0em]center: \scriptsize  $ \Xi $} ] () {};
		\end{tikzpicture}\Biggl)\D^2g(\phi)f(\Tilde{\phi})^2\\
		&~~~~+\ell\Biggl(\begin{tikzpicture}[scale=0.15,baseline=0.1cm]
			\node at (0,0)  [dot,label= {[label distance=-0.2em]below: \scriptsize  $      $} ] (root) {};
			\node at (0,0)  [dot,label= {[label distance=-0.25em]left: \scriptsize  $    $} ] (center) {};
			\node at (2,3)  [dot,label={[label distance=-0.2em]right: \scriptsize  $1  $}] (right) {};
			\node at (-2,3)  [dot,label={[label distance=-0.2em]above: \scriptsize  $ $} ] (left) {};
			\node at (2,5.5)  [dot,label={[label distance=-0.2em]above: \scriptsize  $ $} ] (rightc) {};
			\node at (-2,5.5)  [dot,label={[label distance=-0.2em]above: \scriptsize  $ $} ] (leftc) {};
			\draw[kernel1] (right) to
			node [sloped,below] {\small }     (root); 
			\draw[kernel1] (right) to
			node [sloped,below] {\small }     (rightc); 
			\draw[kernel1] (left) to
			node [sloped,below] {\small }     (root);
			\draw[kernel1] (leftc) to
			node [sloped,below] {\small }     (left);
			\node at (2,4.25) [fill=white,label={[label distance=0em]center: \scriptsize  $ \Xi $} ] () {};
			\node at (-1.25,1.5) [fill=white,label={[label distance=0em]center: \scriptsize  $ 0 $} ] () {};
			\node at (1.25,1.5) [fill=white,label={[label distance=0em]center: \scriptsize   $ 0  $} ] () {};
			\node at (-2,4.25) [fill=white,label={[label distance=0em]center: \scriptsize  $ \Xi $} ] () {};
		\end{tikzpicture}\Biggl)\D^2g(\phi)\D\Tilde{\phi}\D f(\Tilde{\phi})f(\Tilde{\phi})+\ell\Biggl(\begin{tikzpicture}[scale=0.15,baseline=0.1cm]
			\node at (0,0)  [dot,label= {[label distance=-0.2em]below: \scriptsize  $      $} ] (root) {};
			\node at (0,0)  [dot,label= {[label distance=-0.25em]left: \scriptsize  $    $} ] (center) {};
			\node at (2,3)  [dot,label={[label distance=-0.2em]right: \scriptsize  $1$}] (right) {};
			\node at (-2,3)  [dot,label={[label distance=-0.2em]left: \scriptsize  $1$} ] (left) {};
			\node at (2,5.5)  [dot,label={[label distance=-0.2em]above: \scriptsize  $ $} ] (rightc) {};
			\node at (-2,5.5)  [dot,label={[label distance=-0.2em]above: \scriptsize  $ $} ] (leftc) {};
			\draw[kernel1] (right) to
			node [sloped,below] {\small }     (root); 
			\draw[kernel1] (right) to
			node [sloped,below] {\small }     (rightc); 
			\draw[kernel1] (left) to
			node [sloped,below] {\small }     (root);
			\draw[kernel1] (leftc) to
			node [sloped,below] {\small }     (left);
			\node at (2,4.25) [fill=white,label={[label distance=0em]center: \scriptsize  $ \Xi $} ] () {};
			\node at (-1.25,1.5) [fill=white,label={[label distance=0em]center: \scriptsize  $ 0 $} ] () {};
			\node at (1.25,1.5) [fill=white,label={[label distance=0em]center: \scriptsize   $ 0  $} ] () {};
			\node at (-2,4.25) [fill=white,label={[label distance=0em]center: \scriptsize  $ \Xi $} ] () {};
		\end{tikzpicture}\Biggl)\D^2g(\phi)\big(\D\Tilde{\phi}\D f(\Tilde{\phi})\big)^2+\dots
\end{equs}
Note that there is no term with polynomial decoration at the root because of invariance by translation.
\end{example}

In this example, the interpretation of $\Tilde{\phi}$ becomes clear. If we compose the expression above in a Fréchet derivative in $\phi$ (as set in definition \ref{defb}), the functionals in $\Tilde{\phi}$ are going to be ignored, which is the analytic counterpart of the decoration $\mfv$, which authorizes grafting or not.

\subsection{Commutation with the derivative}

\begin{proposition}\label{commutation}
The following commutation relations hold
\begin{equation}
\D_\mu\teval=\teval\dmu,
\end{equation}
and on the multiplicative evaluation map
\begin{equation}
\D_\mu\heval=\heval\dmu.
\end{equation}
\end{proposition}

\begin{proof}
We prove the two statements at once by induction on the size of the tree considered. The base case on elementaty trees is immediate since both sides of the equality are zero. From equation \eqref{pi_R}, one has
\begin{equs}
\D_\mu\teval  = (\ell\TUpsilon\otimes \D_\mu \heval)(\Mloc\otimes\mathrm{id})\Delta_r
\end{equs}
Then, by the induction hypothesis $ \D_\mu\eval^{R\times}=\eval^{R\times}\dmu $, one gets
\begin{equs}
	\D_\mu\teval  & = (\ell\TUpsilon\otimes  \heval \dmu)(\Mloc\otimes\mathrm{id})\Delta_r
	\\ & = (\ell\TUpsilon\otimes  \heval )(\Mloc\otimes \dmu)\Delta_r
	\\ & = (\ell\TUpsilon\otimes  \heval )(\Mloc\otimes\mathrm{id})\Delta_r  \dmu
	\\ & =\teval\dmu
\end{equs}
where we have used \eqref{delta_r_mu} for moving from the second to the third line.
Then, one has
\begin{equs}
 \D_\mu	\heval(\tau\sigma) & =   \D_\mu \heval\tau \times\heval\sigma + \heval\tau \times  \D_\mu \heval\sigma
 \\ & =  \heval \dmu \tau \times\heval\sigma + \heval\tau \times   \heval \dmu  \sigma
 \\ & = \heval (\dmu (\tau \sigma))
\end{equs}
where we have used the induction hypothesis in the second line. We have also used the fact that both $ \D_{\mu} $ and $\dmu$ satisfy the Leibniz rule.
To finish, one gets
\begin{equs}
\D_\mu	\heval\mcI_a(\tau)& = \D_\mu \D^a(G-G_\mu)*\eval^R\tau +  \D^a(G-G_\mu)*(\D_\mu \eval^R\tau)
\\ & = -\D^a\dot{G}_\mu*\eval^R\tau +  \D^a(G-G_\mu)*( \eval^R \dmu\tau)
\\ & = \heval\mcI_a'(\tau) +  \heval\mcI_a(\dmu \tau)
\\ &= \heval \dmu \mcI_a(\tau)
\end{equs}
where we have used the induction hypothesis $  \D_\mu \teval\tau =  \teval \dmu \tau$ on the second line. This concludes the proof.
\end{proof}

\begin{remark}\label{comparison}
We exhibit a parallel between arguments displayed in this section and those on the diagram-free method in regularity structures developed in \cite{BN23}. In this paper, the authors consider trees that share many similarities with the present setting. The particularity is that the noises, denoted $\Xi_0$ therein, can be changed to another noise $\Xi_1$ via the use of an abstract derivative $D_\Xi$, which represents combinatorially a Malliavin derivative with respect to the noise. This latter looks very much like our $\dmu$ with the difference that it acts on the noises rather than on the edges. One can make a striking comparison between Proposition \ref{commutation} and Theorem 4.1 there that reads
\begin{equation}
\delta\Pi_x^{R,1}=\Pi_x^{R,1}D_\Xi,
\end{equation}
where $\Pi_x^{R,1}$ is a specific model. This hints a similarity in the renormalisation procedures: on the one hand, one regularizes the noise associated to a node with $D_\Xi$ and then renormalises the rest of the tree and on the other hand, taking the derivative in $\mu$ of an edge allows to renormalise the two parts of the tree independently. This permits an inductive argument in both cases.
\end{remark}

\subsection{Pre-Lie morphism property}

We present a property of pre-Lie morphism involving bilinear operators $B_a$, defined below, and the grafting operations $\rgraft_a$. An analogue in the context of regularity structures would be the pre-Lie morphism property on the elementary differentials. However, in the present case, there is no clear separation between the elementary differentials and the models, since the functionals are not localised. Arguments in this direction have already been exposed in Chandra and Ferdinand's works \cite{CF24a,CF24b}. If these objects are central in the previous works on the flow approach, their use is quite simple because they are used to define the stochastic terms. The subtlety here is to understand how the Fréchet derivatives act not only on the first term of the evaluated trees (the one representing the non-renormalised object), but also on the terms coming from both the renormalisation and the localisation. Note also that the use of a Fréchet derivative is necessary to keep track of where one wants to graft in a tree.

\begin{definition}\label{defb}
For every $a\in\N^d$, we define a bilinear map $B_a$ by
\begin{equation}
B_a(X,Y)[\phi,\Tilde{\phi}]=(D_{\D^a\phi}Y[\phi,\Tilde{\phi}])\cdot(-\D^a\dot G_\mu)*X[\phi,\Tilde{\phi}]. 
\end{equation}
We recall  that $D_{\D^a\phi}$ is the Fréchet derivative with respect to $\D^a\phi$ only. We further define $B=\sum_{a\in\N^{d+1}}B_a$, keeping in mind that this sum is, in fact, finite, so that 
\begin{equation}
B(X,Y)[\phi,\Tilde{\phi}]=(DY[\phi,\Tilde{\phi}])\cdot(-\dot G_\mu)*X[\phi,\Tilde{\phi}],
\end{equation}
with $D$ being the true Fréchet derivative in $\phi$. We finally define
\begin{equation}
B(X,Y)[\phi]=B(X,Y)[\phi,\phi].
\end{equation}
\end{definition}
\begin{proposition}\label{graft}
We have
\begin{equation}
B\big(\teval\sigma,\teval\tau\big)[\phi,\Tilde{\phi}]=\sum_{a\in\N^{d+1}}\teval(\sigma\rgraft_a\tau)[\phi,\Tilde{\phi}]
\end{equation}
In the particular case where $\tau\in\mcT_0^*$, \textit{i.e.} we can graft everywhere, this gives 
\begin{equation}
B\big(\eval^R\sigma,\eval^R\tau\big)[\phi]=\sum_{a\in\N^{d+1}}\eval^R(\sigma\rgraft_a\tau)[\phi].
\end{equation}
\end{proposition}

\begin{proof}
We start with the four base cases.\\
\textit{Case 1}. $\tau=Y^0\one$. We have $\sigma\rgraft_a Y^0\one=0$ on one side. On the other side, we have $(\teval Y^0\one)[\phi,\Tilde{\phi}]=g(\Tilde{\phi},\dots)$ so that its Fréchet derivative in $\phi$ is zero.\\
\textit{Case 2}. $\tau=Y^0\Xi$. The argument is the same as above.\\
\textit{Case 3}. $\tau=Y^1\one$. We have $\sigma\rgraft_aY^1\one=Y^1\mcI'_a(\sigma)$. We remind that, forgetting $Y^1$, 
\begin{equation*}
\Delta_r(\mcI'_a(\sigma))=\sum_{k\in\N^{d+1}}\frac{1}{k!}\CIg_{a}(X^k) \otimes\mcI'_{a+k}(\sigma),
\end{equation*}
which yields, as $\ell(\one)=1$ and $\ell(X^k)=0$ for $k\neq0$,
\begin{equation*}
(\teval\mcI'_a(\sigma))[\phi,\Tilde{\phi}]=\D_a g(\phi,\dots,\nabla^q\phi)\times(-\dot G_\mu)*\teval\sigma,
\end{equation*}
which is what is given by the right-hand side as $\teval(Y^1\one)=g(\phi,\dots)$.\\
\textit{Case 4}. $\tau=Y^1\Xi$. We have, following the same arguments as above,
\begin{equation*}
\Delta_r\Xi=\Xig \otimes\Xi+\Xi \otimes\one.
\end{equation*}
We get
\begin{equation*}
(\teval\Xi)[\phi,\Tilde{\phi}]=f(\phi,\dots,\nabla^p\phi)\xi_\eps+\sum_{\deg(X^k\Xi)<0}\frac{1}{k!}\ell(X^k\Xi)\D^k(f\circ\phi).
\end{equation*}
The result is obtained following the same lines as the end of this proof, that we will not rewrite in this particular case for conciseness. Let us first prove by induction that 
\begin{equation}
\sum_{a\in\N^{d+1}}\heval(\sigma\rgraft_a^{\text{\tiny{non-root}}}\tau)=\sum_{a\in\N^{d+1}}B_a(\teval\sigma,\heval\tau).
\end{equation}
We have successively
\begin{equs}
&\sum_{a\in\N^{d+1}}B_a(\teval\sigma,\heval\mcI_b(\tau))=\D^b(G-G_\mu)*\sum_{a\in\N^{d+1}}B_a(\teval\sigma,\teval\tau)\\
&=\D^b(G-G_\mu)*\sum_{a\in\N^{d+1}}\teval(\sigma\rgraft_a\tau)=\sum_{a\in\N^{d+1}}\heval(\mcI_b(\sigma\rgraft_a\tau))\\
&=\sum_{a\in\N^{d+1}}\heval(\sigma\rgraft_a^{\text{\tiny{non-root}}}\mcI_b(\tau)).
\end{equs} 
And
\begin{equs}
&\sum_{a\in\N^{d+1}}B_a(\teval\sigma,\heval(\tau\mu))\\
&=\sum_{a\in\N^{d+1}}B_a(\teval\sigma,\heval\tau)\times\heval\mu+\heval\tau\times\sum_{a\in\N^{d+1}} B_a(\eval^R\sigma,\heval\mu)\\
&=\sum_{a\in\N^{d+1}}\heval(\sigma\rgraft_a^{\text{\tiny{non-root}}}\tau)\times\heval\mu+\heval\tau\times\sum_{a\in\N^{d+1}}\heval(\sigma\rgraft_a^{\text{\tiny{non-root}}}\mu)\\
&=\sum_{a\in\N^{d+1}}\heval(\sigma\rgraft_a^{\text{\tiny{non-root}}}(\tau\mu)).
\end{equs}

With this statement in hand, we can finish the proof. We have, using Proposition \ref{lemmagraft} and Lemma \ref{commutggraft},
\begin{equs}
&\sum_{a\in\N^{d+1}}(\Mloc\otimes\mathrm{id})\Delta_r(\sigma\rgraft_a\tau)\\
&=\sum_{a\in\N^{d+1}}\sum_{k\in\N^{d+1}}\frac{1}{k!}(\Mloc\otimes\mathrm{id})\bigl(\ggraft_{a,k}\otimes(\sigma\rgraft_{a+k}^{\text{\tiny{root}}}\cdot)\bigl)\Delta_r\tau\\
&~~~~+\sum_{a\in\N^{d+1}}\bigl(\mathrm{id}\otimes(\sigma\rgraft_a^{\text{\tiny{non-root}}}\cdot)\bigl)\Delta_r\tau\\
&=\sum_{k\in\N^{d+1}}\sum_{a\in\N^{d+1}+k}\frac{1}{k!}(\Mloc\otimes\mathrm{id}))\bigl(\ggraft_{a-k,k}\otimes(\sigma\rgraft_{a}^{\text{\tiny{root}}}\cdot)\bigl)\Delta_r\tau\\
&~~~~+\sum_{a\in\N^{d+1}}\bigl(\mathrm{id}\otimes(\sigma\rgraft_a^{\text{\tiny{non-root}}}\cdot)\bigl)\Delta_r\tau\\
&=\sum_{a\in\N^{d+1}}\left(\left(\Mloc\sum_{k\in\N^{d+1}-a}\frac{1}{k!}\ggraft_{a-k,k}\right)\otimes(\sigma\rgraft_{a}^{\text{\tiny{root}}}\cdot)\right)\Delta_r\tau\\
&~~~~+\sum_{a\in\N^{d+1}}\bigl(\mathrm{id}\otimes(\sigma\rgraft_a^{\text{\tiny{non-root}}}\cdot)\bigl)\Delta_r\tau\\
&=\sum_{a\in\N^{d+1}}\bigl(\ggraft_{a}\Mloc\otimes(\sigma\rgraft_{a}^{\text{\tiny{root}}}\cdot)\bigl)\Delta_r\tau+\bigl(\mathrm{id}\otimes(\sigma\rgraft_a^{\text{\tiny{non-root}}}\cdot)\bigl)\Delta_r\tau
\end{equs}
It yields, using Lemma \ref{lemmaupsilon}, in Sweedler's notation,
\begin{equs}
\sum_{a\in\N^{d+1}}\teval(\sigma\rgraft_a\tau)&=\sum_{a\in\N^{d+1}}\sum_{(\tau)}\ell\TUpsilon[\ggraft_a\Mloc\tau_1]\heval(\mcI'_a(\sigma)\tau_2)\\
&~~~~+\sum_{a\in\N^{d+1}}\sum_{(\tau)}\ell\TUpsilon[\Mloc\tau_1]\teval(\sigma\rgraft_a^{\text{\tiny{non-root}}}\tau_2)\\
&=\sum_{a\in\N^{d+1}}\sum_{(\tau)}D_{\D^a\phi}\ell\TUpsilon[\Mloc\tau_1](\D^a\dot G_\mu*\teval\sigma)\times\heval\tau_2\\
&~~~~+\sum_{a\in\N^{d+1}}\sum_{(\tau)}\ell\TUpsilon[\Mloc\tau_1]B_a(\teval\sigma,\heval\tau_2)\\
&=\sum_{a\in\N^{d+1}}\sum_{(\tau)}B_a(\teval\sigma,\ell\TUpsilon[\Mloc\tau_1])\times\heval\tau_2\\
&~~~~+\sum_{a\in\N^{d+1}}\sum_{(\tau)}\ell\TUpsilon[\Mloc\tau_1]B_a(\teval\sigma,\heval\tau_2)\\
&=\sum_{a\in\N^{d+1}}B_a\left(\teval\sigma,\sum_{(\tau)}\ell\TUpsilon[\Mloc\tau_1]\heval\tau_2\right)\\
&=\sum_{a\in\N^{d+1}}B_a(\teval\sigma,\teval\tau).
\end{equs}
\end{proof}

\subsection{The flow equation on the coefficients}\label{seccoeff}

The purpose of this subsection is to explain how our definition of the evaluation map allows to recover the flow equation on the coefficients, that is usually given as a definition. However, before being able to give a clear statement, we need the following map. It is designed to deal with the fact that in Proposition \ref{graft}, the equality cannot be separated index by index (in the sense that we do not have $B_a(\teval\sigma,\teval\tau)[\phi,\Tilde{\phi}]=\teval(\sigma\rgraft_a\tau)[\phi,\Tilde{\phi}]$ but only the result with a summation in $a$, because of a shift on indices in the proof above), so that, given a tree in $T_0^0$, we need to sum on all the edges decorations. Note that for equations with a simple form such as $\phi^4_3$, KPZ, or PAM, this does not appear.

\begin{definition} We define a linear map $\Uparrow:\mcT_0^0\to \mcT_0$, defined inductively on $T_0^0$ by being the identity on $\one$ and $\Xi$, and
\begin{equation}
\begin{split}
\Uparrow\mcI_0(\tau)=\sum_{a\in\N^{d+1}}\mcI_a(\Uparrow\tau),&~~~~\Uparrow\mcI'_0(\tau)=\sum_{a\in\N^{d+1}}\mcI'_a(\Uparrow\tau),\\
\Uparrow\sigma\tau&=\Uparrow\sigma\times\Uparrow\tau,
\end{split}
\end{equation}
and extended by linearity to $\mcT_0^0$.
\end{definition}

\begin{example}We have
\begin{equation*}
\Uparrow\begin{tikzpicture}[scale=0.2,baseline=0.1cm]
			\node at (0,0)  [dot,label= {[label distance=-0.2em]below: \scriptsize  $      $} ] (root) {};
			\node at (2,3)  [dot,label={[label distance=-0.2em]right: \scriptsize  $  $}] (right) {};
			\node at (-2,3)  [dot,label={[label distance=-0.2em]above: \scriptsize  $ $} ] (left) {};
			\node at (2,5.5)  [dot,label={[label distance=-0.2em]above: \scriptsize  $ $} ] (rightc) {};
			\node at (-2,5.5)  [dot,label={[label distance=-0.2em]above: \scriptsize  $ $} ] (leftc) {};
			\draw[kernel1] (right) to
			node [sloped,below] {\small }     (root); 
			\draw[kernel1] (right) to
			node [sloped,below] {\small }     (rightc); 
			\draw[kernel1,color=red] (left) to
			node [sloped,below] {\small }     (root);
			\draw[kernel1] (leftc) to
			node [sloped,below] {\small }     (left);
			\node at (2,4.25) [fill=white,label={[label distance=0em]center: \scriptsize  $ \Xi $} ] () {};
			\node at (-1.25,1.5) [fill=white,label={[label distance=0em]center: \scriptsize  $ 0 $} ] () {};
			\node at (1.25,1.5) [fill=white,label={[label distance=0em]center: \scriptsize   $ 0  $} ] () {};
			\node at (-2,4.25) [fill=white,label={[label distance=0em]center: \scriptsize  $ \Xi $} ] () {};
		\end{tikzpicture}=\sum_{a,b\in\N^{d+1}}\begin{tikzpicture}[scale=0.2,baseline=0.1cm]
			\node at (0,0)  [dot,label= {[label distance=-0.2em]below: \scriptsize  $      $} ] (root) {};
			\node at (2,3)  [dot,label={[label distance=-0.2em]right: \scriptsize  $  $}] (right) {};
			\node at (-2,3)  [dot,label={[label distance=-0.2em]above: \scriptsize  $ $} ] (left) {};
			\node at (2,5.5)  [dot,label={[label distance=-0.2em]above: \scriptsize  $ $} ] (rightc) {};
			\node at (-2,5.5)  [dot,label={[label distance=-0.2em]above: \scriptsize  $ $} ] (leftc) {};
			\draw[kernel1] (right) to
			node [sloped,below] {\small }     (root); 
			\draw[kernel1] (right) to
			node [sloped,below] {\small }     (rightc); 
			\draw[kernel1,color=red] (left) to
			node [sloped,below] {\small }     (root);
			\draw[kernel1] (leftc) to
			node [sloped,below] {\small }     (left);
			\node at (2,4.25) [fill=white,label={[label distance=0em]center: \scriptsize  $ \Xi $} ] () {};
			\node at (-1.25,1.5) [fill=white,label={[label distance=0em]center: \scriptsize  $ a $} ] () {};
			\node at (1.25,1.5) [fill=white,label={[label distance=0em]center: \scriptsize   $ b  $} ] () {};
			\node at (-2,4.25) [fill=white,label={[label distance=0em]center: \scriptsize  $ \Xi $} ] () {};
		\end{tikzpicture}.
\end{equation*}
Note that, in practice, the sum on the right-hand side will always be finite, as the functions $g$ and $f$ depend only on a finite number of derivatives of $\phi$.
\end{example}

We state two propositions that highlight how this new map behaves with some of the objects we used in the previous sections.

\begin{lemma} We have, on $\mcT_0^0$,
\begin{equation}
\Uparrow\dmu=\dmu\Uparrow,
\end{equation}
as well as, for $\tau,\sigma\in\mcT_0^0$,
\begin{equation}
\Uparrow(\sigma\rgraft_0\tau)=\sum_a\Uparrow\sigma\rgraft_a\Uparrow\tau
\end{equation}
\end{lemma}

\begin{proof}
\textit{First statement}
\begin{equs}
\dmu\Uparrow\mcI_0(\tau)&=\dmu\sum_a\mcI_a(\Uparrow\tau)=\sum_a\mcI'_a(\Uparrow\tau)+\mcI_a(\dmu\Uparrow\tau)\\
&=\sum_a\mcI'_a(\Uparrow\tau)+\mcI_a(\Uparrow\dmu\tau)=\Uparrow(\mcI'_0(\tau)+\mcI_0(\dmu\tau))\\
&=\Uparrow\dmu\mcI_0(\tau).
\end{equs}
and
\begin{equs}
\dmu\Uparrow(\sigma\tau)&=\dmu(\Uparrow\sigma\times\Uparrow\tau)=(\dmu\Uparrow\sigma)\times\Uparrow\tau+\Uparrow\sigma\times(\dmu\Uparrow\tau)\\
&=(\Uparrow\dmu\sigma)\times\Uparrow\tau+\Uparrow\sigma\times(\Uparrow\dmu\tau)=\Uparrow(\dmu\sigma\times\tau+\sigma\times\dmu\tau)\\
&=\Uparrow\dmu(\sigma\tau).
\end{equs}

\noindent\textit{Second statement}

\begin{equs}
\Uparrow(\sigma\rgraft_0\mcI_0(\tau))&=\Uparrow(\mcI'_0(\sigma)\mcI_0(\tau)+\mcI_0(\sigma\rgraft_0\tau))\\
&=\sum_{a,b}\mcI'_a(\Uparrow\sigma)\mcI_b(\Uparrow\tau)+\sum_b\mcI_b(\Uparrow(\sigma\rgraft_0\tau))\\
&=\sum_{a,b}\mcI'_a(\Uparrow\sigma)\mcI_b(\Uparrow\tau)+\sum_{a,b}\mcI_b(\Uparrow\sigma\rgraft_a\Uparrow\tau))\\
&=\sum_{a}\mcI'_a(\Uparrow\sigma)\Uparrow\mcI_0(\tau)+\sum_{a}\Uparrow\sigma\rgraft_a^{\text{\tiny{non-root}}}\left(\sum_b\mcI_b(\Uparrow\tau)\right)\\
&=\sum_a\Uparrow\sigma\rgraft_a\Uparrow\mcI_0(\tau).
\end{equs}
and
\begin{equs}
\Uparrow(\sigma\rgraft_0(\tau\mu))&=\Uparrow(-\mcI'_0(\sigma)\tau\mu+(\sigma\rgraft_0\tau)\times\mu+\tau\times(\sigma\rgraft_0\mu))\\
&=-\sum_a\mcI'_a(\Uparrow\sigma)\Uparrow\tau\Uparrow\mu+\left(\sum_a\Uparrow\sigma\rgraft_a\Uparrow\tau\right)\times\Uparrow\mu\\
&~~~~+\Uparrow\tau\times\left(\sum_a\Uparrow\sigma\rgraft_a\Uparrow\mu\right)\\
&=\sum_a\Uparrow\sigma\rgraft_a(\Uparrow\tau\Uparrow\mu)\\
&=\sum_a\Uparrow\sigma\rgraft_a\Uparrow(\tau\mu).
\end{equs}
\end{proof}

With this two propositions in hand, we can state:

\begin{proposition}\label{flowcoeff}
For a tree $\tau\in T_0^0$ such that the decoration $\mfv$ is $0$ everywhere and that is not an elementary tree, we have
\begin{equation} \label{IBP}
\eval^R\Uparrow\tau=\int_0^\mu\sum_{(\tau)}\sum_aB_{a}(\Pi_{\eps,\nu}^R\Uparrow\tau',\Pi_{\eps,\nu}^R\Uparrow\tau'') d\nu+\Pi_{\eps,0}^R\Uparrow\tau.
\end{equation}
The Sweedler notation used here stands for the coproduct $\Delta_1$.
\end{proposition}

\begin{proof}

Using the Propositions \ref{flow}, \ref{commutation}, and \ref{graft}, we can write, for a such a tree $\tau$,
\begin{equs}
\eval^R\Uparrow\tau&=\int_0^\mu\D_\mu\Pi_{\eps,\nu}^R\Uparrow\tau d\nu+\Pi_{\eps,0}^R\Uparrow\tau\\
&=\int_0^\mu\Pi_{\eps,\nu}^R(\dmu\Uparrow\tau) d\nu+\Pi_{\eps,0}^R\Uparrow\tau\\
&=\int_0^\mu\Pi_{\eps,\nu}^R(\Uparrow\dmu\tau) d\nu+\Pi_{\eps,0}^R\Uparrow\tau\\
&=\int_0^\mu\sum_{(\tau)}\Pi_{\eps,\nu}^R(\Uparrow(\mcI_{0}(\tau')\graft\tau'')) d\nu+\Pi_{\eps,0}^R\Uparrow\tau\\
&=\int_0^\mu\sum_{(\tau)}\Pi_{\eps,\nu}^R(\Uparrow(\tau'\rgraft_{0}\tau'')) d\nu+\Pi_{\eps,0}^R\Uparrow\tau\\
&=\int_0^\mu\sum_{(\tau)}\sum_a\Pi_{\eps,\nu}^R(\Uparrow\tau'\rgraft_{a}\Uparrow\tau'') d\nu+\Pi_{\eps,0}^R\Uparrow\tau\\
&=\int_0^\mu\sum_{(\tau)}\sum_aB_{a}(\Pi_{\eps,\nu}^R\Uparrow\tau',\Pi_{\eps,\nu}^R\Uparrow\tau'') d\nu+\Pi_{\eps,0}^R\Uparrow\tau.
\end{equs}

We have used the Sweedler notation for a tree $\tau\in T_0^0$,
\begin{equation*}
\Delta_1\tau=\sum_{(\tau)}\mcI_0(\tau')\otimes\tau''.
\end{equation*}
\end{proof}
The content of the proposition is exactly the equation the previous works on the flow approach used to define the force coefficients. Therefore, we show that, in our context, this equation is a consequence of the definition of the evaluation map and not the other way around. Moreover, as the initial conditions (in $\mu=0$) match, it explains why our definition corresponds to the one of the previous works. Note that it is hierarchical in the size of the trees.	One can also note that the initial condition in $0$ vanishes for most of the trees. When it does not, it brings an additional counterterm for the renormalisation.\\
Even though this proposition is not used to proved the main results of the next section, we found it enlightening since it is the backbone of the method. It also shows why the decoration $\mfv$ is important if one wants to write a hierarchical equation. If one does not use this decoration, the following problem appears. Let us take the derivative of a simple tree.
\begin{equs}
\dmu\begin{tikzpicture}[scale=0.2,baseline=0.1cm]
			\node at (0,0)  [dot,label= {[label distance=-0.2em]below: \scriptsize  $      $} ] (root) {};
			\node at (2,3)  [dot,label={[label distance=-0.2em]right: \scriptsize  $  $}] (right) {};
			\node at (-2,3)  [dot,label={[label distance=-0.2em]above: \scriptsize  $ $} ] (left) {};
			\node at (2,5.5)  [dot,label={[label distance=-0.2em]above: \scriptsize  $ $} ] (rightc) {};
			\node at (-2,5.5)  [dot,label={[label distance=-0.2em]above: \scriptsize  $ $} ] (leftc) {};
			\draw[kernel1] (right) to
			node [sloped,below] {\small }     (root); 
			\draw[kernel1] (right) to
			node [sloped,below] {\small }     (rightc); 
			\draw[kernel1] (left) to
			node [sloped,below] {\small }     (root);
			\draw[kernel1] (leftc) to
			node [sloped,below] {\small }     (left);
			\node at (2,4.25) [fill=white,label={[label distance=0em]center: \scriptsize  $ \Xi $} ] () {};
			\node at (-1.25,1.5) [fill=white,label={[label distance=0em]center: \scriptsize  $ 0 $} ] () {};
			\node at (1.25,1.5) [fill=white,label={[label distance=0em]center: \scriptsize   $ 0  $} ] () {};
			\node at (-2,4.25) [fill=white,label={[label distance=0em]center: \scriptsize  $ \Xi $} ] () {};
		\end{tikzpicture}=2\begin{tikzpicture}[scale=0.2,baseline=0.1cm]
			\node at (0,0)  [dot,label= {[label distance=-0.2em]below: \scriptsize  $      $} ] (root) {};
			\node at (2,3)  [dot,label={[label distance=-0.2em]right: \scriptsize  $  $}] (right) {};
			\node at (-2,3)  [dot,label={[label distance=-0.2em]above: \scriptsize  $ $} ] (left) {};
			\node at (2,5.5)  [dot,label={[label distance=-0.2em]above: \scriptsize  $ $} ] (rightc) {};
			\node at (-2,5.5)  [dot,label={[label distance=-0.2em]above: \scriptsize  $ $} ] (leftc) {};
			\draw[kernel1] (right) to
			node [sloped,below] {\small }     (root); 
			\draw[kernel1] (right) to
			node [sloped,below] {\small }     (rightc); 
			\draw[kernel1,color=red] (left) to
			node [sloped,below] {\small }     (root);
			\draw[kernel1] (leftc) to
			node [sloped,below] {\small }     (left);
			\node at (2,4.25) [fill=white,label={[label distance=0em]center: \scriptsize  $ \Xi $} ] () {};
			\node at (-1.25,1.5) [fill=white,label={[label distance=0em]center: \scriptsize  $ 0 $} ] () {};
			\node at (1.25,1.5) [fill=white,label={[label distance=0em]center: \scriptsize   $ 0  $} ] () {};
			\node at (-2,4.25) [fill=white,label={[label distance=0em]center: \scriptsize  $ \Xi $} ] () {};
		\end{tikzpicture},
\end{equs}
which suggests that it is split into two trees via $\Delta_1$: $\begin{tikzpicture}[scale=0.2,baseline=0.1cm]
			\node at (0,0)  [dot,label={[label distance=-0.2em]above: \scriptsize  $ $} ] (left) {};
			\node at (0,2.5)  [dot,label={[label distance=-0.2em]above: \scriptsize  $ $} ] (leftc) {};
			\draw[kernel1] (leftc) to
			node [sloped,below] {\small }     (left);
			\node at (0,1.25) [fill=white,label={[label distance=0em]center: \scriptsize  $ \Xi $} ] () {};
			\end{tikzpicture}$ and $\begin{tikzpicture}[scale=0.2,baseline=0.1cm]
			\node at (0,0)  [dot,label= {[label distance=-0.2em]below: \scriptsize  $      $} ] (root) {};
			\node at (0,3)  [dot,label={[label distance=-0.2em]above: \scriptsize  $ $} ] (left) {};
			\node at (0,5.5)  [dot,label={[label distance=-0.2em]above: \scriptsize  $ $} ] (leftc) {};
			\draw[kernel1] (left) to
			node [sloped,below] {\small }     (root); 
			\draw[kernel1] (leftc) to
			node [sloped,below] {\small }     (left);
			\node at (0,1.5) [fill=white,label={[label distance=0em]center: \scriptsize  $ 0 $} ] () {};
			\node at (0,4.25) [fill=white,label={[label distance=0em]center: \scriptsize  $ \Xi $} ] () {};
		\end{tikzpicture}$. However, when one grafts them back without paying attention to $\mfv$, we get
\begin{equs}
\begin{tikzpicture}[scale=0.2,baseline=0.1cm]
			\node at (0,0)  [dot,label={[label distance=-0.2em]above: \scriptsize  $ $} ] (left) {};
			\node at (0,2.5)  [dot,label={[label distance=-0.2em]above: \scriptsize  $ $} ] (leftc) {};
			\draw[kernel1] (leftc) to
			node [sloped,below] {\small }     (left);
			\node at (0,1.25) [fill=white,label={[label distance=0em]center: \scriptsize  $ \Xi $} ] () {};
			\end{tikzpicture}\rgraft_a\begin{tikzpicture}[scale=0.2,baseline=0.1cm]
			\node at (0,0)  [dot,label= {[label distance=-0.2em]below: \scriptsize  $      $} ] (root) {};
			\node at (0,3)  [dot,label={[label distance=-0.2em]above: \scriptsize  $ $} ] (left) {};
			\node at (0,5.5)  [dot,label={[label distance=-0.2em]above: \scriptsize  $ $} ] (leftc) {};
			\draw[kernel1] (left) to
			node [sloped,below] {\small }     (root); 
			\draw[kernel1] (leftc) to
			node [sloped,below] {\small }     (left);
			\node at (0,1.5) [fill=white,label={[label distance=0em]center: \scriptsize  $ 0 $} ] () {};
			\node at (0,4.25) [fill=white,label={[label distance=0em]center: \scriptsize  $ \Xi $} ] () {};
		\end{tikzpicture}=\begin{tikzpicture}[scale=0.2,baseline=0.1cm]
			\node at (0,0)  [dot,label= {[label distance=-0.2em]below: \scriptsize  $      $} ] (root) {};
			\node at (2,3)  [dot,label={[label distance=-0.2em]right: \scriptsize  $  $}] (right) {};
			\node at (-2,3)  [dot,label={[label distance=-0.2em]above: \scriptsize  $ $} ] (left) {};
			\node at (2,5.5)  [dot,label={[label distance=-0.2em]above: \scriptsize  $ $} ] (rightc) {};
			\node at (-2,5.5)  [dot,label={[label distance=-0.2em]above: \scriptsize  $ $} ] (leftc) {};
			\draw[kernel1] (right) to
			node [sloped,below] {\small }     (root); 
			\draw[kernel1] (right) to
			node [sloped,below] {\small }     (rightc); 
			\draw[kernel1,color=red] (left) to
			node [sloped,below] {\small }     (root);
			\draw[kernel1] (leftc) to
			node [sloped,below] {\small }     (left);
			\node at (2,4.25) [fill=white,label={[label distance=0em]center: \scriptsize  $ \Xi $} ] () {};
			\node at (-1.25,1.5) [fill=white,label={[label distance=0em]center: \scriptsize  $ 0 $} ] () {};
			\node at (1.25,1.5) [fill=white,label={[label distance=0em]center: \scriptsize   $ 0  $} ] () {};
			\node at (-2,4.25) [fill=white,label={[label distance=0em]center: \scriptsize  $ \Xi $} ] () {};
		\end{tikzpicture}+\begin{tikzpicture}[scale=0.2,baseline=0.1cm]
			\node at (0,0)  [dot,label= {[label distance=-0.2em]below: \scriptsize  $      $} ] (root) {};
			\node at (2,3)  [dot,label={[label distance=-0.2em]right: \scriptsize  $  $}] (right) {};
			\node at (0,6)  [dot,label={[label distance=-0.2em]above: \scriptsize  $ $} ] (left) {};
			\node at (3,5)  [dot,label={[label distance=-0.2em]above: \scriptsize  $ $} ] (rightc) {};
			\node at (0,8.5)  [dot,label={[label distance=-0.2em]above: \scriptsize  $ $} ] (leftc) {};
			\draw[kernel1] (right) to
			node [sloped,below] {\small }     (root); 
			\draw[kernel1] (right) to
			node [sloped,below] {\small }     (rightc); 
			\draw[kernel1,color=red] (left) to
			node [sloped,below] {\small }     (right);
			\draw[kernel1] (leftc) to
			node [sloped,below] {\small }     (left);
			\node at (2.5,4) [fill=white,label={[label distance=0em]center: \scriptsize  $ \Xi $} ] () {};
			\node at (1,4.5) [fill=white,label={[label distance=0em]center: \scriptsize  $ 0 $} ] () {};
			\node at (1,1.5) [fill=white,label={[label distance=0em]center: \scriptsize   $ 0  $} ] () {};
			\node at (0,7.25) [fill=white,label={[label distance=0em]center: \scriptsize  $ \Xi $} ] () {};
		\end{tikzpicture},
\end{equs}
and we get a new tree in addition of the original one. We circumvent this problem by adding the decoration $\mfv=1$ when cutting the trees (see Example \ref{exdelta1}). For simpler equations, a choice of ansatz with a rougher combinatorial description, as done for polynomial equations by Duch \cite{Duc21, Duc22}, tackles this issue without requiring to this trick, but then loses the tree-based framework.

\section{Main results}\label{mainresults}

In this section, we present three theorems that are a consequence of the properties developed above, constituting the main results of the paper. 

\subsection{Coherence}\label{coherence}

We turn to showing what we call a coherence property for the method we developed. We choose this designation accordingly to the theory introduced in \cite{BCCH}. In a few words, we explain why the ansatz we set is compatible with the flow equation. On the side of regularity structures, it is the result of Theorem 3.25 in the above-mentioned paper. However, we stress that the satisfied coherence property is not exactly the same. In regularity structures, one wants to check that the ansatz at the level of the modelled distribution representing the solution of the equation is compatible with the one at the level of the right-hand side of the SPDE, when plugged back into the non-linearity. On our side, the result is not as involved, and is also, by far, easier to prove than its counterpart in regularity structures.\\
We show in the following theorem, as a consequence of the Propositions \ref{commutation} and \ref{graft}, themselves coming from the non-recursive definition of $\eval^R$, that the ansatz solves the flow equation.

\begin{definition}
For a given $\gamma>0$, we denote $\mcQ_\gamma$ the projection in $\mcT$ on trees of degree less or equal than $\gamma$. Similarly, we denote $\mcP_\gamma$, defined on functionals, the projection on $\mathrm{Span}(\eval^R\tau,~\deg(\tau)\leq\gamma)$.
\end{definition}
We get the obvious lemma
\begin{lemma}
 We have 
 \begin{equation}
 \mcP_\gamma\eval^R=\eval^R\mcQ_\gamma.
 \end{equation}
\end{lemma}
We may write, allowing the formal series $F_\mu=F_\mu^\infty=\sum_{\tau\in T_0^*}\frac{1}{S(\tau)}\eval^R\tau$ for a moment, making the truncation clear,
\begin{equation}
\Fem^\gamma=\mcP_\gamma \Fem=\eval^R\mcQ_\gamma\sum_{\tau\in T_0^*}\frac{1}{S(\tau)}\tau.
\end{equation}

We are ready to state the main result of this section.

\begin{theorem} \label{coherence}
The map $F_\mu^\gamma$ satisfies the following truncated version of the Polchinski equation
\begin{equation}\label{truncflow}
\D_\mu\Fem^\gamma+\mcP_{\gamma-\lambda}(D\Fem^\gamma\cdot\dot G_\mu*\Fem^\gamma)=0.
\end{equation}
If moreover the series defining $\Fem^\gamma$ converges as $\gamma$ goes to infinity, $\Fem$ satisfies the true Polchinski equation
\begin{equation}
\partial_{\mu} \Fem  + D \Fem \cdot \dot{G}_{\mu} * \Fem=0.
\end{equation}
\end{theorem}

Before moving to the proof, let us add a few comments on this result. In his original paper \cite{Duc21}, Duch considers an ansatz that is defined by a series instead of a finite sum. It allows to give a version of the method where no remainder appears in the Polchinski equation. It is the choice that most of the works using the flow approach use, as it allows more freedom. This is, in fact, equivalent to truncating the sum by removing the so-called irrelevant terms and having a remainder in the flow equation. As, in this paper, we choose not to tackle in details of the convergence of such series, we prefer the version \eqref{truncflow} of the equation. The content of this theorem is not to be confused with Proposition \ref{flowcoeff}. Indeed, even though they look alike, the latter is backwards in the sense of the size of the trees. More precisely one takes the stochastic coefficient corresponding to a tree and splits it in smaller trees. In the present case one takes the full ansatz and grafts all the trees onto each other, thus increasing the size of the trees. It is also why we do not require the decoration $\mfv$ here.

\begin{proof}
	If we plug the ansatz into the left-hand side of the truncated Polchinski flow equation, we get
	\begin{equation*}
		\D_\mu \sum_{\substack{\tau\in T_0^*\\ \deg(\tau)\leq\gamma}}\frac{(\eval^R\tau)[\phi]}{S(\tau)}+\mcP_{\gamma-\lambda} \sum_{\substack{\tau_1\in T_0^*\\ \deg(\tau_1)\leq\gamma}}\frac{D (\eval^R\tau_1)[\phi]}{S(\tau_1)}\cdot\dot G_\mu*\sum_{\substack{\tau_2\in T_0^*\\ \deg(\tau_2)\leq\gamma}}\frac{(\eval^R\tau_2)[\phi]}{S(\tau_2)},
	\end{equation*}
which can be rewritten as
\begin{equation*}
		\D_\mu \sum_{\substack{\tau\in T_0^*\\ \deg(\tau)\leq\gamma}}\frac{(\eval^R\tau)[\phi]}{S(\tau)}- \mcP_{\gamma-\lambda}\sum_{\substack{\tau_1,\tau_2\in T_0^*\\ \deg(\tau_1)\vee\deg(\tau_2)\leq\gamma}} \sum_{a \in\N^{d+1}}\frac{1}{S(\tau_1)S(\tau_2)}  B_a(\eval^R\tau_2,\eval^R\tau_1).
	\end{equation*}
	Then, from Proposition \ref{commutation} and Proposition \ref{graft}, we rewrite it as
	\begin{equation*}
	\sum_{\substack{\tau\in T_0^*\\ \deg(\tau)\leq\gamma}}\frac{(\eval^R \uparrow_{\mu} \tau)[\phi]}{S(\tau)}- \mcP_{\gamma-\lambda}\sum_{\substack{\tau_1,\tau_2\in T_0^*\\ \deg(\tau_1)\vee\deg(\tau_2)\leq\gamma}} \sum_{a \in\N^{d+1}} \frac{1}{S(\tau_1)S(\tau_2)}  \eval^R(\tau_2\rgraft_a\tau_1).
	\end{equation*}
	
It suffices to show the equality
\begin{equation*}
\sum_{\substack{\tau\in T_0^*\\ \deg(\tau)\leq\gamma}}\frac{\dmu \tau}{S(\tau)}= \mcQ_{\gamma-\lambda}\sum_{\substack{\tau_1,\tau_2\in T_0^*\\ \deg(\tau_1)\vee\deg(\tau_2)\leq\gamma}} \sum_{a \in\N^{d+1}} \frac{\tau_2\rgraft_a\tau_1}{S(\tau_1)S(\tau_2)} .
\end{equation*}
Let us fix an arbitrary $\sigma\in T_0^*$ such that $\sigma\neq\one,\Xi$ and $\deg(\sigma)\leq\gamma$, and $e\in E_\sigma$. This choice of $\sigma$ prevents the use of indicator functions on the degree in the sequel. We get, on the left-hand side, using Lemma \ref{combidmu},
\begin{equation*}
\left\langle\sum_{\substack{\tau\in T_0^*\\ \deg(\tau)\leq\gamma}}\frac{\dmu \tau}{S(\tau)},\dmu^e\sigma\right\rangle=\left\langle\frac{\dmu\sigma}{S(\sigma)},\dmu^e\sigma\right\rangle=\frac{1}{S(\sigma)}\frac{S(\sigma)}{S(\dmu^e\sigma)}\langle\dmu^e\sigma,\dmu^e\sigma\rangle=1.
\end{equation*}
On the right-hand side, we have, using Lemma \ref{duality},
\begin{equs}
~&\left\langle\mcQ_{\gamma-\lambda}\sum_{\substack{\tau_1,\tau_2\in T_0^*\\ \deg(\tau_1)\vee\deg(\tau_2)\leq\gamma}}\sum_{a \in\N^{d+1}}  \frac{\tau_2\rgraft_a\tau_1}{S(\tau_1)S(\tau_2)},\dmu^e\sigma\right\rangle\\
&=\sum_{\substack{\tau_1,\tau_2\in T_0^*\\ \deg(\tau_1)\vee\deg(\tau_2)\leq\gamma}} \sum_{a \in\N^{d+1}} \frac{1}{S(\tau_1)S(\tau_2)} \langle\tau_2\rgraft_a\tau_1,\dmu^e\sigma\rangle\\
&=\sum_{\substack{\tau_1,\tau_2\in T_0^*\\ \deg(\tau_1)\vee\deg(\tau_2)\leq\gamma}} \sum_{a \in\N^{d+1}} \frac{1}{S(\tau_1)S(\tau_2)} \langle\tau_2\otimes\tau_1,\Delta_a\dmu^e\sigma\rangle.
\end{equs}
One can then write $\Delta_a\dmu^e\sigma=\one_{a=a^*}\sigma_1\otimes\sigma_2$ for a unique couple $(\sigma_1,\sigma_2)\in T_0\times T_0$ and $a^*=\mfe(e)$. It yields
\begin{align*}
\left\langle\sum_{\substack{\tau_1,\tau_2\in T_0^*\\ \deg(\tau_1)\vee\deg(\tau_2)\leq\gamma}} \sum_{a \in\N^{d+1}} \frac{\tau_2\rgraft_a\tau_1}{S(\tau_1)S(\tau_2)},\dmu^e\sigma\right\rangle=\frac{S(\sigma_1)S(\sigma_2)}{S(\sigma_1)S(\sigma_2)}=1,
\end{align*}
which concludes the proof.
\end{proof}

\subsection{Renormalisation at the level of the SPDE}\label{backspde}

In this section, we explain how the renormalisation procedures we have developed in the previous sections come back at the level of the SPDE \eqref{maineq}. It is the counterpart of one of the main results of \cite{BCCH} in the context of regularity structures. Nevertheless, the method of proof is by far simpler in the flow approach, as we use quite elementary arguments. Note that this result has also been proved with a recursive definition of the models with a duality argument in \cite{BB21b}. We choose not to use duality as, on the contrary of recursive models in regularity structures, we do not have a clear cut between the algebraic renormalisation and the evaluation of trees. Both procedures are intertwined. This simplification is noteworthy and advocating in favour of the flow approach. We forget for this section the decoration $\mfv$, since it will be set to $1$ everywhere and thus will not play any role. In other words, grafts can occur on every node, accordingly to the rules set in Section \ref{trees}. We start this section with the intuitive (in the view of Example \ref{exMloc}) lemma.

\begin{lemma} \label{localisation_prop}
	For a tree $\tau\in T_0$, there is a sequence of trees of $T_0^X$ such that
\begin{equation}
	\Mloc\tau = \sum_{i} \frac{S(\tau)}{S(\tau_i)}\tau_i.
\end{equation}
\end{lemma}

\begin{proof}
	We proceed by induction on the construction of the decroated tree $ \tau $. One has, for the base case,
	\begin{equs}
		\Mloc \zeta_l =  \sum_{k \in  \N^{d+1}} \frac{X^k}{k!} \zeta_l = \sum_{k \in  \N^{d+1}}  \frac{X^k \zeta_l}{S(X^k \zeta_l)}. 
	\end{equs}
Since $ S(X^k \zeta_l  ) = k! $, the base case is tackled. Moving to the induction part, one has, for a tree $\tau= \zeta_l\prod_{i=1}^n\mcI_{a_i}(\tau_i)^{\beta_i}$ with the $ \mcI_{a_i}(\tau_i) $ pairwise distinct,
\begin{equs}
	\Mloc\tau=M_{\text{\tiny{loc}}} \zeta_l\prod_{i=1}^n\mcI_{a_i}(\tau_i)^{\beta_i} =  \sum_{k \in \N^{d+1}} \frac{X^k}{k!} \zeta_l\prod_{i=1}^n\mcI_{a_i}( M_{\text{\tiny{loc}}} \tau_i)^{\beta_i}.
\end{equs}
Using the definition of $S$, we can write
\begin{equs}
	S\left(\zeta_l\prod_{i=1}^n\mcI_{a_i}(\tau_i)^{\beta_i}\right) = \prod_{i=1}^n  S(\tau_i)^{\beta_i}\beta_i !,
\end{equs}
which yields
\begin{equation*}
\frac{M_{\text{\tiny{loc}}} \zeta_l\prod_{i=1}^n\mcI_{a_i}(\tau_i)^{\beta_i}}{S\left(\zeta_l\prod_{i=1}^n\mcI_{a_i}(\tau_i)^{\beta_i}\right)}=\sum_{k \in \N^{d+1}} \frac{X^k}{k!} \zeta_l\prod_{i=1}^n\frac{1}{\beta_i!}\mcI_{a_i}\left(\frac{ M_{\text{\tiny{loc}}} \tau_i}{S(\tau_i)}\right)^{\beta_i}.
\end{equation*}
For each $\tau_i$, we write the induction hypothesis as
\begin{equs}	\frac{M_{\text{\tiny{loc}}} \tau_i}{S(\tau_i)} = \sum_{j} \frac{\tau_{i,j}}{S(\tau_{i,j})}.
\end{equs}
Then, one can observe, using the multinomial formula,
\begin{equs}
 \frac{1}{\beta_i ! } \mcI_{a_i}\Biggl(\sum_j\frac{\tau_{i,j}}{S(\tau_{i,j})}\Biggl)^{\beta_i} = \sum_{\sum_{j}\beta_{i,j}  = \beta_i} \prod_{j} \frac{1}{\beta_{i,j} !}\left(\frac{\CI_{a_i}(\tau_{i,j})}{S(\tau_{i,j})}\right)^{\beta_{i,j}}.
\end{equs}
We can then write successively, using two times a pairwise distinctness argument with the symmetry factors,
\begin{equs}
\Mloc\tau&=\sum_{k \in \N^{d+1}} \frac{X^k}{k!} \zeta_l\prod_{i=1}^n\sum_{\sum_{j}\beta_{i,j}  = \beta_i} \prod_{j} \frac{1}{\beta_{i,j} !}\left(\frac{\CI_{a_i}(\tau_{i,j})}{S(\tau_{i,j})}\right)^{\beta_{i,j}}\\
&=\sum_{k \in \N^{d+1}} \frac{X^k}{k!} \zeta_l\prod_{i=1}^n\sum_{\sum_{j}\beta_{i,j}  = \beta_i} \frac{ \prod_{j}\CI_{a_i}(\tau_{i,j})^{\beta_{i,j}}}{S(\prod_j\mcI_{a_i}(\tau_{i,j})^{\beta_{i,j}})}\\
&=\sum_{k \in \N^{d+1}} \frac{X^k}{k!} \zeta_l\sum_{\forall i,~\sum_{j}\beta_{i,j}  = \beta_i}\prod_{i=1}^n\frac{ \prod_{j}\CI_{a_i}(\tau_{i,j})^{\beta_{i,j}}}{S(\prod_j\mcI_{a_i}(\tau_{i,j})^{\beta_{i,j}})}\\
&=\sum_{k \in \N^{d+1}} \frac{X^k}{k!} \zeta_l\sum_{\forall i,~\sum_{j}\beta_{i,j}  = \beta_i}\frac{ \prod_{i,j}\CI_{a_i}(\tau_{i,j})^{\beta_{i,j}}}{S(\prod_{i,j}\mcI_{a_i}(\tau_{i,j})^{\beta_{i,j}})},
\end{equs}
which concludes the proof.
\end{proof}

We can now turn to the main theorem of this section, whose proof is quite elementary.

\begin{theorem} \label{renormalised_equation}
We have, for any $\eps\in (0,1]$,
\begin{equation}
\sum_{\tau\in T_0^*}\frac{1}{S(\tau)}\Pi_{\eps,0}^R\tau=\sum_{\sigma\in T_0^{X,-,*}}\frac{\ell(\sigma)}{S(\sigma)}\Upsilon[\sigma],
\end{equation}
so that $\phi_{\eps,0}$ solves the SPDE
\begin{equation}
L\phi=g(x,\phi,\dots,\nabla^q\phi)+f(\phi,\dots,\nabla^p\phi)\xi+\sum_{\sigma\in T_0^{X,-,*}\backslash \{\one,\Xi\}}\frac{\ell(\sigma)}{S(\sigma)}\Upsilon[\sigma][\phi].
\end{equation}
\end{theorem}

\begin{proof}
Let $\tau\in T_0^*$. We write, in Sweedler's notation, as well as Lemma \ref{localisation_prop},
\begin{equation*}
\Pi_{\eps,0}^R\tau=\sum_{(\tau)}\sum_i\frac{S(\tau_1)}{S(\tau_1^i)}\ell(\tau_1^i)\Upsilon[\tau_1^i]\Pi_{\eps,0}^{R\times}\tau_2.
\end{equation*}
One can then note that $\Pi_{\eps,0}^{R\times}$ vanishes on every tree excepted $\one$ and $\Xi$, since $G-G_0=0$. As we have imposed that $\mcI_a$ vanishes on trees composed only of grey symbols and $\one$, the case $\tau_2=\Xi$ only lets the possibility $\tau=\Xi$. For $\tau\neq\Xi$, we then get, noting that for $\tau\in \mcT_0^*$, one can write $\Delta_r\tau=\tau\otimes\one+\sum_{(\tau)}\tau_1\otimes\tau_2$, with $\tau_2\neq \one,\Xi$,
\begin{equation*}
\Pi_{\eps,0}^R\tau=\sum_i\frac{S(\tau)}{S(\tau^i)}\ell(\tau^i)\Upsilon[\tau^i]\Pi_{\eps,0}^{R\times}\one.
\end{equation*}
Then,
\begin{equation*}
\sum_{\tau\in T_0^*}\frac{1}{S(\tau)}\Pi_{\eps,0}^R\tau=\sum_{\tau\in T_0^*}\sum_i\frac{\ell(\tau^i)}{S(\tau^i)}\Upsilon[\tau^i]=\sum_{\sigma\in T_0^{X,-,*}}\frac{\ell(\sigma)}{S(\sigma)}\Upsilon[\sigma].
\end{equation*}
A simple rewriting gives the result at the level of the SPDE.
\end{proof}

\begin{remark}
One can check that the counterterm is exactly the same as in regularity structures in \cite[Theorem 2.22]{BCCH}, which was, of course, expected, but is found through different techniques.
\end{remark}

\subsection{BPHZ Renormalisation}\label{bphz}

In this section, we explain how to choose the character $\ell$. For this purpose, we need to localize our evaluation map $\eval^R$. Once this step is done, it allows us to compare the renormalisation in this context with the BPHZ renormalisation in regularity structures developed in \cite{BHZ}.\\
For the rest of this section, we ignore the decoration $\mfv$, as it is of no use here. We adapt our notations accordingly. We also warn that the proofs of this section are somewhat tedious and technical, so that the proofs might be skipped for a first reading in order to grasp the main ideas. We define what we call a local evaluation map, that is, up to the parameter $\mu$, identical to renomalised pre-models in regularity structures. The term \textit{local} refers here to the fact that it depends on a given base point.

\begin{definition}[Local evaluation map]
We define inductively on $\mcT_0$ a linear map, depending as usual on the parameters $\eps$ and $\mu$, as well as a new parameter $x\in[0,1]\times\T^d$ that is the base point of the localization. We first define
\begin{equation}
(\hlocevalx X^k)(y)=(y-x)^k~~\mathrm{and}~~(\hlocevalx\Xi)(y)=\xi_\eps(y).
\end{equation}
Then inductively
\begin{equation}
\hlocevalx(\tau\sigma)=\hlocevalx\tau\times\hlocevalx\sigma,
\end{equation}
and
\begin{equation}
\hlocevalx\mcI_a(\tau)=\D^a(G-G_\mu)*\locevalx\tau,
\end{equation}
as well as the renormalisation procedure
\begin{equation}\label{prepmap}
\locevalx\tau=(\ell\otimes\hlocevalx)\Delta_r.
\end{equation}

Finally, we define the base-point-independent local evaluation map
\begin{equation}
(\loceval\tau)(y)=(\hat\Pi_{\eps,\mu,y}^R\tau)(y).
\end{equation}

\end{definition}

Note in this definition that we use the renormalisation through preparation maps introduced in \cite{BR18}. It is easily seen in equation \eqref{prepmap}, since, in the context of inductive definitions in regularity structures, one defines a linear map $R=(\ell\otimes\mathrm{id})\delta_r$ and applies after the model. Another important remark is that the definition above matches the one of pre-models but not the one of models. In the present case, we do not have any recentering. This step is rather dealt with by the flow on $\mu$.\\
We show an algebraic relation between the localization map $\Mloc$ and the coproduct $\Delta_r$. It is of interest as it is, up to the projection $\pi$, a property of comorphism.

\begin{proposition}\label{Mlocdelta}
It holds on $\mcT_0$ that 
\begin{equation}
(\Mloc\otimes\pi\Mloc)\Delta_r=(\mathrm{id}\otimes\pi)\Delta_r\Mloc,
\end{equation}
where $\pi$ is a projection defined on $T_0^X$ by
\begin{equation}
\pi \zeta_l X^k\prod_{i=1}^n\mcI_{a_i}(\tau_i)=\one_{k=0}\zeta_l\prod_{i=1}^n\mcI_{a_i}(\tau_i),
\end{equation}
and is extended by linearity to $\mcT_0^X$. It acts by removing the trees that have polynomials at the root.
\end{proposition}

\begin{proof} We proceed by induction. We start with the two base cases. We have first
\begin{equs}
(\mathrm{id}\otimes\pi)\Delta_r\Mloc\one&=(\mathrm{id}\otimes\pi)\Delta_r\sum_{k\in\N^{d+1}}\frac{X^k}{k!}=(\mathrm{id}\otimes\pi)\sum_k\frac{1}{k!}\sum_j\binom{k}{j}X^j\otimes X^{k-j}\\
&=(\mathrm{id}\otimes\pi)\left(\sum_{k\in\N^{d+1}}\frac{X^k}{k!}\otimes\sum_{j\in\N^{d+1}}\frac{X^j}{j!}\right)=\sum_{k\in\N^{d+1}}\frac{X^k}{k!}\otimes\one\\
&=(\Mloc\otimes\pi\Mloc)\one\otimes\one=(\Mloc\otimes\pi\Mloc)\Delta_r\one.
\end{equs}
The other base case on $\Xi$ is treated identically. We have then, for a tree of the form $\mcI_a(\tau)$,
\begin{equs}
&(\mathrm{id}\otimes\pi)\Delta_r\Mloc\mcI_a(\tau)=(\mathrm{id}\otimes\pi)\Delta_r\sum_{k\in\N^{d+1}}\frac{X^k}{k!}\mcI_a(\Mloc\tau)\\
&=(\mathrm{id}\otimes\pi)\left(\sum_{k\in\N^{d+1}}\frac{1}{k!}\sum_{j\in\N^{d+1}}\binom{k}{j}X^j\otimes X^{k-j}\right)\times\Delta_r\mcI_a(\Mloc\tau)\\
&=(\mathrm{id}\otimes\pi)\left(\sum_{k\in\N^{d+1}}\frac{X^k}{k!}\otimes\sum_{j\in\N^{d+1}}\frac{X^j}{j!}\right)\\
&~~~~\times\left((\mcI_a\otimes\mathrm{id})\Delta_r\Mloc\tau+\sum_{i\in\N^{d+1}}\frac{1}{i!}{\color{gray}\mcI}_a(X^i)\otimes\mcI_{a+i}(\Mloc\tau)\right)\\
&=\left(\sum_{k\in\N^{d+1}}\frac{X^k}{k!}\mcI_a\otimes\mathrm{id}\right)(\mathrm{id}\otimes\pi)\Delta_r\Mloc\tau\\
&~~~~+(\Mloc\otimes\mathrm{id})\sum_{i\in\N^{d+1}}\frac{1}{i!}{\color{gray}\mcI}_a(X^i)\otimes\mcI_{a+i}(\Mloc\tau)\\
&=\left(\sum_{k\in\N^{d+1}}\frac{X^k}{k!}\mcI_a\otimes\mathrm{id}\right)(\Mloc\otimes\pi\Mloc)\Delta_r\tau\\
&~~~~+\left(\Mloc\otimes\pi\sum_{j\in\N^{d+1}}\frac{X^j}{j!}\right)\sum_{i\in\N^{d+1}}\frac{1}{i!}{\color{gray}\mcI}_a(X^i)\otimes\mcI_{a+i}(\Mloc\tau)\\
&=(\Mloc\mcI_a\otimes\pi\Mloc)\Delta_r\tau+(\Mloc\otimes\pi\Mloc)\sum_{i\in\N^{d+1}}\frac{1}{i!}{\color{gray}\mcI}_a(X^i)\otimes\mcI_{a+i}(\tau)\\
&=(\Mloc\otimes\pi\Mloc)\left((\mcI_a\otimes\mathrm{id})\Delta_r\tau+\sum_{i\in\N^{d+1}}\frac{1}{i!}{\color{gray}\mcI}_a(X^i)\otimes\mcI_{a+i}(\tau)\right)\\
&=(\Mloc\otimes\pi\Mloc)\Delta_r\mcI_a(\tau).
\end{equs}
For the product, we use the fact that $\pi$, $\Delta_r$, and $\Mloc$ are all multiplicative maps.
\begin{equs}
(\mathrm{id}\otimes\pi)\Delta_r\Mloc(\sigma\tau)&=(\mathrm{id}\otimes\pi)\Delta_r(\Mloc\sigma\times\Mloc\tau)\\
&=(\mathrm{id}\otimes\pi)(\Delta_r\Mloc\sigma)\times(\Delta_r\Mloc\tau)\\
&=(\mathrm{id}\otimes\pi)(\Delta_r\Mloc\sigma)\times(\mathrm{id}\otimes\pi)(\Delta_r\Mloc\tau)\\
&=(\Mloc\otimes\pi\Mloc)\Delta_r\sigma\times(\Mloc\otimes\pi\Mloc)\Delta_r\tau\\
&=(\Mloc\otimes\pi\Mloc)(\Delta_r\sigma\times\Delta_r\tau)\\
&=(\Mloc\otimes\pi\Mloc)\Delta_r(\sigma\tau).
\end{equs}
\end{proof}

\begin{definition}
We define on $\mcT_0^X$ an elementary differential acting on all the nodes but not the one at the root. We set for elementary trees $\hat\Upsilon[X^k\zeta_l]=1$ and 
\begin{equation}
\hat\Upsilon\left[ \zeta_l X^k\prod_{i=1}^n\mcI_{a_i}(\tau_i)\right]=\prod_{i=1}^n\Upsilon[\tau_i].
\end{equation}
\end{definition}

From this definition, we can prove the following lemma. We advise the reader to check on Example \ref{exdeltar} that it holds on simple cases, for intuition.

\begin{lemma}\label{lemmaupsilondelta}
Let $\tau\in T_0^X$, writing in Sweedler's notation $\Delta_r\tau=\sum_{(\tau)}\tau_1\otimes\tau_2$, we have
\begin{equation}
\Upsilon[\tau_1]\hat\Upsilon[\tau_2]=\Upsilon[\tau].
\end{equation}
\end{lemma}

\begin{proof}
The statement is immediate on the elementary trees $\one$ and $\Xi$. We start by showing the result for a tree in $T_0$ of the form $\tau=\prod_{i=1}^n\mcI_{a_i}(\tau^i)$. We have,
\begin{equs}
\Delta_r\tau=\sum_{I\subset\{1,\dots, n\}}\sum_{\substack{k_i\in\N^{d+1}\\i\in I}}\sum_{\substack{(\tau^j)\\j\in I^c}}\prod_{i\in I}\frac{1}{k_i!}\CIg_{a_i}(X^{k_i})\prod_{j\in I^c}\mcI_{a_j}(\tau_1^j)\otimes\prod_{i\in I}\mcI_{a_i+k_i}(\tau^i)\prod_{j\in I^c}\tau_2^j
\end{equs}
Then
\begin{equs}
\Upsilon\left[\prod_{i\in I}\CIg_{a_i}(X^{k_i})\prod_{j\in I^c}\mcI_{a_j}(\tau_1^j)\right]=D_{a_1}\dots D_{a_n}\Upsilon[\one]\prod_{j\in I^c}\Upsilon[\tau_1^j],
\end{equs}
and
\begin{equs}
\hat\Upsilon\left[\prod_{i\in I}\mcI_{a_i+k_i}(\tau^i)\prod_{j\in I^c}\tau_2^j\right]=\prod_{i\in I}\Upsilon[\tau^i]\prod_{j\in I^c}\hat\Upsilon[\tau_2^j],
\end{equs}
so that the product of these two quantities is, using the induction,
\begin{align*}
&D_{a_1}\dots D_{a_n}\Upsilon[\one]\prod_{j\in I^c}\Upsilon[\tau_1^j]\hat\Upsilon[\tau_2^j]\prod_{i\in I}\Upsilon[\tau^i] \\ & =D_{a_1}\dots D_{a_n}\Upsilon[\one]\prod_{j\in I^c}\Upsilon[\tau^j]\prod_{i\in I}\Upsilon[\tau^i]\\
&=D_{a_1}\dots D_{a_n}\Upsilon[\one]\prod_{i=1}^n\Upsilon[\tau^i]=\Upsilon[\tau].
\end{align*}
The proof in the case where there is a symbol $\Xi$ at the root is the same but slightly longer to write. The full statement on $T_0^X$ follows in the light of \cite[Prop 2.2]{BB21b}. Since $\hat\Upsilon$ completely ignores the root of the input tree, the only difficulty that could appear because of the polynomial decoration is avoided. Thus, it suffices to follow the lines of the present proof. We chose not to do it for the sake of readability, but the essential arguments remain. 
\end{proof}
We turn to one of the main results of this section. We show how the evaluation map defined in Definition \ref{eval} can be localized, \textit{i.e.} the elementary differentials that were stuck in the convolutions are all taken out. By doing this, we exhibit a link with what is done by Hairer in regularity structures.
\begin{theorem}\label{localization}
For every $x\in[0,1]\times\T^d$, $\tau\in\mcT_0$ and $\phi$,
\begin{equation}\label{xneqy}
(\eval^R\tau)[\phi](y)=\Bigl(\Upsilon[\cdot][\phi](x)(\locevalx\cdot)(y)\Bigl)\Mloc\tau.
\end{equation}
The product between $\Upsilon$ and $\locevalx$ in the right hand side is defined by a simple product on $T_0^X$ and extended by linearity to $\mcT_0^X$. In particular,
\begin{equation}\label{xeqy}
(\eval^R\tau)[\phi](y)=\Bigl(\Upsilon[\cdot][\phi](y)(\loceval\cdot)(y)\Bigl)\Mloc\tau.
\end{equation}
\end{theorem}

\begin{remark}
Only the second equation \eqref{xeqy} will be useful to prove Theorem \ref{bphztheorem}, but the first version, where $x\neq y$, is necessary to prove it by induction.
\end{remark}

\begin{proof} 
We start with the two base cases. For $\tau=\one$, we first compute for $k\in\N^{d+1}$,
\begin{equs}
(\locevalx X^k)(y)=\sum_{j\in\N^{d+1}}\binom{k}{j}\ell(X^j)(\hloceval X^{k-j})(y)=(y-x)^k,
\end{equs}
so that 
\begin{align*}
&(\Upsilon[\cdot][\phi](x)\locevalx(y))(\Mloc\one)=\sum_{k\in\N^{d+1}}\frac{1}{k!}\D^k g(\phi(x),\dots,\nabla^q\phi(x))\locevalx X^k\\
&=\sum_{k\in\N^{d+1}}\frac{1}{k!}\D^k g(\phi(x),\dots,\nabla^q\phi(x))(y-x)^k=g(\phi(y),\dots,\nabla^q\phi(y))\\
&=(\eval^R\one)[\phi](y).
\end{align*}
We recall that the symbol $\D^k$ is on the function $x\mapsto g(\phi(x),\dots,\nabla^q(x))$, accordingly to the same notation on elementary differentials. The case $\tau=\Xi$ is dealt with in a similar manner.\\

We now show that, in order to be able to perform the induction,
\begin{equation*}
(\heval\tau)[\phi](y)=\Bigl(\hat\Upsilon[\cdot][\phi](x)(\hlocevalx\cdot)(y)\Bigl)\pi\Mloc\tau.
\end{equation*}

We have, using the notation $\Mloc\tau=\sum_i a_i\tau^i$
\begin{equs}
(\heval\mcI_a(\tau))[\phi](y)&=\int_{[0,1]\times\T^d}\D^a(G-G_\mu)(z-y)(\eval^R\tau)[\phi](z)dz\\
&=\int_{[0,1]\times\T^d}\D^a(G-G_\mu)(z-y)(\Upsilon[\cdot][\phi](x)\locevalx(z))(\Mloc\tau)dz\\
&=\sum_i a_i \Upsilon[\tau^i][\phi](x)\int_{[0,1]\times\T^d}\D^a(G-G_\mu)(z-y)(\locevalx\tau^i)(z)dz\\
&=\sum_i a_i \Upsilon[\tau^i][\phi](x)(\hlocevalx\mcI_a(\tau^i))(y)\\
&=\Bigl(\hat\Upsilon[\cdot][\phi](x)(\hlocevalx\cdot)(y)\Bigl)\pi\Mloc\mcI_a(\tau).
\end{equs}
It is immediate to show this same property on a product of the form $\sigma\tau$ since all the maps involved are multiplicative.

Before continuing, we need further definitions. Given a tree $\tau\in T_0^X$, we recall that the nodes set is denoted by $N_\tau$. We will make the confusion to also use this notation to denote the cardinality of this set. We define, for a $N_\tau$-uple $k_1,\dots,k_{N_\tau}$,
\begin{equs}
\uparrow_{N_\tau}^{k_1,\dots, k_{N_\tau}}\tau=\prod_{v\in N_\tau}\uparrow_v^{k_v}\tau,
\end{equs}
where $\uparrow_v^{k_v}$ increases the polynomial decoration of $k_v$ at the node $v$. We also define
\begin{equs}
\uparrow^k\tau=\sum_{\sum_{v\in N_\tau}k_v=k}\uparrow_{N_\tau}^{k_1,\dots, k_{N_\tau}}\tau.
\end{equs}
With these notation in hand, we can perform further computations. We will voluntarily forget to denote under the sums that all the parameters run over $\N^{d+1}$ for compactness of the equations. We take a tree $\tau\in T_0$. It is easy to see that
\begin{equs}
\Mloc\tau=\sum_{k\in\N^{d+1}}\frac{1}{k!}\uparrow^k\tau=\sum_{k_1,\dots,k_{N_{\tau}}}\frac{1}{k_1!\dots k_{N_\tau}!}\uparrow_{N_\tau}^{k_1,\dots, k_{N_\tau}}\tau.
\end{equs}
Then, one can note that, writing in Sweedler's notation $\Delta_r\tau=\sum_{(\tau)}\tau_1\otimes\tau_2$, we have, assuming that we sort the $k_1,\dots,k_{N_\tau}$'s as $k_1,\dots,k_{N_{\tau_1}}$ for $\tau_1$ and $l_2,\dots,l_{N_{\tau_2}}$ for $\tau_2$,
\begin{equs}
\Delta_r\Mloc\tau&=\sum_{\substack{k^1_1,k^2_1,\dots,k_{N_{\tau_1}}^1,k_{N_{\tau_1}}^2\\ l_1,\dots,l_{N_{\tau_2}}}}\frac{1}{k^1_1!k^2_1!\dots k_{N_{\tau_1}}^1!k_{N_{\tau_1}}^2! l_1!\dots l_{N_{\tau_2}}!}\\
&~~~~\times\sum_{(\tau)}\uparrow_{N_{\tau_1}}^{k_1^1,\dots, k_{N_\tau}^1}\tau_1\otimes X^{k_1^2+\dots +k_{N_{\tau_1}}^2}\uparrow_{N_{\tau_2}\backslash\text{\tiny{root}}}^{l_2,\dots, l_{N_\tau}}\tau_2.
\end{equs}
Note that we have started the labelling of the parameters $l$ at $2$ as we do not touch the root. We further have
\begin{equs}
(\mathrm{id}\otimes\pi)\Delta_r\Mloc\tau=\sum_{k_1,\dots,k_{N_{\tau}}}\frac{1}{k_1!\dots k_{N_\tau}!}\sum_{(\tau)}\uparrow_{N_{\tau_1}}^{k_1,\dots, k_{N_{\tau_1}}}\tau_1\otimes\uparrow_{N_{\tau_2}\backslash\text{\tiny{root}}}^{l_2,\dots, l_{N_{\tau_2}}}\tau_2
\end{equs}
It gives
\begin{equs}
&\Bigl(\ell\Upsilon[\cdot][\phi](y)\otimes\hat\Upsilon[\cdot][\phi](x)\hlocevalx\pi\Bigl)\Delta_r\Mloc\tau\\
&=\sum_{k_1,\dots,k_{N_{\tau}}}\frac{1}{k_1!\dots k_{N_{\tau_1}}!}\sum_{(\tau)}\ell\Bigl(\uparrow_{N_{\tau_1}}^{k_1,\dots, k_{N_{\tau_1}}}\tau_1\Bigl)\Upsilon\Bigl[\uparrow_{N_{\tau_1}}^{k_1,\dots, k_{N_{\tau_1}}}\tau_1\Bigl][\phi](y)\\
&~~~~~~~~\times\hat\Upsilon\Bigl[\uparrow_{N_{\tau_2}\backslash\text{\tiny{root}}}^{l_2,\dots, l_{N_{\tau_2}}}\tau_2\Bigl][\phi](x)\Bigl(\hlocevalx\uparrow_{N_{\tau_2}\backslash\text{\tiny{root}}}^{l_2,\dots, l_{N_{\tau_2}}}\tau_2\Bigl)(y)
\end{equs}
We apply a Taylor expansion on the term $\Upsilon\Bigl[\uparrow_{N_{\tau_1}}^{k_1,\dots, k_{N_\tau}}\tau_1\Bigl][\phi](y)$ to localize it in $x$. It is equal to
\begin{equs}
\sum_{m_1,\dots,m_{N_{\tau_1}}}\frac{1}{m_1!\dots m_{N_{\tau_1}}!}\Upsilon\Bigl[\uparrow_{N_{\tau_1}}^{k_1+m_1,\dots, k_{N_\tau}+m_{N_{\tau_1}}}\tau_1\Bigl][\phi](x)(y-x)^{m_1+\dots+m_{N_{\tau_1}}},
\end{equs}
Where we have used that for any $\sigma\in T_0^X$, $\D^k\Upsilon[\sigma]=\Upsilon[\uparrow^k\sigma]$. For a proof of this statement, see \cite[Prop 2.2]{BB21b}. Combining this with Lemma \ref{lemmaupsilondelta}, and with a re-parametrization, we get that
\begin{equs}
\Bigl(\ell\Upsilon[\cdot][\phi](y)\otimes\hat\Upsilon[\cdot][\phi](x)\hlocevalx\pi\Bigl)\Delta_r\Mloc\tau=\Bigl(\Upsilon[\cdot][\phi](x)\bigl((\ell\otimes\hlocevalx)\Delta_r\bigl)\Bigl)\Mloc\tau.
\end{equs}

In a nutshell, we recover with the Taylor expansion what $\pi$ removed. To conclude the proof, we combine all the elements proved above to get the following sequence of equalities,

\begin{equs}
(\eval^R\tau)[\phi](y)&=\Bigl(\ell\Upsilon[\cdot][\phi](y)\otimes(\heval\cdot)(y)\Bigl)(\Mloc\otimes\mathrm{id})\Delta_r\tau\\
\mathrm{(Induction)}&=\Bigl(\ell\Upsilon[\cdot][\phi](y)\otimes\hat\Upsilon[\cdot][\phi](x)(\hlocevalx\cdot)(y)\Bigl)(\Mloc\otimes\Mloc\pi)\Delta_r\tau\\
\mathrm{(Prop~\ref{Mlocdelta})}&=\Bigl(\ell\Upsilon[\cdot][\phi](y)\otimes\hat\Upsilon[\cdot][\phi](x)(\hlocevalx\pi\cdot)(y)\Bigl)\Delta_r\Mloc\tau\\
\mathrm{(Localization)}&=\Bigl(\Upsilon[\cdot][\phi](x)\bigl((\ell\otimes(\hlocevalx\cdot)(y))\Delta_r\bigl)\Bigl)\Mloc\tau\\
&=\Bigl(\Upsilon[\cdot][\phi](x)(\locevalx\cdot)(y)\Bigl)\Mloc\tau.
\end{equs}
\end{proof}

\begin{remark}
With this theorem, we can compare our ansatz with the one in regularity structures. We recall that, in this context, the force term is written at the level of modelled distributions as
\begin{equation}
\sum_{\deg(\tau)<\gamma}\frac{\Upsilon[\tau][u](x)}{S(\tau)}\Pi_{\eps,x}^R\tau,
\end{equation}
where the sum runs over trees with polynomial decorations, truncated at a given level $\gamma>0$. $\Pi_{\eps,x}^R$ is a renormalised model at a base point $x$. $u$ is a modelled distribution, representing the lift of the solution of the equation. With Theorem \ref{localization}, Lemma \ref{localisation_prop}, and the Ansatz \eqref{ansatz}, it is easy to see that, we get almost the same expression as in regularity structures, where recentering around a base point has been replaced with the use of the kernels $G-G_\mu$. Here the terms that have at least one edge vanish when $\mu=0$ where in regularity structures the model vanishes on trees with at least one positive subtree when evaluated at the base point. The main difference is that, in the flow approach, after localizing, one needs to keep full series to have non-local terms. \textit{A contrario}, in regularity structures, this localization step and the truncation that comes alongside is done at the very beginning of the procedure. The price to pay then is that one has to reconstruct these truncated expansions with the help of the reconstruction theorem.
\end{remark}

We can now state the subsequent result, which concludes the understanding of the algebraic renormalisation in the method.
\begin{theorem}\label{bphztheorem}
The renormalisation character used in the flow approach is the BPHZ character defined in \cite{BHZ} for regularity structures.
\end{theorem}

\begin{proof}
We recall from \cite[Theorem 6.18]{BHZ}, that there is a unique choice of $\ell$, called the BPHZ renormalisation and denoted by $\ell_{ \text{\tiny{BPHZ}}}$, such that for any tree $\tau\in T_0^X$ of negative degree,
\begin{equation}\label{zeromeanbphz}
\E\left[(\hat\Pi_{\eps,1}^R\tau)(0)\right]=0.
\end{equation}
where the map $R$ is given by
\begin{equation*}
	R = \left( \ell_{ \text{\tiny{BPHZ}}} \otimes \id \right) \Delta_r.
\end{equation*}
Indeed, we recall that $G-G_1$ corresponds to the compactly supported truncation of the heat kernel, generally denoted $K$ in the regularity structures literature, so that we recover the renormalised pre-model of regularity structures. It turns out that it is the choice made in the previous works on the approach. Concretely, this localization step comes up at the initialization of the flow. One writes, for a tree $\tau\in T_0$, 
\begin{equation}
\begin{split}
(\eval^R\tau)[\phi](y)&=\Bigl(\Upsilon[\cdot][\phi](y)(\loceval\cdot)(y)\Bigl)\Mloc\tau\\
&=\Bigl(\Upsilon[\cdot][\phi](y)(\loceval\cdot)(y)\Bigl)\pi_-\Mloc\tau+R_{\eps,\mu,\tau}[\phi](y),
\end{split}
\end{equation}
where $\pi_-$ is the projection from $\mcT_0^X$ to trees of negative degree, and $R_{\eps,\mu,\tau}$ is a remainder, of better regularity than the previous terms. At this step, a choice of counterterms has to be done for the local stochastic terms of the form $\loceval\sigma$, where $\sigma$ appears in $\pi_-\Mloc\tau$. The choice made in \cite{Duc21, Duc22, CF24a} corresponds exactly to equation \eqref{zeromeanbphz} (see for example \cite[Equation 4.8]{CF24a} in the context of multi-indices), thus concluding the proof.
\end{proof}

\begin{remark}
For many equations of interest, such as the $\phi^4_3$ and KPZ equations, it holds that $\pi_-\Mloc\tau=\tau$, so that trees containing polynomials do not appear in the counterterms. See for example the lecture notes \cite{Duc23} on an elliptic $\phi^4$ model.
\end{remark}

\end{document}